\documentclass[a4paper,reqno,11pt]{amsart}
\usepackage{amssymb,verbatim}
\usepackage{a4wide}
\usepackage{amssymb,verbatim}
\usepackage[T1]{fontenc}
\usepackage{mathtools}
\usepackage{amsmath, amscd, a4, amssymb, latexsym, amsthm, xspace, color}
\usepackage{ifpdf, hyperref}
\usepackage{tikz,ifthen}
\usetikzlibrary{arrows,decorations.markings}
\usetikzlibrary[calc]
\usepackage{multirow}
\usepackage{breqn}
\usepackage{paralist}
\usepackage{graphicx}
\usepackage{subfigure}
\usepackage[all]{xy}
\usepackage{caption} 

 \usepackage[normalem]{ulem}
 \newcommand\redsout{\bgroup\markoverwith{\textcolor{red}{\rule[0.5ex]{2pt}{0.8pt}}}\ULon}

\newcommand{\cita}{}

\newcommand{\R}{\mathbb{R}}

\newcommand{\Z}{\mathbb{Z}}
\newcommand{\N}{\mathbb{N}}
\newcommand\D{d\hspace{-0.5pt}}
\newcommand\pA{\phi}

\DeclarePairedDelimiter\ceil{\lceil}{\rceil}
\DeclarePairedDelimiter\floor{\lfloor}{\rfloor}

\theoremstyle{plain}
\newtheorem{Theorem}{Theorem}[section]

\newtheorem{Proposition}[Theorem]{Proposition}
\newtheorem{Lemma}[Theorem]{Lemma}
\newtheorem{Remark}{Remark}[section]

\newtheorem{Definition}{Definition}[section]

\newtheorem*{NoNumberTheorem}{Theorem}
\newtheorem*{ThmA}{Theorem A}
\newtheorem*{ThmB}{Theorem B}
\newtheorem*{ThmC}{Theorem C}

\newtheorem{Convention}{Convention}
\newtheorem*{Notation}{Notation}

\tikzset{fleche/.style args={#1:#2}{ postaction = decorate,decoration={name=markings,mark=at position #1 with {\arrow[#2,scale=2]{>}}}}}

\begin{document}{\large}
\bibliographystyle{halpha}
\title[Lengths spectrum of hyperelliptic components]
{Geodesic length spectrum of hyperelliptic connected components
}
\dedicatory{To the memory of Jean-Christophe Yoccoz}

\date{\today}
\author{Corentin Boissy and Erwan Lanneau}

\address{
Institut Fourier, Universit\'e Grenoble-Alpes, BP 74, 38402 Saint-Martin-d'H\`eres, France
}
\email{erwan.lanneau@univ-grenoble-alpes.fr}

\address{ Institut de Math\'ematiques de Toulouse,
Universit\'e Paul Sabatier F-31062 Toulouse, France}
\email{corentin.boissy@math.univ-toulouse.fr}

\begin{abstract}
We propose a general framework for studying pseudo-Anosov homeomorphisms on 
translation surfaces. This new approach, among other consequences, allows us to 
compute the systole of the Teichm\"uller geodesic flow restricted to the hyperelliptic connected components, settling a question of~\cite{Farb}.
We stress that all proofs and computations are performed without the help of a computer.
As a byproduct, our methods give a way to describe the bottom of the lengths spectrum of the
hyperelliptic components. 
\end{abstract}

\maketitle

\section{Introduction}

Every affine pseudo-Anosov map $\pA$ on a half-translation surface $S$ has an expansion factor $\lambda(\pA)\in \R$
recording the exponential growth rate of the lengths of the curves under iteration of $\pA$. The set of the logarithms of all 
expansion factors (when fixing the genus $g$ of the surfaces) is a discrete subset of $\R$: this is the lengths spectrum of 
the moduli space $\mathcal M_g$.

One can also consider other lengths spectrum $\mathrm{spec}(H),\mathrm{spec}(\mathcal C) \subset \mathrm{spec}(\mathrm{Mod}_g)$
for various subgroups $H < \mathrm{Mod}(g)$ or for various connected components of strata $\mathcal C$ of the moduli spaces of 
quadratic differentials. 

These objects have been the subject of many investigations recently (we refer to the recent work of
McMullen, Farb--Leininger--Margalit, Leininger, Agol--Leininger--Margalit).

Describing $\mathrm{spec}(\mathcal C)$ is a difficult problem and, so far, only bounds on the systole $L(\mathrm{spec}(\mathcal C))$
are known for various cases. 

In this paper we present a general framework for studying pseudo-Anosov homeomorphisms.
As a consequence, we provide a complete description of $L(\mathrm{spec}(\mathcal C))$ when 
$\mathcal C=C^{\mathrm{hyp}}$ is a {\em hyperelliptic connected component} of the moduli space of Abelian differentials. This is the very first instance of an explicit computation of the systole for an infinite family of strata.
\par
\begin{ThmA}
The minimum value of the expansion factor $\lambda(\pA)$ over all affine pseudo-Anosov maps $\pA$ 
on a translation surface $S\in \mathcal C^{\mathrm{hyp}}$ is given by the largest root of the polynomial
$$
\begin{array}{lll}
X^{2g+1} - 2X^{2g-1} - 2X^2 + 1 & \textrm{if } S\in \mathcal H^{hyp}(2g-2),\\
X^{2g+2} -  2 X^{2g}  - 2 X^{g+1}  - 2  X^2 + 1 & \textrm{if } S\in  \mathcal H^{hyp}(g-1,g-1),\ g \textrm{ is even}, \\
X^{2g+2} - 2 X^{2g} -  4 X^{g+2} + 4  X^{g} + 2 X^2  - 1 & \textrm{if } S\in    \mathcal H^{hyp}(g-1,g-1),\ g \textrm{ is odd}.
\end{array}
$$
Moreover the conjugacy mapping class realizing the minimum is unique.
\end{ThmA}
\par
This theorem settles a question of Farb~\cite[Problem 7.5]{Farb} for an infinite family of connected component of strata of the moduli space of Abelian differentials.

With more efforts, our method gives a way to also compute the other elements of the
spectrum $\mathrm{Spec}(\mathcal C^{\mathrm{hyp}})$ where $\mathcal C^{\mathrm{hyp}}$
ranges over all hyperelliptic connected components, for any genus. As an instance
we will also prove (see Theorem~\ref{thm:n:even}):
\begin{ThmB}
For any $g$ even, $g\not \equiv 2 \mod 3$, $g\geq 9$, the second least dilatation of an affine pseudo-Anosov maps $\pA$ 
on a translation surface $S\in \mathcal H^{hyp}(2g-2)$ is given by the largest root of the polynomial
$$
X^{2g+1} - 2X^{2g-1} - 2X^{\ceil{4g/3}}- 2X^{\floor{2g/3}+1} - 2X^2 + 1.
$$
Moreover the conjugacy mapping class realizing the minimum is unique.
\end{ThmB}
It is proved in~\cite{BL12} that $L(\mathrm{Spec}(\mathcal C^{\mathrm{hyp}})) \in \left] \sqrt{2},\sqrt{2}+\frac1{2^{g-1}}\right[$.\\
For $\mathcal H^{hyp}(2)$ and $\mathcal H^{hyp}(1,1)$, the corresponding systoles are the largest root of the polynomials 
$X^{5} - 2X^{3} - 2X^2 + 1=(X+1)(X^4-X^3-X^2-X+1)$ and $X^{6} - 2X^{4}- 2X^{3} - 2X^2 + 1=(X+1)^2(X^4-2X^3+X^2-2X+1)$ 
respectively, recovering previous result of~\cite{Lanneau:Thiffeault}. \medskip
\par
For small values of $g$, as an illustration of our construction, we are able to produce a complete description of the bottom of the spectrum of $\mathcal C^{\mathrm{hyp}}$.
\begin{ThmC}
\label{theoremC}
For $g\leq 10$, the lengths $l$ of the closed Teichm\"uller geodesics on $\mathcal{H}^{hyp}(2g-2)$ satisfying $l<2$ are recorded in the table below. For genus between 4 and 10, we only indicate the number of geodesic lengths.
$$
\begin{array}{|c|l|}
\hline
g & \textrm{lengths of closed Teichm\"uller geodesics on } \mathcal{H}^{hyp}(2g-2) \textrm{ less that 2} \\
\hline
2 & \textrm{Perron root of } X^5-2X^3-2X^2+1 \sim 1.72208380573904\\
\hline
3 & \textrm{Perron root of } X^7-2X^5-2X^2+1 \sim 1.55603019132268 \\
 &\textrm{Perron root of }  X^7-2X^5-X^4-X^3-2X^2+1 \sim 1.78164359860800 \\
 & \textrm{Perron root of }  X^7-3X^5-3X^2+1 \sim 1.85118903363607 \\
 &\textrm{Perron root of }  X^7-2X^5-2X^4-2X^3-2X^2+1 \sim 1.94685626827188\\
 \hline
4 & \textrm{11 geodesic lengths}\\
5 & \textrm{22 geodesic lengths}\\
6 & \textrm{79 geodesic lengths}\\
7 & \textrm{142 geodesic lengths}\\
8 & \textrm{452 geodesic lengths}\\
9 & \textrm{1688 geodesic lengths}\\
10 & \textrm{4887 geodesic lengths}\\
\hline
\end{array}
$$
\end{ThmC}
Obviously our techniques also provide a way to investigate the bottom of the spectrum for the stratum  $\mathcal{H}^{hyp}(g-1,g-1)$ for various $g$ and various bound (not necessarily $2$). 
This will appear in the forthcoming paper.

\subsection*{Acknowledgments} The authors thank Artur Avila and Jean-Christophe Yoccoz for helpful conversations and for asking the question on the systoles. This article would not be possible without the seminal work of Jean-Christophe Yoccoz, whose mathematics is still a source of inspiration for both authors.

This collaboration began during the program sage days at CIRM in March 2011, and continued 
during the program ``Flat Surfaces and Dynamics on Moduli Space" at Oberwolfach in March 2014. Both authors attended these programs and are grateful to the organizers, the CIRM, and MFO. This work was partially supported by the ANR Project GeoDyM and the Labex Persyval.

\section{Overall strategy}
\label{sec:strategy}

Our strategy is to convert the computation of mapping classes and their expanding factors
into a finite combinatorial problem. This is classical in pseudo-Anosov theory: the Rauzy--Veech theory
and the train track theory are now well established. However a major difficulty comes from the 
very complicated underlying combinatorics.

\subsection{Hyperelliptic connected components}

In the sequel, for any integer $n\geq 4$, we will consider the hyperelliptic Rauzy diagram $\mathcal{D}_n$ of size $2^{n-1}-1$ containing the permutation
$$
\pi_n= \begin{pmatrix}
1&2&\dots &n \\ n&n-1&\dots &1
\end{pmatrix}
$$
and by $\mathcal C_n^{\mathrm{hyp}}$ the associated connected component.
If $n=2g$ is even then $\mathcal C_n^{\mathrm{hyp}}=\mathcal{H}^{hyp}(2g-2)$ and if $n=2g+1$ is odd then
$\mathcal C_n^{\mathrm{hyp}}=\mathcal{H}^{hyp}(g-1,g-1)$. The precise description of these diagrams was given by 
Rauzy~\cite{Rauzy}.

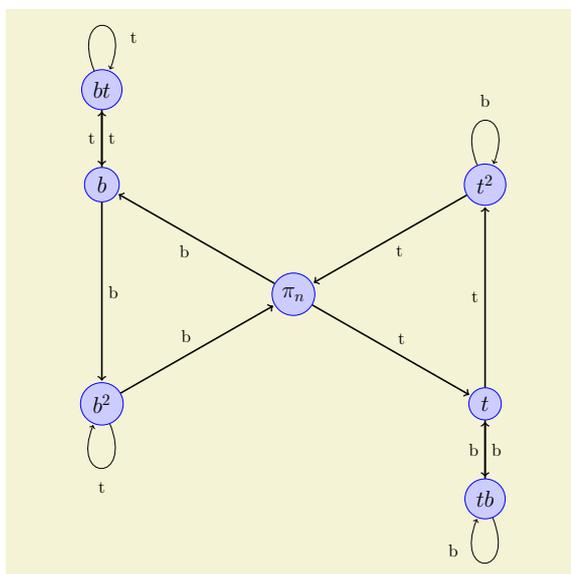
\begin{figure}[htbp]
\scalebox{0.6}{
\begin{tikzpicture}[scale=0.7,shorten >=1pt, auto, node distance=3cm, 
   edge_style/.style={line width=1pt,draw=black,->}]
   
   \fill[fill=yellow!80!black!20,even   odd  rule]   (-13,-9)  rectangle (5,9);

   \foreach [count=\i] \x/\y/\t in {-4/0/$\pi_n$,2/-3.46/$t$,2/3.46/$t^2$,-10/3.46/$b$,-10/-3.46/$b^2$,2/-6.46/$tb$,-10/6.46/$bt$}
     \node [circle,draw=blue,fill=blue!20!,font=\sffamily\Large\bfseries]
        (v\i) at (\x,\y) {\t};

   \foreach \i/\j/\t in {1/2/t, 2/3/t,3/1/t,1/4/b,4/5/b,5/1/b,2/6/b,6/2/b,4/7/t,7/4/t}
    \draw [edge_style]  (v\i) edge node {\t} (v\j); 
    
   \foreach \i/\j/\t in {3/3/b,7/7/t} \draw [->]  (v\i) ..  controls  +(-1,2.5) and +(1,2.5) .. (v\j); 
    \draw  (2,6.1) node  {b}; \draw  (1,-8.1) node  {b};

   \foreach \i/\j/\t in {5/5/t,6/6/b} \draw [->]  (v\i) ..  controls  +(1,-2.5) and +(-1,-2.5) .. (v\j); 
    \draw  (-10,-6.1) node  {t}; \draw  (-9,8.1) node  {t};
\end{tikzpicture}
}
\caption{Rauzy diagram for $n=4$}
\end{figure}
Equipped with this and the Kontsevich-Zorich~\cite{KZ} classification, 
one can convert our Main Theorem into Theorem~\ref{thm:main} below:
\par
\begin{Theorem}
\label{thm:main}
For any $n\geq 4$, the minimum value of the expansion factor $\lambda(\pA)$ over all affine pseudo-Anosov map $\pA$ 
on a translation surface $S\in \mathcal C_n^{\mathrm{hyp}}$ is given by the largest root of the polynomial
$$
\begin{array}{ll}
X^{n+1} - 2X^{n-1} - 2X^2 + 1 & \textrm{if n is even},\\
X^{n+1} -  2 X^{n-1}  - 2 X^{(n-1)/2+1}  - 2  X^2 + 1 & \textrm{if } n\equiv1 \mod 4, \\
X^{n+1} - 2 X^{n-1} -  4 X^{(n-1)/2+2} + 4  X^{(n-1)/2} + 2 X^2  - 1 & \textrm{if } n\equiv3 \mod 4. \\
\end{array}
$$
Moreover the conjugacy mapping class realizing the minimum is unique.
\end{Theorem}
\par
The same applies to the theorem for the second least expanding factor~:
\begin{Theorem}\label{thm:main:second}
For any $n\geq 18$, if $n\not \equiv 4 \mod 6$ is even then the second least element of 
$\mathrm{Spec}(\mathcal C^{\mathrm{hyp}}_n)$ is given by the largest root of the polynomial
$$
X^{2g+1} - 2X^{2g-1} - 2X^{\ceil{4g/3}}- 2X^{\floor{2g/3}+1} - 2X^2 + 1.
$$
Moreover the conjugacy mapping class realizing the minimum is unique.
\end{Theorem}
\par
As mentioned above, one possible way to tackle Theorem~\ref{thm:main} is by the use of the Rauzy--Veech induction. It is now
well established that closed loops in the Rauzy diagram $\mathcal{D}_n$ furnishes pseudo-Anosov
maps {\em fixing a separatrix} (see Section~\ref{sec:rauzy}). Hence the ``usual'' construction is not sufficient to capture all relevant maps. 
In~\cite{BL12}, it is shown that the square of \emph{any} pseudo-Anosov maps fixes 
a separatrix, up to adding a regular point. The cost to pay is that this produces very complicated Rauzy diagrams and the computation of the precise systole seems
out of reach with this method.

We will propose a new construction in order to solve these two difficulties at
the same time. 
The proof of our Main Theorem is divided into three parts of different nature: geometric, dynamical, combinatoric.

\subsection{Geometric part}
In this paper we propose a new construction of pseudo-Anosov map, 
denoted the \emph{Symmetric Rauzy-Veech construction}. One key of this construction is the following
definition.
\begin{Definition}
A pseudo Anosov homeomorphism is positive (respectively, negative) if it fixes (respectively, reverses)  the orientation of the invariant measured foliations.
\end{Definition}
The classical Rauzy--Veech construction associates to any closed path $\gamma$ in the Rauzy diagram a non negative matrix $V(\gamma)$. If this matrix is {\em primitive} (we will also say Perron--Frobenius), \emph{i.e.} a power of $V(\gamma)$ has all its coefficients greater than zero, we obtain a positive pseudo-Anosov map $\pA(\gamma)$ with 
expansion factor $\rho(V(\gamma))$.

In Section~\ref{sec:new:construction} we will associate to any non-closed path $\gamma$ in a given Rauzy diagram $\mathcal{D}$ connecting a permutation $\pi$ to its symmetric $s(\pi)$ (see Section~\ref{sec:new:construction}), a matrix $V(\gamma)$. If the matrix is {\em primitive} we obtain a
negative pseudo-Anosov map $\pA(\gamma)$ with expansion factor $\rho(V(\gamma))$. 

The converse holds in the following sense (see Theorem~\ref{thm:construction:converse} and Proposition~\ref{all:is:SRV}).
\begin{NoNumberTheorem}[Geometric Statement]\
\begin{enumerate}
\item Any affine negative pseudo-Anosov map $\pA$ on a translation surface $S$ that fixes a point is obtained by the Symmetric Rauzy-Veech construction.
\item Any affine  pseudo-Anosov $\pA$ on a translation surface $S\in \mathcal{C}_n^{hyp}$ is obtained as above, up to replacing $\pA$ by $\tau\circ \pA$, where $\tau $ is the hyperelliptic involution. 
\end{enumerate}
\end{NoNumberTheorem}

\subsection{Dynamical part}
From now on, we assume that the underlying connected component is hyperelliptic. 
The theorem above reduces our problem to a combinatorial problem on graphs. However the new problem is still very complicated since the complexity of paths is too large. 
To bypass this difficulty we introduce a renormalization process (ZRL for Zorich Right Left induction), analogously to the Veech--Zorich and Marmi--Moussa--Yoccoz acceleration of the Rauzy induction, see Section~\ref{sec:reduce}. 
This allows us to reduce  considerably the range of paths to analyze in the Rauzy diagram $\mathcal D_n$. 

The particular shape of $\mathcal{D}_n$ will be strongly used. We will usually refer to the permutation $\pi_n$ as the \emph{central permutation}. For any permutation $\pi\in \mathcal{D}_n$, there is a unique shortest path from $\pi_n$ to $\pi$, that can be expressed as a word $w$ with letters `$t$' and `$b$'. We will often write this permutation as $\pi_n.w$.
The loop in $\mathcal{D}_n$ composed by the vertices $\{\pi_n.t^k\}_{k=0,\dots,n-2}$ will be referred to as the \emph{central loop}.

In the next, an \emph{admissible path} is a path in $\mathcal{D}_n$ from a permutation $\pi$ to $s(\pi)$ whose corresponding matrix $V(\gamma)$ is primitive. The path is \emph{pure} if the corresponding pseudo-Anosov map is not obtained by the usual Rauzy--Veech construction. This definition is motivated by Proposition~\ref{prop:fix:sing} and the main result of \cite{BL12}: pseudo-Anosov homeomorphism in hyperelliptic connected components with small expansion factors are \emph{not} obtained by the usual Rauzy--Veech construction.

We show (see Theorem~\ref{thm:main:reduction}):

\begin{NoNumberTheorem}[Dynamical Statement]
There exists a map (denoted by $ZRL$) defined on the set of pure admissible paths on $\mathcal{D}_n$ satisfying the following property: 
The paths $\gamma$ and $ZRL(\gamma)$ define the same pseudo-Anosov homeomorphism up to conjugacy, and there exists an iterate $m\geq 0$ such that $\gamma'=ZRL^{(m)}(\gamma)$ is an admissible path starting from a permutation $\pi$ that belongs to the central loop. In addition the first step of $\gamma'$ is of type 'b'. 
\end{NoNumberTheorem}

This results allows us to restrict the computation to paths $\gamma$ starting from the central loop. As a byproduct one can actually compute {\em all} the expansion factors less than two {\em i.e.} one can compute explicitly the finite set
$$
\{\lambda(\pA);\ \pA : S\rightarrow S,\ S\in \mathcal H_g^{\mathrm{hyp}} \textrm{ and } \lambda(\pA) < 2 \}
$$
for small values of $g$ up to $20$ (see Theorem~\ref{theoremC}). 

\subsection{Combinatoric part}
The last part of the proof deals with the estimates of the spectral radius of all the paths $\gamma$ 
starting from the central loop. We first show that this problem reduces to a problem on a finite number of paths (for a given $n\geq 4$). Then we show how to compare spectral radius $\rho(V(\gamma))$ for various $\gamma$ starting from the permutation $\pi:=\pi_n.t^{k}$ for some $k$.
At this aim, for a given $n\geq 4$, we introduce several notations.
\begin{Notation}
\label{notation:paths}
For $n\in \N$ we set $K_n=\floor{\frac{n}{2}}-1$ and $L_n=n-2-K_n$.\\

For any $k \in \{1,\dots,K_n\}$ we define the path
$$
\begin{array}{lll}
\gamma_{n,k} :& \pi \longrightarrow s(\pi):&  b^{n-1-k}t^{n-1-2k}
\end{array}
$$
For any $k\in\{1,\dots,K_n\}$ and $l\in \{1,\dots ,2n-2-3k\}$ we define the path
$$
\gamma_{n,k,l} : \pi \longrightarrow s(\pi): \left\{
\begin{array}{lll}
b^{l}t^{n-1-k-l}b^{n-1-k-l}t^{n-1-2k} &\mathrm{if}& 1\leq l \leq n-2-k \\
b^{n-1-k}t^{l-(n-1-k)}b^{2(n-1-k)-l}t^{2n-2-3k-l} &\mathrm{if}& n-2-k < l \leq 2n-2-3k
\end{array}
\right.
$$
Loosely speaking $\gamma_{n,k}$ is the shortest path from $\pi$ to $s(\pi)$ starting with a label `b', \emph{i.e.} 
it is the concatenation of the small loop attached to $\pi$ labelled by  `b' followed by the shortest path from $\pi$ to $s(\pi)$. 
Similarly, $\gamma_{n,k,l}$ is a path joining $\pi$ to $s(\pi)$  obtained from $\gamma_{n,k}$ by adding a loop: if $l\leq n-2-k$
then it is a `t' loop based at $\pi_n.t^kb^l$, if $l\geq n-1-k$ then it is a `b' loop based at $\pi_n.t^{l-(n-1-2k)}$ (see Figure~\ref{fig:rauzy}).
\end{Notation}
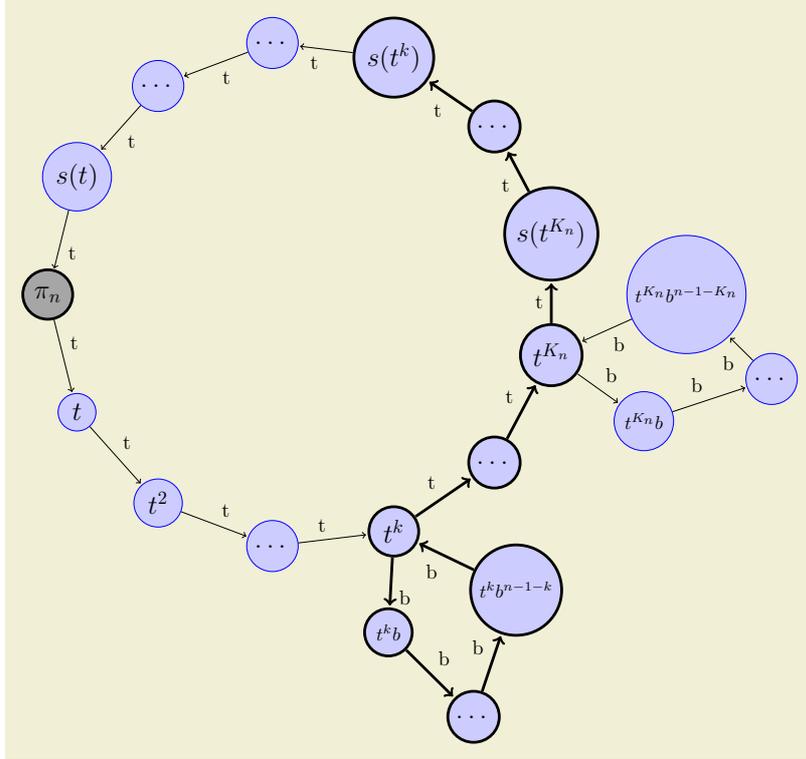
\begin{figure}[htbp]
\scalebox{0.7}{
\begin{tikzpicture}[scale=0.8,shorten >=1pt, auto, node distance=3cm, 
   edge_style/.style={draw=black,->}]

 \fill[fill=yellow!70!black!20,even   odd  rule]   (-7,-11)  rectangle (12,7);
   
\foreach [count=\i] \t in {$\pi_n$,$t$,$t^2$,$\dots$,$t^k$,$\dots$,$t^{K_n}$, $s(t^{K_n})$,$\dots$,$s(t^k)$,$\dots$,$\dots$,$s(t)$}
\ifthenelse{\i<5 \OR \i>10}
        {\node [circle,draw=blue,fill=blue!20!,font=\sffamily\Large\bfseries] (\i) at (180+\i*360/13-360/13:6cm) {\t}}
        {\node [circle,draw=black,fill=blue!20!,font=\sffamily\Large\bfseries,line width=1.5] (\i) at (180+\i*360/13-360/13:6cm) {\t}};
        
\foreach \i/\j in {1/2,2/3,3/4,4/5,10/11,11/12,12/13,13/1}
    \draw[edge_style]   (\i) edge node{t} (\j); 

\node [circle,draw=black,fill=gray!70!,font=\sffamily\Large\bfseries,line width=1.5] at (180+1*360/13-360/13:6cm) {$\pi_n$};

%

\foreach [count=\i] \t/\x/\y in {$\scriptstyle t^{K_n}b$/8/-3,$\dots$/11/-2,$\scriptstyle t^{K_n}b^{n-1-K_n}$/9/0}
  \node [circle,draw=blue,fill=blue!20!,font=\sffamily\Large\bfseries]
        (K\i) at (\x,\y) {\t};

\foreach \i/\j in {7/K1,K1/K2,K2/K3,K3/7}
    \draw[edge_style]   (\i) edge node{b} (\j); 

\foreach [count=\i] \t/\x/\y in {$\scriptstyle t^kb$/2/-8,$\dots$/4/-10,$\scriptstyle t^kb^{n-1-k}$/5/-7}
  \node [circle,draw=black,fill=blue!20!,font=\sffamily\Large\bfseries,line width=1.5]
        (k\i) at (\x,\y) {\t};
\foreach \i/\j in {5/k1,k1/k2,k2/k3,k3/5}
    \draw[edge_style,line width=1.5]   (\i) edge node{b} (\j); 

        
\foreach \i/\j in {5/6,6/7,7/8,8/9,9/10}
    \draw[edge_style,line width=1.5]   (\i) edge node{t} (\j);

\end{tikzpicture}
}
\caption{
\label{fig:rauzy}
Half of the Rauzy diagram. In bold the path $\gamma_{n,k}$.
}
\end{figure}

For any path $\gamma$ in the Rauzy diagram, we will denote by $\theta(\gamma)$ the 
maximal real eigenvalue of the matrix $V(\gamma)$. We will also use the notations $\theta_{n,k}$ (respectively, $\theta_{n,k,l}$) for $\theta(\gamma_{n,k})$ (respectively, $\theta(\gamma_{n,k,l})$).
Observe that the matrix $V(\gamma)$ is not necessarily primitive, but the spectral radius is still an 
eigenvalue by the Perron--Frobenius theorem. We show (see Proposition~\ref{prop:reduce:to:thetank}, Lemma~\ref{lm:reduce:gamma:nk} and Proposition~\ref{prop:n:odd:nKl}):

\begin{NoNumberTheorem}[Combinatoric Statement]
Let $n\geq 4$ be any integer. The followings hold.
\begin{enumerate}
\item Let $\gamma$ be an admissible path starting from $\pi=\pi_n.t^{k}$, for some $k\leq n-1$, such that the first step is `b'. We assume that $\theta(\gamma)<2$. Then,  $k\leq K_n$ and  $\theta(\gamma) \geq \theta_{n,k}$.
Furthermore, if $\gamma\neq \gamma_{n,k}$ then there exists $l\in\{1,\dots,2n-2-3k\}$ such that
$$
\theta(\gamma) \geq \theta_{n,k,l}.
$$
If $V(\gamma_{n,k})$ is not primitive, we can also assume that $l\neq n-1-k$ in the previous statement.
\item For any $k\in\{1,\dots,K_n\}$, let $d=\gcd(n-1,k)$. If $n'=\frac{n-1}{d}+1$ and $k'=\frac{k}{d}$ then the matrix $V_{n',k'}$ is primitive and $\theta_{n,k} = \theta_{n',k'}$.
\item Let $n\equiv 3 \mod 4$ and $l\in \{1,\dots,L_n+3\}$. The matrix
$V_{n,K_n,l}$ is primitive if and only if $l$ is odd. Moreover if $l$ is even then 
$\theta_{n,K_n,l} = \theta_{n',K_{n'},l'}$ with $n'=(n+1)/2$, $l'=l/2$ and $K_{n'}=K_n/2$.
\end{enumerate}
\end{NoNumberTheorem}

\subsection{Proof of the Main Theorem}

Theorems~\ref{thm:main} and ~\ref{thm:main:second} will follow from Theorem~\ref{thm:n:even}, Theorem~\ref{thm:n:3:mod4} and Theorem~\ref{thm:n:1:mod4} below.

\begin{Theorem}
\label{thm:n:even}
The following holds:
\begin{enumerate}
\item If $n\geq 4$ is even then $L(\mathrm{Spec}(\mathcal C_n^{\mathrm hyp}))=\log(\theta_{n,K_n})$.
\item If $n\not \equiv 4 \mod 6$ and $n\geq 18$ is even then the second least element of $\mathrm{Spec}(\mathcal C^{\mathrm{hyp}}_n)$ is $\log(\theta_{n,K_n-1})$.
\end{enumerate}
Moreover the conjugacy mapping class realizing the minimum is unique.
\end{Theorem}
\par
Letting $P_{n,k}$ the characteristic polynomial of $V(\gamma_{n,k})$ {\em multiplied} by $X+1$, by definition $\theta_{n,k}$ is the maximal real root of $P_{n,k}$. By Lemma~\ref{lm:a} we have, when $n$ is even:
$$
P_{n,K_n} =X^{n+1} - 2X^{n-1} - 2X^2 + 1
$$
and when $n\not \equiv 4 \mod 6$ and $\gcd(n-1,K_n-1)=1$:
$$
P_{n,K_n-1} = X^{n+1} - 2X^{n-1} - 2X^{\ceil{2n/3}}- 2X^{\floor{n/3}+1} - 2X^2 + 1
$$
that are the desired formulas.
\par
\begin{proof}[Proof of Theorem~\ref{thm:n:even}]
We will show that $L(\mathrm{Spec}(\mathcal C_n^{\mathrm hyp}))=\log(\theta_{n,K_n})$. Observe that $K_n=n/2-1$ thus $\gcd(n-1,K_n)=1$ and the matrix $V_{n,K_n}$ is primitive by the Combinatoric Statement. This shows $L(\mathrm{Spec}(\mathcal C_n^{\mathrm hyp}))\leq \log(\theta_{n,K_n})$. A simple computation shows $\theta_{n,K_n}<2$ (see Lemma~\ref{lm:a}).

Now by the Dynamical Statement, $L(\mathrm{Spec}(\mathcal C_n^{\mathrm hyp}))=\log(\theta(\gamma))$ for some path $\gamma$ starting from $\pi=\pi_n.t^k$ with first step of type 'b'.
By the Combinatoric Statement, $\theta(\gamma)\geq \theta_{n,k}$ for some $k\in \{1,\dots,K_n\}$. 
Thus $\theta_{n,k} \leq \theta_{n,K_n}$. We need to show that $k=K_n$. Let us assume that $k<K_n$.
\begin{enumerate}
\item If $\gcd(k,n-1)=1$ then Lemma~\ref{lm:compare:n}
implies $\theta_{n,k} > \theta_{n,K_n}$ that is a contradiction.

\item If $\gcd(k,n-1)=d>1$ then by the second point of the Combinatoric Statement, $\theta_{n,k} = \theta_{n',k'}$ where $n'=\frac{n-1}{d}+1$ and $k'=\frac{k}{d}$.
Since $n'$ is even, $k'\neq K_{n'}$ and $\gcd(k',n'-1)=1$, the previous step applies and $\theta_{n',k'}  > \theta_{n',K_{n'}}$.
By Lemma~\ref{lm:decreasing} the sequence $(\theta_{2n,K_{2n}})_n$ is a decreasing
sequence, hence $\theta_{n',K_{n'}} > \theta_{n,K_n}$. Again we run into a contradiction.
\end{enumerate}
In conclusion $k=K_n$ and $L(\mathrm{Spec}(\mathcal C_n^{\mathrm hyp}))=\log(\theta_{n,K_n})$. By construction
and since all inequality above are strict, the conjugacy mapping class realizing this minimum is unique. \medskip

We now finish the proof of the theorem with the second least dilatation. The assumption $n \not \equiv 4 \mod 6$ implies $\gcd(K_n-1,n-1)=1$. Hence $V_{n,K_n-1}$ is primitive and the second least dilatation is less than $\theta_{n,K_n-1}$. As before, a simple computation shows that $\theta_{n,K_n-1}<2$ (see Lemma~\ref{lm:a}).

Conversely, the second least dilatation equals $\theta(\gamma)$ for some admissible path $\gamma$ starting from $\pi=\pi_n.t^k$, for some $k\in\{1,\dots,K_n\}$, with a 'b' as the first step. If $k=K_n$, since $\gamma\neq \gamma_{n,K_n}$, the Combinatoric statements implies that $\theta(\gamma)\geq \theta_{n,K_n,l}$ for some $l\in\{1,\dots,2n-2-3K_n\} = \{1,\dots ,L_n+2\}$. Again the Combinatoric statements shows that if $k\leq K_n-1$ then $\theta(\gamma)\geq \theta_{n,k}$. 

The theorem will follow from the following two assertions, that are proven in Proposition~\ref{compare:roots:second}.
\begin{itemize}
\item For any $k=1,\dots,K_n-2$ one has $\theta_{n,k}>\theta_{n,K_n-1}$.
\item For any $l=1,\dots,L_n+2$ one has $\theta_{n,K_n,l} > \theta_{n,K_n-1}$.
\end{itemize}
This ends the proof of Theorem~\ref{thm:n:even}.
\end{proof}

\begin{Theorem}
\label{thm:n:1:mod4}
If $n\geq 5$ and $n\equiv 1 \mod 4$ then $L(\mathcal C_n^{\mathrm hyp})=\log(\theta_{n,K_n})$.
\end{Theorem}
\par
By Lemma~\ref{lm:a}, when $n=1 \mod 4$, we have that $\theta_{n,K_n}$ is the maximal real root of 
$$
P_{n,K_n}=X^{n+1} - 2X^{n-1} -2X^{\frac{n+1}{2}}- 2X^2 + 1.
$$
\begin{proof}[Proof of Theorem~\ref{thm:n:1:mod4}]
We follow the same strategy than that of the proof of the previous theorem. 
Namely $L(\mathcal C_n^{\mathrm hyp}) \leq \log(\theta_{n,K_n})$ since 
$V_{n,K_n}$ is irreducible (by the Combinatoric statement).

Conversely $L(\mathcal C_n^{\mathrm hyp}) = \log(\theta(\gamma))$, where $\gamma$ is an admissible path starting from $\pi_n.t^k$ for some $k\in\{1,\dots,K_n\}$, and the path starts by a `b'. 
Recall that we have shown $\theta(\gamma) \leq \theta_{n,K_n}$. 
We need to show that for any $k\leq K_n-1$, $\theta_{n,k}>\theta_{n,K_n}$. Again by
the Combinatoric statement, $\theta_{n,k}=\theta_{n',k'}$, where $n'=\frac{n-1}{d}+1$, 
$k'=\frac{k}{d}$ and $d=\gcd(n-1,k)$. By Proposition~\ref{compare:roots:n:1:mod4}
$$
\theta_{n',k'} > \theta_{n,K_n}.
$$
This finishes the proof of Theorem~\ref{thm:n:1:mod4}.
\end{proof}

\begin{Theorem}
\label{thm:n:3:mod4}
If $n\geq 7$ and $n\equiv3 \mod 4$ then $L(\mathcal C_n^{\mathrm hyp})=\log(\theta_{n,K_n,L_n})$.
\end{Theorem}
\par
Proposition~\ref{prop:n:odd:nKl} gives that $\theta_{n,K_n,L_n}$ is the maximal real root of
$$
P_{n,K_n,L_n} = X^{n+1} - 2 X^{n-1} -  4 X^{(n-1)/2+2} + 4  X^{(n-1)/2} + 2 X^2  - 1
$$
that gives the desired formula.
\par
\begin{proof}[Proof of Theorem~\ref{thm:n:3:mod4}]
We follow the strategy used in the two previous proofs. Namely $L(\mathcal C_n^{\mathrm hyp}) \leq \log(\theta_{n,K_n,L_n})$ since 
$V_{n,K_n,L_n}$ is irreducible (see the Combinatoric statement where $n\equiv3 \mod 4$ and $L_n$ is odd). \medskip

Conversely $L(\mathcal C_n^{\mathrm hyp}) = \log(\theta(\gamma))$, where $\gamma$ is an admissible path starting from $\pi_n.t^k$ for some $k\in\{1,\dots,K_n\}$, and the path starts by a `b' (recall that $\theta(\gamma) \leq \theta_{n,K_n,L_n}$).

Assume $k\leq K_n-1$. Then the Combinatoric statement shows that $\theta(\gamma)\geq \theta_{n,k}$.
Letting $n'=(n-1)/d+1$ and $k'=k/d$ where $d=\gcd(n-1,k)$ we have $\theta_{n,k}=\theta_{n',k'}$.
By Proposition~\ref{compare:roots:n:3:mod4}
$$
\theta_{n',k'} > \theta_{n,K_n,L_n}.
$$
This is a contradiction with $\theta(\gamma) \leq  \theta_{n,K_n,L_n}$. Hence we necessarily have $k=K_n$.

Now, the matrix $V_{n,K_n}$ is {\em not} primitive, hence $\gamma\neq \gamma_{n,K_n}$. Thus the Combinatoric statement implies that $\theta(\gamma)\geq \theta_{n,K_n,l}$ for some $l\in\{1,\dots,2n-2-3K_n\} = \{1,\dots ,L_n+3\}$ with $l \neq n-K_n-1=L_n+1$. We will discuss two different cases depending on the
parity of $l$ (recall that $L_n$ is odd). \medskip

\noindent {\bf Case 1. $l$ is even.} By the Combinatoric statement, one has $\theta_{n,K_n,l} = \theta_{n',K_{n'},l'}$ with 
$n'=(n+1)/2$, $l'=l/2$ and $K_{n'}=K_n/2$. If $l< L_n$ then $l' < L_{n'}$ and Lemma~\ref{lm:n:even:nK:2} applies (since $n'\geq 4$ is even):
$$
\theta_{n',K_{n'},l'} > \theta_{n',K_{n'},L_{n'}}.
$$
If $l>L_n$ then $l=L_n+3$. Again the Combinatoric Statement gives $\theta_{n,K_n,L_n+3}=\theta_{n',K_{n'},L_{n'}+2}$. In all cases, Proposition~\ref{compare:roots:n:3:mod4} applies and:
$$
\theta_{n',K_{n'},L_{n'}} > \theta_{n,K_n,L_n}
$$
running into a contradiction. \medskip

\noindent {\bf Case 2. $l$ is odd.} If $l < L_n$ then by Proposition~\ref{cor:l:odd} we  have
$\theta_{n,K_n,l} > \theta_{n,K_n,L_n}$. This is again a contradiction. The case $l>L_n$, namely $l=L_n+2$, 
is ruled out by Lemma~\ref{lm:comparing:matrix}: $\theta_{n,K_n,L_n+2} > 2 > \theta_{n,K_n,L_n}$. \medskip

\noindent In conclusion $l=L_n$ and $\gamma=\gamma_{n,K_n,L_n}$. The Main Theorem is proved.
\end{proof}
\section{Rauzy--Veech induction and pseudo-Anosov homeomorphism}
\label{sec:rauzy}

In this section, we briefly recall the notions of interval exchange transformations, suspension data, Rauzy--Veech induction, and the associated construction of pseudo-Anosov homeomorphisms. We also provide a slight generalization of these standard notions.

\subsection{Interval exchange transformation}
Let $I  \subset \mathbb  R$ be an  open interval  and let us  choose a
finite partition  of $I$  into $d\geq 2$  open subintervals  $\{I_j, \
j=1,\dots,d \}$.  An interval  exchange transformation is a one-to-one
map  $T$  from  $I$  to  itself that  permutes,  by  translation,  the
subintervals $I_j$. It is easy to see that $T$ is precisely determined
by the  following data:  a permutation that
encodes how  the intervals are  exchanged, and a vector  with positive
entries that encodes the lengths of the intervals. \medskip

We  use the representation introduced first by 
Kerckhoff~\cite{Kerckhoff1985} and formalised later by Bufetov~\cite{Bufetov2006} and 
Marmi, Moussa \&Yoccoz~\cite{Marmi:Moussa:Yoccoz}.

We will  attribute a name to  each interval $I_{j}$. In  this case, we
will speak of \emph{labeled} interval  exchange maps.  One gets a pair
of  one-to-one   maps  $(\pi_t,\pi_b)$  (t  for  ``top''   and  b  for
``bottom'') from  a finite alphabet  $\mathcal{A}$ to $\{1,\ldots,d\}$
in  the following  way. In  the partition  of $I$  into  intervals, we
denote the  $k^{\textrm{th}}$ interval, when counted from  the left to
the right,  by $I_{\pi_t^{-1}(k)}$. Once the  intervals are exchanged,
the  interval  number  $k$  is  $I_{\pi_b^{-1}(k)}$.  Then  with  this
convention,  the permutation  encoding  the map  $T$  is $\pi_b  \circ
\pi_t^{-1}$. We  will denote the length  of the intervals  by a vector
$\lambda=(\lambda_\alpha)_{\alpha\in \mathcal{A}}$.

We  will call  the  pair $(\pi_t,\pi_b)$  a  \emph{labeled} permutation,  and
$\pi_b   \circ    \pi_t^{-1}$   a   permutation    (or   \emph{reduced}
permutation).
One usually  represents labeled permutations $\pi=(\pi_t,\pi_b)$ by
a table:
\begin{eqnarray*}
\pi=
\left(\begin{array}{ccccc}\pi_t^{-1}(1)&\pi_t^{-1}(2)&\ldots&\pi_t^{-1}(d)
\\ \pi_b^{-1}(1)&\pi_b^{-1}(2)&\ldots&\pi_b^{-1}(d)
\end{array}\right).
\end{eqnarray*}

\subsection{Suspension data and weak suspension data}
\label{sec:suspension}

The  next  construction  provides  a link  between  interval  exchange
transformations  and  translation surfaces.   A  suspension datum  for
$T=(\pi,\lambda)$      is       a     vector     
$(\tau_\alpha)_{\alpha\in \mathcal{A}}$ such that
\begin{itemize}
\item  $\forall  1  \leq   k  \leq  d-1,\  \sum_{\pi_t(\alpha)\leq  k}
\tau_\alpha>0$,
\item  $\forall  1  \leq   k  \leq  d-1,\  \sum_{\pi_b(\alpha)\leq  k}
\tau_\alpha<0$.
\end{itemize}

We  will  often  use  the notation  $\zeta=(\lambda,\tau)$.  To each
suspension  datum  $\tau$,  we  can associate  a  translation  surface
$(X,\omega)=X(\pi,\zeta)$ in the following way.

Consider the broken line  $L_t$ on $\mathbb{C}=\mathbb R^2$ defined by concatenation of  the vectors $\zeta_{\pi_t^{-1}(j)}$  (in this order) for  $j=1,\dots,d$ with starting  point at  the origin.  Similarly, we consider the broken line $L_b$ defined by concatenation of the vectors $\zeta_{\pi_b^{-1}(j)}$  (in   this  order)  for   $j=1,\dots,d$  with starting point  at the origin.  If  the lines $L_t$ and  $L_b$ have no intersections other than the endpoints, we can construct a translation surface $X$ by identifying each  side $\zeta_j$ on $L_t$ with the side $\zeta_j$  on $L_b$  by  a  translation. The  resulting  surface is  a translation  surface endowed  with the  form $\D  z^2$. Note  that the lines $L_t$ and $L_b$ might  have some other intersection points.  But in this case, one can still  define a translation surface by using the \emph{zippered   rectangle  construction},   due   to  Veech~\cite{Veech1982}.

We will need to extend a little the definition of a suspension datum.
\begin{Definition} Let $T=(\pi,\lambda)$ an interval exchange map. A \emph{weak suspension data} for $T$ is an element $\tau\in \mathbb{R}^{\mathcal{A}}$, such that there exists $h\in \mathbb{R}$ such that:
\begin{itemize}
\item[i-]  $\forall  1  \leq   k  \leq  d-1,\  h+\sum_{\pi_t(\alpha)\leq  k}
\tau_\alpha>0$,
\item[ii-]  $\forall  1  \leq   k  \leq  d-1,\  h+\sum_{\pi_b(\alpha)\leq  k}
\tau_\alpha<0$.
\item[iii-] If $\pi_t^{-1}(1)=\pi_b^{-1}(d)$, then $\sum_{\pi_{t}(\alpha)\neq 1} \tau_{\alpha}<0$
\item[iv-] If $\pi_t^{-1}(d)=\pi_b^{-1}(1)$, then $\sum_{\pi_{b}(\alpha)\neq 1} \tau_{\alpha}>0$
\end{itemize}
\end{Definition}

The parameter $h$ above will be called \emph{height} of the weak suspension datum $\tau$. Note that a usual suspension datum corresponds to the case $h=0$:  in this case, the first two conditions imply the two others.

\begin{figure}[htb]
\begin{tikzpicture}[scale=0.75]
\coordinate (h) at (0,2) ;
\coordinate (e) at (0.02,0) ;
\coordinate (l1) at (1,0);
\coordinate (t1) at (0,-1);
\coordinate (z1) at ($(t1)+(l1)$); 
\coordinate (l2) at (2.29,0);
\coordinate (t2) at (0,-0.423);
\coordinate (z2) at ($(t2)+(l2)$); 
\coordinate (l3) at (1.54,0);
\coordinate (t3) at (0,3.54);
\coordinate (z3) at ($(t3)+(l3)$); 
\coordinate (l4) at (1.18,0);
\coordinate (t4) at (0,-6.27);
\coordinate (z4) at ($(t4)+(l4)$); 
\coordinate (o) at (0,0);
\coordinate (a) at (10,0);

\draw [black,fill=gray!25]  (o)--(h)--++($-1*(t4)-(h)$)--++ ($(l1)-(e)$)--++($(h)+(t4)+(t1)$)--++(e)--++($-1*(h)-(t1)$)--cycle;
\draw  [black,fill=blue!15] (l1)--++($(h)+(t1)$)--++(e)--++($-1*(t4)-(h)-(t1)-(t3)$)--++($(l2)-2*(e)$)--++($(h)+(t4)+(t1)+(t3)+(t2)$)--++(e)--++($-1*(h)-(t1)-(t2)$)--cycle;
\draw  [black,fill=red!15] ($(l1)+(l2)$)--++($(h)+(t1)+(t2)$)--++(e)--++($-1*(t4)-(h)-(t1)$)--++($(l3)-2*(e)$)--++($(h)+(t4)+(t1)+(t3)$)--++(e)--++($-1*(h)-(t1)-(t2)-(t3)$)--cycle;
\draw  [black,fill=green!15] ($(l1)+(l2)+(l3)$)--++($(t1)+(t2)+(t3)$)--++($(l4)$)--++($-1*(t1)-(t2)-(t3)$)--cycle;

\draw [dashed] (h)--++ (z1) node[midway,above] {$z_1$} node {\tiny $\bullet$}
                     --++ (z2) node[midway,above] {$z_2$} node {\tiny $\bullet$}
                     --++ (z3) node[midway,above] {$z_3$} node {\tiny $\bullet$}
                     --++ (z4) node[midway,above] {$z_4$} node {\tiny $\bullet$}
                     --++ ($-1*(z2)$) node[midway,below] {$z_2$} node {\tiny $\bullet$}
                     --++ ($-1*(z3)$) node[midway,below] {$z_3$} node {\tiny $\bullet$}
                     --++($-1*(z1)$) node[midway,below] {$z_1$} node {\tiny $\bullet$}
                     --++($-1*(z4)$) node[midway,below] {$z_4$}  node {\tiny $\bullet$}                  
                     --cycle;
 
\draw [very thick] (o)--($(l1)+(l2)+(l3)+(l4)$)  node[midway,below] {$I_h$};

\draw  [black,fill=blue!15] ($(a)+(l1)$)--++($(h)+(t1)$)--++($(z2)$)--++($-1*(h)-(t1)-(t2)$)--cycle;
\draw  [black,fill=blue!15] ($(a)+(l4)+(l1)+(l3)$)--++($(h)+(t4)+(t1)+(t3)$)--++($(z2)$)--++($-1*(h)-(t1)-(t2)-(t3)-(t4)$)--cycle;
\draw [black,fill=blue!15] ($(a)+(h)+(z4)$)--++($-1*(h)-(t1)-(t2)-(t3)-(t4)$)--++($-0.7*(l4)$)--cycle;

\draw  [black,fill=red!15] ($(a)+(l4)+(l1)$)--++($(h)+(t4)+(t1)$)--++(z3)--++($-1*(h)-(t1)-(t4)-(t3)$)--cycle;
\draw  [black,fill=red!15] ($(a)+(l1)+(l2)$)--++($(h)+(t1)+(t2)$)--++(z3)--++($-1*(h)-(t1)-(t2)-(t3)$)--cycle;

\draw  [black,fill=green!15] ($(a)+(l1)+(l2)+(l3)$)--++($(t1)+(t2)+(t3)$)--++($(l4)$)--++($-1*(t1)-(t2)-(t3)$)--cycle;
\draw [black,fill=green!15] (a)--++(l4)--++($-1*(t1)-(t2)-(t3)$)--++($-1*(l4)$)--cycle;

\draw [black,fill=gray!25]  ($(a)+(h)$)--++($(z1)$)--++ ($-1*(t1)-(h)$)--++($-1*(l1)$)--cycle;
\draw [black,fill=gray!25]  ($(a)+(l4)$)--++($(h)+(t4)$)--++(z1)--++($-1*(h)-(t4)-(t1)$)--cycle;
\draw [black,fill=gray!25]  ($(a)+(h)+(z1)+(z2)+(z3)$)--($(a)+(z1)+(z2)+(z3)$)--++(l4)--cycle;

\draw  [white,fill=white!100] ($(a)+(h)+(z1)+(z2)+(z3)+(z4)$)--($(a)+(h)+(z1)+(z2)+(z3)$)--++($2*(l4)$)--cycle;
\draw  [white, fill=white!100] ($(a)+(h)+(z4)$)--($(a)+(h)$)--++(-1,0)--($(a)+(h)+(t4)$)--cycle;

\draw ($(a)+(h)$)--++ (z1) node[midway,above] {$z_1$} node {\tiny $\bullet$}
                     --++ (z2) node[midway,above] {$z_2$} node {\tiny $\bullet$}
                     --++ (z3) node[midway,above] {$z_3$} node {\tiny $\bullet$}
                     --++ (z4) node[midway,above] {$z_4$} node {\tiny $\bullet$}
                     --++ ($-1*(z2)$) node[midway,below] {$z_2$} node {\tiny $\bullet$}
                     --++ ($-1*(z3)$) node[midway,below] {$z_3$} node {\tiny $\bullet$}
                     --++($-1*(z1)$) node[midway,below] {$z_1$} node {\tiny $\bullet$}
                     --++($-1*(z4)$) node[midway,below] {$z_4$}  node {\tiny $\bullet$}                  
                     --cycle;
                                        
\end{tikzpicture}
\caption{Zippered rectangle construction for a weak suspension datum}
\label{ex:weak:susp}
\end{figure}
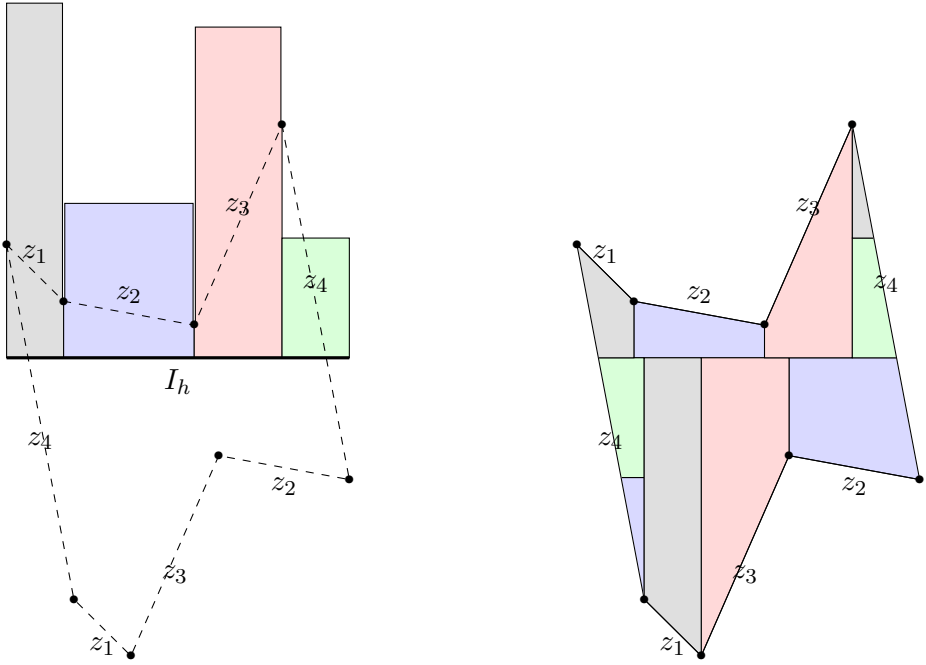

Observe that, in a very similar way as for the case of suspension data, one can associate to a pair $(\pi,\lambda,\tau,h)$ a translation surface $X=X(\pi,\lambda,\tau)$ and a horizontal interval $I_h\subset X$. One use a small variation of the zippered rectangle construction: let $I_h$ be a horizontal interval of length $\sum_\alpha \lambda_\alpha$.  For $\alpha\in \mathcal{A}$, we consider the rectangle $R_\alpha$ of width $\lambda_\alpha$ and of height 
$$
h_\alpha=\sum_{\pi_t(\beta)\leq \pi_t(\alpha)}\tau_\beta-\sum_{\pi_b(\beta)\leq \pi_b(\alpha)}\tau_\beta$$
By definition of weak suspension datum, $h_\alpha>0$ for each $\alpha$. Note that, for instance,  Condition \emph{iii-} is needed to insure that $h_a>0$ for a permutation $\pi$ of the form $\left(\begin{smallmatrix}a &** &*\\ *& **&a \end{smallmatrix}\right)$). Then, we glue the rectangles $(R_\alpha)_\alpha$ in a similar way as for the usual zippered rectangle construction. Two values $h_1,h_2$ give canonically isometric surfaces $X(\pi,\lambda,\tau)$, where the intervals $I_{h_1}$ and $I_{h_2}$ differs by a vertical translation (of length $h_2-h_1$).  In other words, there is a immersed Euclidean rectangle of height $|h_2-h_1|$ whose horizontal sides are $I_{h_1}$ and $I_{h_2}$.  Also, vertical leaves starting from the endpoints of $I_h$ satisfy the ``classical conditions'', \emph{i.e.} each leaf will hit a singularity before intersecting $I_h$ (in the positive or negative direction depending on $h$ and the suspension data). 

Conversely, let $X$ be a translation surface and $I\subset X$ be a horizontal interval with the same classical conditions as above. We assume that there are no vertical saddle connections. In a similar way to the classical case,  there exists a (unique) weak suspension datum $(\pi,\lambda,\tau,h)$ such that $(X,I)=(X(\pi,\lambda,\tau),I_h)$. The datum $(\pi,\lambda)$ is given by considering the interval exchange $T$ defined by the first return map of the vertical flow on $I$, $h$ is the time (positive or negative) for which the vertical geodesic starting from  the left end hits a singularity. The parameters $\tau_\alpha$ are obtained by considering vertical geodesics starting from the discontinuities of $T$ and the time where they hit singularities: the corresponding time $t_k$ for the $k$-th discontinuity is $h+\sum_{\pi_t(\alpha)\leq  k} \tau_\alpha$. Also, if $\pi_t^{-1}(1)=\pi_b^{-1}(d)$ then Condition $iii$- corresponds to the fact that the vertical geodesics starting from  the right end of $I$ hits a singularity before intersecting $I$ again (and similarly for Condition $iv$-).

\subsection{Rauzy-Veech induction and other Rauzy operations}\label{Rauzy:moves}
The Rauzy-Veech induction $\mathcal R(T)$ of $T$ is defined as
the  first return  map of  $T$  to a  certain subinterval  $J$ of  $I$
(see~\cite{Rauzy,Marmi:Moussa:Yoccoz} for details).

We            recall            very            briefly            the
construction.  Following~\cite{Marmi:Moussa:Yoccoz}   we  define  the
\emph{type}   of   $T$    by   $t$   if   $\lambda_{\pi_t^{-1}(d)}   >
\lambda_{\pi_b^{-1}(d)}$   and  $b$   if   $\lambda_{\pi_t^{-1}(d)}  <
\lambda_{\pi_b^{-1}(d)}$. When $T$ is  of type $t$ (respectively, $b$)
we   will   say   that   the  label   $\pi_t^{-1}(d)$   (respectively,
$\pi_b^{-1}(d)$) is the winner and that $\pi_b^{-1}(d)$ (respectively,
$\pi_t^{-1}(d)$) is the loser.  We define a subinterval $J$ of $I$ by
$$  J=\left\{ \begin{array}{ll}  I  \backslash T(I_{\pi_b^{-1}(d)})  &
\textrm{if  $T$ is  of  type t};\\  I  \backslash I_{\pi_t^{-1}(d)}  &
\textrm{if $T$ is of type b.}
\end{array} \right.
$$ The image of $T$  by the Rauzy-Veech induction $\mathcal R$
is defined  as the  first return  map of $T$  to the  subinterval $J$.
This  is again  an interval  exchange transformation,  defined  on $d$
letters  (see  e.g.~\cite{Rauzy}).  The  data  of  $\mathcal R(T)$ are very easy to express 
in term of those of $T$.

There are two cases to distinguish depending whether $T$ is of type
$t$ or $b$; the labeled permutations of $\mathcal{R}(T)$ only depends on
$\pi$ and  on the type of  $T$. If $\varepsilon \in \{t,b\}$ is the type of $T$, 
this defines  two maps $\mathcal{R}_t$ and  $\mathcal{R}_b$ by $\mathcal{R}(T)=(\mathcal{R}_\varepsilon(\pi),
\lambda^{\prime})$. We
will often make use of the following notation: if $\varepsilon\in \{t,b\}$ we denote by 
$1-\varepsilon$ the other element of $\{t,b\}$. 
\begin{enumerate}
\item $T$ has type  $t$. Let $k\in \{1,\dots ,d-1\}$ such that $\pi_b^{-1}(k)=\pi_t^{-1}(d)$.  Then $\mathcal  R_t(\pi_t,\pi_b)=(\pi_t',\pi_b')$ where
$\pi_t=\pi_t'$ and
$$ \pi_b^{'-1}(j) = \left\{
\begin{array}{ll}
\pi_b^{-1}(j) & \textrm{if $j\leq  k$}\\ \pi_b^{-1}(d) & \textrm{if $j
= k+1$}\\ \pi_b^{-1}(j-1) & \textrm{otherwise.}
\end{array} \right.
$$

\item $T$ has type  $b$. Let $k\in \{1,\dots ,d-1\}$ such that $\pi_t^{-1}(k)=\pi_b^{-1}(d)$. Then $\mathcal  R_b(\pi_t,\pi_b)=(\pi_t',\pi_b')$ where
$\pi_b=\pi_b'$ and
$$ \pi_t^{'-1}(j) = \left\{
\begin{array}{ll}
\pi_t^{-1}(j) & \textrm{if $j\leq  k$}\\ \pi_t^{-1}(d) & \textrm{if $j
= k+1$}\\ \pi_t^{-1}(j-1) & \textrm{otherwise.}
\end{array} \right.
$$

\item Let  us denote  by $E_{\alpha\beta}$ the  $d\times d$  matrix of
which the  $\alpha,\beta$-th element  is equal to  $1$, all  others to
$0$.   If  $T$   is   of   type  $t$   then   let  $(\alpha,\beta)   =
(\pi^{-1}_t(d),\pi^{-1}_b(d))$   otherwise   let   $(\alpha,\beta)   =
(\pi^{-1}_b(d),\pi^{-1}_t(d))$. Then $V_{\alpha\beta}\lambda'=\lambda$, 
where $V_{\alpha\beta}$  is  the transvection          matrix          $I+E_{\alpha\beta}$.
\end{enumerate}

If $\tau$  is a  suspension data  over $(\pi,\lambda)$
then we define $\mathcal R(\pi,\lambda,\tau)$ by
$$        \mathcal        R(\pi,\lambda,\tau)       =        (\mathcal
R_{\varepsilon}(\pi),V^{-1}\lambda,V^{-1}\tau),
$$ where $\varepsilon$ is the type of $T=(\pi,\lambda)$ and $V$ is the
corresponding      transition     matrix.      In      other     terms
$V_{\alpha\beta}\zeta'=\zeta$ where $\zeta=(\lambda,\tau)$.

\begin{Remark}
\label{rk:isometric}
By  construction  the  two  translation  surfaces  $X(\pi,\zeta)$  and
$X(\pi',\zeta')$ are naturally isometric (as translation surfaces).
\end{Remark}

\begin{Remark}
 We can extend the Rauzy--Veech operation on the space of weak suspension data by using the same formulas.  We easily see that $V^{-1}\zeta$ is a weak suspension data for $\mathcal{R}_{\varepsilon}(\pi)$. Also, $X(\mathcal{R}(\pi,\zeta))$ and $X(\pi,\zeta)$ are also naturally isometric as translation surfaces.
\end{Remark}

Now  if we  iterate the  Rauzy  induction, we  get  a sequence
$(\alpha_k,\beta_k)$          of          winners/losers.      Denoting 
$\mathcal{R}^{(n)}(\pi,\lambda)=(\pi^{(n)},\lambda^{(n)})$,   the
transition  matrix  that  relates  $\lambda^{(n)}$ to  $\lambda$  is  the
product of the transition matrices:
\begin{eqnarray}
\label{eq:path}
\left(\prod_{k=1}^{n} V_{\alpha_k\beta_k}\right)\lambda^{(n)}=\lambda.
\end{eqnarray}

Now, we define other Rauzy moves that will be used later. Let $\pi$ be a labeled permutation. 
\begin{eqnarray*}
\pi=
\left(\begin{array}{ccccc}\pi_t^{-1}(1)&\pi_t^{-1}(2)&\ldots&\pi_t^{-1}(d)
\\ \pi_b^{-1}(1)&\pi_b^{-1}(2)&\ldots&\pi_b^{-1}(d)
\end{array}\right).
\end{eqnarray*}

We define the \emph{symmetric} of $\pi$, denoted by $s(\pi)$, the following labeled permutation.
$$s(\pi)=
\left(\begin{array}{ccccc}\pi_b^{-1}(d)&\pi_b^{-1}(d-1)&\ldots&\pi_b^{-1}(1)
\\ \pi_t^{-1}(d)&\pi_t^{-1}(d-1)&\ldots&\pi_t^{-1}(1)
\end{array}\right).$$
Observe that if $\tau$ is a weak suspension datum for $(\pi,\lambda)$, then $\tau$ is also a weak suspension datum for $(s(\pi),\lambda)$ (it is not necessarily true for usual suspension data). For simplicity, we define $s(\pi,\lambda,\tau)=(s(\pi),\lambda,\tau)$.
 Also, the translation surfaces $X(\pi,\lambda,\tau)$ and $X(s(\pi,\lambda,\tau))$ are related by the element $-I\in \textrm{SL}(2,\mathbb{R})$.

Left Rauzy induction can be defined analogously as the Rauzy induction, by ``cutting'' the interval on the left. It can also be defined by $\mathcal{R}_L=s\circ \mathcal{R} \circ s$. From the above study, we see that $\mathcal{R}_L$ preserves weak suspension data, and that $X((\pi,\lambda,\tau))$ and $X(\mathcal{R}_L(\pi,\lambda,\tau))$ are naturally isometric as translation surfaces.

\subsection{Labeled Rauzy diagrams}

For  a labeled  permutation $\pi$,  we call  the  \emph{labeled Rauzy  diagram},  denoted  by  $\mathcal  D(\pi)$,  the  graph  whose
vertices are all  labeled permutations that we can  obtained from $\pi$
by the  combinatorial Rauzy moves.  From  each vertices, there
are two edges labeled $t$ and  $b$ (the type) corresponding to the two
combinatorial    Rauzy   moves.   We  will   denote  by   $\pi
\xrightarrow{\alpha,\beta}  \pi'$   for  the  edge   corresponding  to
$\mathcal R_\varepsilon(\pi) = \pi'$ where $\varepsilon\in\{t,b\}$ and
$\alpha/\beta$  is the winner/loser.  To each  path $\gamma$  of this
diagram, there is thus a sequence  of winners/losers. We will denote by
$\widetilde{V}(\gamma)$    the   product   of    the   transition    matrices   in
Equation~(\ref{eq:path}).

Similarly, one can define the reduced Rauzy diagram $\mathcal{D}_{red}(\pi)$ by considering (reduced) permutations as vertices. There is clearly a canonical map $\mathcal{D}(\pi)\to \mathcal{D}_{red}(\pi)$.

\subsection{Coordinates on the hyperelliptic Rauzy diagrams}\label{coordinates}
The hyperelliptic Rauzy diagram $\mathcal{D}_n$ of size $2^{n-1}-1$ is the one that contain the permutation
$$\pi_n= \begin{pmatrix}
1&2&\dots &n \\ n&n-1&\dots &1
\end{pmatrix}.
$$
When $n$ is even, it corresponds to the hyperelliptic connected component $\mathcal{H}^{hyp}(n-2)$ and when $n$ is odd, it corresponds to the hyperelliptic component $\mathcal{H}^{hyp}(\frac{n-1}{2},\frac{n-1}{2})$.

The precise description of these diagrams was given by Rauzy \cite{Rauzy}. See also \cite[Section~3.7.1]{BL12}. An easy corollary is the following proposition. 
\begin{Proposition}
Let $n\geq 2$. For any $\pi \in \mathcal{D}_n$ there exists a unique shortest path joining $\pi_n$ to $\pi$. 
\end{Proposition}

For a permutation $\pi$ in $\mathcal{D}_n$ as above. We can write the path of the proposition as a unique sequence of $n_1>0$ Rauzy moves of type $\varepsilon\in \{t,b\}$, then $n_2>0$ Rauzy moves of type $1-\varepsilon$, etc. The sequence of non negative integers $n_1,\dots ,n_{k-1}$ defines the permutation $\pi$ once $\varepsilon$ is chosen. Observe that replacing the starting Rauzy type $\varepsilon$ by $1-\varepsilon$ changes $\pi=(\pi_t,\pi_b)$ by $(\pi_b,\pi_t)$ up to renaming. For our purpose, we wont need to distinguish these two cases. 
Observe also that we necessarily have $\sum_{i=1,\dots,k-1} n_i< n-1$. 

\begin{Definition}
Let $\pi$ be a permutation in $\mathcal{D}_n$ and let $n_1,\dots ,n_{k-1}$ be as above, the \emph{coordinates} of $\pi$ are 
$$(n_1,\dots ,n_{k-1},n-1-\sum_{i=1,\dots,k-1} n_i)$$
\end{Definition}

This definition is motivated by the following easy proposition.

\begin{Proposition}
\label{prop:coordinates}
Let $\pi$ be a permutation in $\mathcal{D}_n$. The followings hold:
\begin{itemize}
\item If $\pi$ has coordinates $(n_1,\dots ,n_k)$ then $s(\pi)$ has coordinates $(n_k,\dots ,n_1)$.
\item If $k$ is even then $\pi$ and $s(\pi)$ belong to the same  component of $\mathcal{D}_n\backslash\{\pi_n\}$.
\item If $k$ is odd then $\pi$ and $s(\pi)$ belong to a different  component of $\mathcal{D}_n\backslash\{\pi_n\}$.
\end{itemize}
\end{Proposition}

\begin{proof}[Proof of Proposition~\ref{prop:coordinates}]
The proof is straightforward. To see that the last two assertions hold, we can remark 
that $\pi$ and $s(\pi)$ belong to the same connected component of $\mathcal{D}_n\backslash\{\pi_n\}$
if and only if the minimal paths (from $\pi_n$) to $\pi$ and $s(\pi)$ have the same starting Rauzy type.
\end{proof}

We end this section by recalling the construction of Veech on pseudo-Anosov homeomorphisms.

\subsection{Pseudo-Anosov homeomorphism and Rauzy-Veech induction}
We follow Veech's work~\cite{Veech1982}.
To any path $\gamma$ in the labeled Rauzy diagram, whose image in the reduced Rauzy diagram is closed,
one  can associate  a matrix  $V$ as  follows. We denote by  $(\pi_t, \pi_b)$ and $(\pi_t',\pi_b')$ the labeled  permutations corresponding to the endpoints of $\gamma$. By definition of $\gamma$, they both define the same underlying permutation. We associate to it a matrix $\widetilde{V}$ as before. Let $P$ be  the permutation matrix defined by permuting the columns  of   the  $d\times  d$  identity  matrix   according  to  the
maps $\pi_t,\pi_t'$, \emph{i.e.} $P=[p_{\alpha\beta}]_{\alpha,\beta\in \mathcal{A}^2}$, with $p_{\alpha\beta}=1$ if $\beta=\pi^{-1}_t(\pi'_t(\alpha))$ and $0$ otherwise.
The transition  matrix associated to the path $\gamma$ is then:
\begin{eqnarray}
\label{eq:matrix:unlabel}
V = \widetilde{V} \cdot P.
\end{eqnarray}
Observe that $V$ is obtained from $\widetilde{C}$ by replacing, for each $k\in \{1,\dots ,d\}$, the column $\pi_t^{-1}(k)$ by the column $\pi_t^{'-1}(k)$
or for each letter $\alpha$, the column labelled $\alpha$ by the column labelled ${\pi'}_t^{-1}\circ\pi_t(\alpha)$.

Assume now that the matrix $V$ is a primitive (\emph{i.e.} it as a power with only positive entries). Let   $\theta > 1$ be its Perron-Frobenius eigenvalue.  We choose a positive eigenvector
$\lambda$   for   $\theta$.   It    can   be   shown   that   $V$   is
symplectic~\cite{Veech1982}, thus let  us choose an eigenvector $\tau$
for the  eigenvalue $\theta^{-1}$ with  $\tau_{\pi^{-1}_{t}(d)}>0$. It
turns   out    that   $\tau$   defines   a    suspension   data   over
$T=(\pi,\lambda)$. Indeed, the set of suspension data is an open cone,
that  is  preserved by  $V^{-1}$.  Since  the  matrix $V^{-1}$  has  a
dominant eigenvalue $\theta$ (for  the eigenvector $\tau$), the vector
$\tau$ must belong to this cone. If $\zeta=(\lambda,\tau)$, one has, for some integer $k$
\begin{multline*}
\mathcal{R}^k(\pi,\zeta) = (\pi,V^{-1}\zeta) = (\pi,V^{-1}\lambda,V^{-1}
\tau)   =   (\pi,\theta^{-1}\lambda,\theta   \tau)   =  \\   =   g_{t}
(\pi,\lambda,\tau), \qquad \textrm{ where } \qquad t=\log(\theta) > 0.
\end{multline*}
Hence the  two surfaces $X(\pi,\zeta)$  and $g_{t}X(\pi,\zeta)$ differ
by    some    element    of    the   mapping    class    group    (see
Remark~\ref{rk:isometric}).    In   other   words   there   exists   a
pseudo-Anosov  homeomorphism   $\pA$ affine with  respect  to  the
translation surface  $X(\pi,\zeta)$ and such that $D \pA  = g_{t}$. The action 
of $\pA$ on the relative homology of $(X,\omega)$ is $V(\gamma)$ thus 
the expanding factor of  $\pA$ is $\theta$. 

By construction $\pA$ fixes the zero on the left of the interval $I$ and thus it fixes a horizontal separatrix
adjacent to this zero (namely, the oriented half line corresponding to the interval $I$). Conversely:

\begin{NoNumberTheorem}[Veech]
Any pseudo-Anosov homeomorphism (preserving a orientable foliation, and) fixing a horizontal separatrix is obtained
by the above construction.
\end{NoNumberTheorem}

Recall that the main result of~\cite{BL12} was based on the following key proposition:

\begin{Proposition}[\cite{BL12}, Propositions 4.3 \& 4.4]
\label{prop:fix:sing}
For any $n\geq 2$ and any pseudo-Anosov homeomorphism $\phi$ that is affine with respect to
$(X,\omega) \in  \mathcal{C}^{hyp}_n$, if $\phi$  fixes a horizontal separatrix attached to a zero of $\omega$, or a marked point, 
then its expanding factor is bounded from below by~$2$.
\end{Proposition}

\section{Obtaining all ``small'' pseudo-Anosov homeomorphisms}
\label{sec:new:construction}

The proof of Theorem~\ref{thm:main} uses a generalization of the Rauzy--Veech construction of pseudo-Anosov homeomorphisms. We explain in Appendix~\ref{first:attempt} why a  naive generalization of the Rauzy induction does not work.

\subsection{Construction of negative pseudo-Anosov maps}

The usual Rauzy--Veech construction naturally produces pseudo-Anosov maps that preserve the orientation of the stable and unstable foliation. The proposed generalization  produces maps that reverse the unstable foliation. 

We consider a non closed path $\gamma$ in the labeled Rauzy diagram such that its starting point $\pi$ and its ending point $\pi'$ satisfy $s(\pi)=\pi'$ up to a relabelling. By a slight abuse of language we say that $\gamma$ is a path from $\pi$ to $s(\pi)$.

As above,  we associate to such path a matrix $V$ by multiplying the corresponding product of the transition matrices by a suitable permutation matrix that corresponds to the relabelling between $s(\pi')$ and $\pi$.
As before, $V$ is symplectic, thus let  us choose an eigenvector $\lambda$ for the eigenvalue $\theta$ and $\tau$
for the  eigenvalue $\theta^{-1}$. It turns out that $\tau$ is not necessarily a suspension datum, but it is a weak suspension datum, as we show in the next Proposition.

\begin{Proposition}
\label{prop:construction:path}
Let $\gamma$ by a path in a Rauzy diagram from $\pi$ to $s(\pi)$, and let $V:=V(\gamma)$ be the corresponding matrix. 
We assume that $V$ is primitive. If $\lambda,\tau$ are the eigenvectors as above then
\begin{itemize}
\item up to multiplying $\tau$ by $-1$, $(\lambda,\tau)$ is a weak suspension data for $\pi$.
\item the constructed surface $S(\pi,\lambda,\tau)$ admits an affine map $\phi$ whose derivative is $$ \begin{pmatrix}
-\theta &0 \\ 0& -1/\theta 
\end{pmatrix}$$
where $\theta $ is the maximal eigenvalue of $M$.
\item Furthermore, $\phi$ admits a regular fixed point (with positive index).
\end{itemize}
\end{Proposition}

\begin{proof}[Proof of Proposition~\ref{prop:construction:path}]
Let $C(h)$ be the cone of weak suspension data of  heights $h$ over $\pi$. We have a map
$V^{-1} : \mathbb PC(h) \longrightarrow \mathbb PC(h')$. Observe that $C(h)$ is open and the union 
$C=\bigcup_p C(p)$ is an open convex cone. Hence there is an element $\tau\in C$ such that $[\tau] \in \mathbb P\overline{C}$ is fixed by $V^{-1}$. Hence there is $\theta \in \R$ such that $V^{-1}\tau = \theta\tau$. If $\tau \in \partial C$ then it still defines a nice translation surface since there is no vertical saddle connection (see~\cite[Figure 14]{BL} for a detailed argument). In this case  the corresponding horizontal segment contains some singularities and thus $[\tau]$ cannot be fixed by $V^{-1}$ that is a contraction.
Hence $V^{-1}\tau = \theta\tau \in \mathring{C}$. Furthermore, $\theta$ is the greatest eigenvalue of $V^{-1}$, hence of $V$.  We construct the surface $(X,\omega)=X(\pi,\lambda, \tau)$. If $\zeta=(\lambda,\tau)$,
 one has, for some integer $k$
$$
\mathcal{R}^k(\pi,\zeta) = (s(\pi),V^{-1}\zeta) =   (s(\pi),\theta^{-1}\lambda+i\theta   \tau)  
$$
Now, there is a natural map from $f_1$ from $S_1=X(\pi,\zeta)=X(s(\pi),\theta^{-1}\lambda+i\theta   \tau)$ to $S_2=X(s(\pi),\zeta)$ with $Df_1=\left(\begin{smallmatrix} -1 &0\\0 &-1\end{smallmatrix}\right)$ and a natural map $f_2$ from $S_2=X(s(\pi),\zeta)$ to $S_1=X(s(\pi),\theta^{-1}\lambda+i\theta   \tau)$ with  $Df_2=\left(\begin{smallmatrix} \theta^{-1} &0\\0 &\theta\end{smallmatrix}\right)$ and the composition $\pA=f_2\circ f_1$ gives a pseudo-Anosov map affine on $S_1$ with $\theta$ as expansion factor and that reverse vertical and horizontal foliations.

\begin{figure}[htb]
\begin{tikzpicture}[yscale=0.5]
\coordinate (h) at (0,2) ;
\coordinate (l1) at ($(1,0)$);
\coordinate (t1) at (0,-1);
\coordinate (z1) at ($(t1)+(l1)$); 
\coordinate (l2) at ($(2.29,0)$);
\coordinate (t2) at (0,-0.423);
\coordinate (z2) at ($(t2)+(l2)$); 
\coordinate (l3) at ($(1.54,0)$);
\coordinate (t3) at (0,3.54);
\coordinate (z3) at ($(t3)+(l3)$); 
\coordinate (l4) at ($(1.18,0)$);
\coordinate (t4) at (0,-6.27);
\coordinate (z4) at ($(t4)+(l4)$); 
\coordinate (o) at (0,0);
\coordinate (a) at (9,-18);
\coordinate (th) at ($2.3*(h)$);
\coordinate (b) at ($(9,0)+(h)-(th)$);
\coordinate (tl1) at (0.43,0);
\coordinate (tt1) at (0,-2.3);
\coordinate (tz1) at ($(tt1)+(tl1)$); 
\coordinate (tl2) at ($0.43*(l2)$);
\coordinate (tt2) at ($2.3*(t2)$);
\coordinate (tz2) at ($(tt2)+(tl2)$); 
\coordinate (tl3) at ($0.43*(l3)$);
\coordinate (tt3) at ($2.3*(t3)$);
\coordinate (tz3) at ($(tt3)+(tl3)$); 
\coordinate (tl4) at ($0.43*(l4)$);
\coordinate (tt4) at ($2.3*(t4)$);
\coordinate (tz4) at ($(tt4)+(tl4)$);


\draw (h)--++ (z1) node[midway,above] {$1$} node {\tiny $\bullet$}
                     --++ (z2) node[midway,above] {$2$} node {\tiny $\bullet$}
                     --++ (z3) node[midway,above] {$3$} node {\tiny $\bullet$}
                     --++ (z4) node[midway,right] {$4$} node {\tiny $\bullet$}
                     --++ ($-1*(z2)$) node[midway,below] {$2$} node {\tiny $\bullet$}
                     --++ ($-1*(z3)$) node[midway,below] {$3$} node {\tiny $\bullet$}
                     --++($-1*(z1)$) node[midway,below] {$1$} node {\tiny $\bullet$}
                     --++($-1*(z4)$) node[midway,left] {$4$}  node {\tiny $\bullet$}                  
                     --cycle;
 
\draw [thick] (o)--($(l1)+(l2)+(l3)+(l4)$)  node[midway,below] {$I_h$}; 
                                        
\draw ($(a)-(h)$)--++ ($-1*(z1)$) node[midway,below] {$1$} node {\tiny $\bullet$}
                     --++ ($-1*(z2)$) node[midway,below] {$2$} node {\tiny $\bullet$}
                     --++ ($-1*(z3)$) node[midway,below] {$3$} node {\tiny $\bullet$}
                     --++ ($-1*(z4)$)  node[midway,left] {$4$} node {\tiny $\bullet$}
                     --++ ($(z2)$) node[midway,above] {$2$} node {\tiny $\bullet$}
                     --++ ($(z3)$) node[midway,above] {$3$} node {\tiny $\bullet$}
                     --++($(z1)$) node[midway,above] {$1$} node {\tiny $\bullet$}
                     --++($(z4)$) node[midway,right] {$4$}  node {\tiny $\bullet$}                  
                     --cycle;
                     
\draw [thick] (a)--++($-1*(l1)-(l2)-(l3)-(l4)$)  node[midway,above] {$f_1(I_h)$};                     
    
\draw ($(b)+(th)$)--++ ($(tz2)$) node[midway,above] {$2$} node {\tiny $\bullet$}
                     --++ ($(tz3)$) node[midway,left] {$3$} node {\tiny $\bullet$}
                     --++ ($(tz1)$) node[midway,right] {$1$} node {\tiny $\bullet$}
                     --++ ($(tz4)$)  node[near start,right] {$4$} node {\tiny $\bullet$}
                     --++ ($-1*(tz1)$) node[midway,left] {$1$} node {\tiny $\bullet$}
                     --++ ($-1*(tz2)$) node[midway,below] {$2$} node {\tiny $\bullet$}
                     --++($-1*(tz3)$) node[midway,right] {$3$} node {\tiny $\bullet$}
                     --++($-1*(tz4)$) node[midway,left] {$4$}  node {\tiny $\bullet$}                  
                     --cycle;
                     
\draw [thick] (b)--++($(tl1)+(tl2)+(tl3)+(tl4)$)  node[midway,below] {$f_2\circ f_1(I_h)$};         
\draw [thick] ($(b)+(th)-(h)$)--++($(tl1)+(tl2)+(tl3)+(tl4)$)  node[midway,above] {$I'_h$};         
\draw [dotted] ($(b)+(th)-(h)$)--++($(l1)+(l2)+(l3)+(l4)$)  node[midway,below] {$I_h$};                   

\draw ($(b)+(th)+0.3*(l1)+0.3*(l2)+0.3*(l3)+0.3*(l4)+0.7*(t1)+0.7*(t2)+0.7*(t3)+0.7*(t4)$) node {\tiny $\bullet$};
\draw [dashed] ($(b)+0.3*(l1)+0.3*(l2)+0.3*(l3)+0.3*(l4)$)--($(b)+(th)-(h)+0.3*(l1)+0.3*(l2)+0.3*(l3)+0.3*(l4)$) node[midway, right] {$J_x$};

\node at (7.5,0) {$ \simeq$};
\node at (7.5,-1) {\tiny{(Rauzy--Veech)}};
           
\draw [->]  (3,-6)--++(2,-6)     node[midway, left] {$f_1$}       ;
\draw [->]  (7,-12)--++(2,6)     node[midway, left] {$f_2$}       ;
                                        
\end{tikzpicture}
\caption{The Symmetric Rauzy--Veech construction}
\label{SRV:construction}
\end{figure}
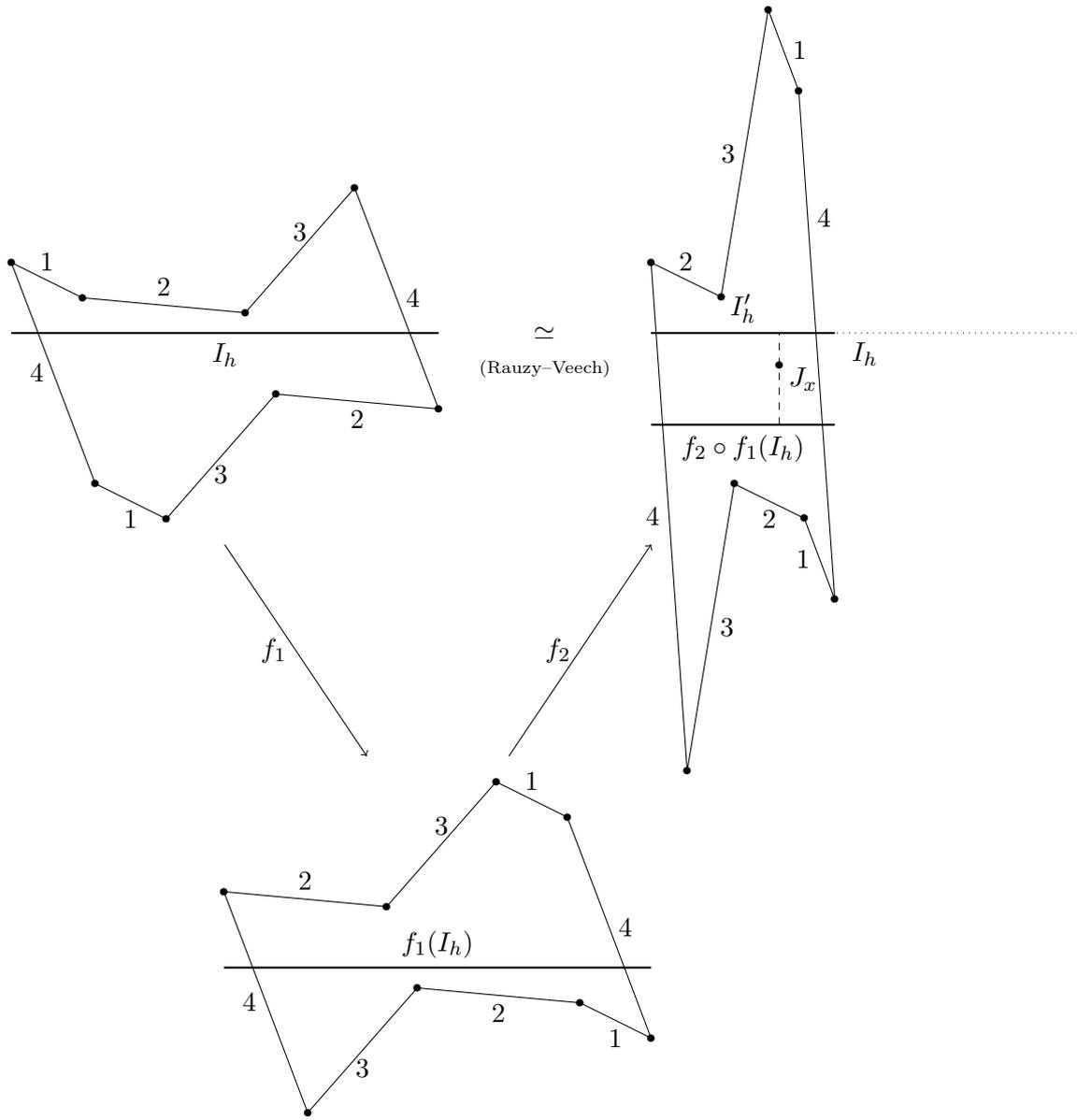

Now we prove that $\pA$ has a regular fixed point. Let $h$ be a height of $(\pi,\lambda,\tau)$, and $I_h$ the corresponding interval. Then, Rauzy--Veech induction gives a subinterval $I'_{h}$ of $I_h$, and the image of $I_h$ by $\pA$ gives a interval $I'_{h_2}$ defining the same weak suspension datum, with a different height (see Figure~\ref{SRV:construction}). Hence, there is a map from $\pA(I_h)$ to $I'_{h_1}\subset I_h$ with derivative $-\theta^{-1}$. It as a fixed point $x$, hence $x$ and $\pA(x)$ are the endpoint a vertical segment $J_x$ that do not contain a singularity. There is a fixed point of $\pA$ in $J_x$, which concludes the proof.
 \end{proof}

\subsection{Converse of the Symmetric Rauzy--Veech construction}

As for the usual Rauzy--Veech construction, we have a converse.

\begin{Definition}
Let  $\pA$ be an affine pseudo-Anosov on $S$, that that reverse the orientation of the foliations, and with a regular fixed point (with positive index). A curve $L$ is suitable for $\pA$ if
\begin{enumerate}
\item  it is made by a  horizontal segment, starting from $p$, and then followed by a  vertical segment, ending at a singular point. We do
not allow $L$ to have self-intersections. 
\item $L$ and $\pA(L)$ do not have intersections in their interior.
\end{enumerate}
Given a suitable curve, we call base segment the horizontal part of $L\cup \pA(L)$
\end{Definition}

\begin{Proposition}\label{base:segment:to:weak:sd}
A base segment defines, by considering first return map of the vertical flow, an interval exchange transformation and weak suspension datum  that can be obtained by the above construction. 
\end{Proposition}

\begin{proof}
The fact that the base segment $I=]a,b[$ determines an interval exchange transformation $T=(\pi,\lambda)$ and a  weak suspension datum $\tau,h$ for $T$ is similar to the case for classical suspension data, except that the segment is not attached on the left to a singularity. This is left to the reader. 

We only need to check that  $T$ defines a Rauzy path from $\pi$ to $s(\pi)$ such that $\lambda,\tau$ are the corresponding eigenvectors.

The key remark is the following: the horizontal part of $\pA(L)\cup \pA^2(L)$ is a segment $I'\subset I$ that has the same left end as $I$. Observe that $\pA(L)$ and $\pA^2(L)$ do not have intersection on their interior hence, by a classical argument, the  interval exchange transformation $T'$ associated to $I'$ is obtained from $T$ by applying a finite number of time the Rauzy induction on the right to $T$. Similarly, for the weak suspension datum, we get
$(\pi',\lambda',\tau')=(\pi^{(n)},\lambda^{(n)},\tau^{(n)})$.

Rotating by $180^\circ$ the picture we have, up to relabelling,  $s(\pi^{(n)})=\pi$,  $\lambda^{(n)}=1/\theta \lambda$ and $\tau^{(n)}=\theta \tau$. This ends the proof.
\end{proof}

Now we prove the first part of the Geometric Statement.
\begin{Theorem}
\label{thm:construction:converse}
Let $\pA$ be an affine pseudo-Anosov on $S$ having a regular fixed point $P$ and that reverse horizontal and vertical foliations. Then $S$ is obtained by the above construction.
\end{Theorem}

\begin{proof}[Proof of Theorem~\ref{thm:construction:converse}]
By Proposition~\ref{base:segment:to:weak:sd}, all we need to show is to produce a suitable curve $L$ for $\pA$.

We start from any oriented curve $L$ in $X$ made by a finite horizontal segment, starting from $P$, and then followed by a finite vertical segment, ending at a singular point. We do not allow $L$ to have self-intersections. Such a curve always exists ({\em e.g.} by using the Veech's polygonal representation of translation surfaces). We denote by $L_x$ and $L_y$ the 
(oriented) components of $L$: they bound a rectangle $R$ whose opposite corners are
$P$ and a zero $\sigma$ of $\omega$. We denote by $\{c(L)\} = L_x\cap L_y$ and by $h$ the length of $L_y$. \medskip 

Now if $\pA(L) \cap \overset{\circ}{L} = \emptyset$ we are done. 
Otherwise one of the following intersections 
is non empty (possibly the two):
$$
\pA(L_x)\cap\overset{\circ}{L}_y \neq \emptyset \qquad \textrm{or} \qquad \pA(L_y)\cap\overset{\circ}{L}_x\neq\emptyset.
$$
We will perform several operations on $L$, in order to obtain the required condition. The strategy is the following:
\begin{itemize}
\item 1st Step: arrange that $L$ bounds an immersed euclidean rectangle $i(R)$.
\item 2nd Step: arrange that $\pA(L_x)\cap\overset{\circ}{L}_y=\emptyset$.
\item 3nd Step: change the fixed point in order to get suitable curve.
\end{itemize}

\noindent {\bf 1st Step.} 
We assume that $h$ is \emph{minimal} in the following sense: for each $x\in L_x$, the unit speed vertical geodesic starting from $x$ (in the same direction as $L_y$), does not hit a singularity at a time less than $h$. 
Now,  let $R\subset \mathbb{R}^2$ be the open rectangle of width $|L_x|$ and of height $h$.There is a natural translation map $i:R\to S$ that sends the bottom side of $R$ to $L_x$. By the above hypothesis, $i(R)$ does not contain any singularity. Note that $i$ might not be an embedding but just an immersion. However, $P$ is not in $i(R)$, otherwise one easily see that the interior of $L_x$ intersects the interior of $L_y$ (see also Figure~\ref{step1:fig}).
\medskip 

Assume that $h$ is not minimal, then there is a $x_0\in L_x$ whose corresponding vertical geodesic hits a singularity for a time $h_0$ minimal. We then consider the new oriented curve $L_0$, starting from $P$ such that ${L_0}_x$ is the segment joining $P$ to $x_0$ and ${L_0}_y$ is the vertical segment of length $h_0$. Note that $L_0$ still satisfies the no self intersection hypothesis otherwise we would find an element $x_1\in L_x$ with $h_1<h_0$, contradicting the minimality hypothesis.

Note also that if $\pA(L_x)\cap\overset{\circ}{L}_y=\emptyset$, then it is still also the case for $L_0$. Indeed, the rectangle $R_0$ of width $|L_x|$ and of height $h_0$ immerses in $S$ (in a similar way as $R$). If $\pA({L_0}_x)\cap\overset{\circ}{L_0}_y\neq \emptyset$, then we find $P$ in $i(R_0)$, which is not possible for the same reason as above (see Figure~\ref{step1:fig}).

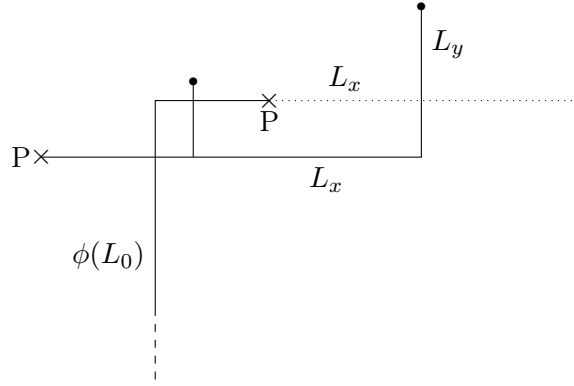
\begin{figure}[htb]
\begin{center}
\begin{tikzpicture}
\draw  (0,0) node {$\times$} node[left] {P}--(5,0) node[near end,below] {$L_x$}--(5,2) node {\tiny $\bullet$} node[near end,right] {$L_y$};
\draw  (2,1) node {\tiny $\bullet$ } --++ (0,-1);
\draw  (3,0.75) node {$\times$} node[below] {P}--(1.5,0.75)--(1.5,-2) node[near end, left] {$\pA(L_0)$};
\draw[dashed] (1.5,-2)--(1.5,-3);
\draw[dotted]  (3,0.75)--(7,0.75) node[near start,above] {$L_x$};
\end{tikzpicture}
\caption{Step 1: If $\pA(L_x)\cap {L}_y=\emptyset$, and $h$ is not minimal}
\label{step1:fig}
\end{center}
\end{figure}

\medskip
\noindent {\bf 2d Step.} Now we assume $\pA(L_x)\cap\overset{\circ}{L}_y \neq \emptyset$. We first show that $\pA(L_y)\cap\overset{\circ}{L}_x\neq\emptyset$. 
Let $Q$ be the point in the intersection $\pA(L_x)\cap\overset{\circ}{L}_y$ such that the vertical distance from $c(L)$ to $Q$ is minimal. Since 
$|\pA(L_x)| = \lambda^{-1}|L_x| < |L_x|$ one has $\pA(c(L))\in i(R)$. If 
$\pA(L_y)\cap\overset{\circ}{L}_x = \emptyset$ then the vertical segment $\pA(L_y)$ is contained in $i(R)$, in particular $\pA(\sigma)\in i(R)$: this contradicts the 1st Step since there is no singularity inside $i(R)$. \medskip

Now we replace $L$ by $L'$ as follows: We choose $Q'$ in $\pA(L_y)\cap\overset{\circ}{L}_x$ such that the horizontal distance from $P$ to $Q'$ is minimal. Then, we define  $L'$ by considering the horizontal segment, starting from $P$ and ending at $Q'$, and the vertical segment from $Q'$ and ending at $\pA(\sigma)$.
Since $\pA(L'_x)\cap L'_y \subset \pA(L_x)\cap\pA(L_y) = \{\pA(c(L))\}$, one has 
$\pA(L'_x)\cap \overset{\circ}L'_y =\emptyset$ as required. Now, up to shortening $L'$ as in the first step, we can assume that $L'_x$ and $L'_y$ bound an immerse rectangle $R'$ and we still have 
 $\pA(L'_x)\cap\overset{\circ}{L'}_y =\emptyset$.  \medskip
 
\noindent {\bf 3d Step.} Let $\widetilde{S}$ be the universal covering of $S$. Choose $\widetilde{P}$ a preimage of $P$, $\widetilde{L}$ a preimage of $L$ attached to $\widetilde{P}$. Now the rectangle $R$, as defined in the 1st Step, embeds as a rectangle $\widetilde{R}$ in $\widetilde{S}$ with $\widetilde{L}$ as bottom and right sides. For any lift $\widetilde{\pA}$ of $\pA$, such that $\widetilde{\pA}(\widetilde{L_y})$ intersects $\widetilde{R}$, $\widetilde{\pA}(\widetilde{L_y})$ intersects the interior of $\widetilde{L_x}$ (in a unique point $Q$, since we are working on the universal cover). Now, we choose a lift $\widetilde{\pA}$ of $\pA$ that minimize the length $d$ of the vertical segment joining $Q$ to the singular point that is the  end of $\widetilde{\pA}(\widetilde{L_y})$. Now, we easily see that $\widetilde{\pA}(\widetilde{R})$ intersects $\widetilde{R}$ as in Figure~\ref{step3}. As in the proof of Proposition~\ref{prop:construction:path}, we find $x\in L_x$ such that the corresponding vertical leaf is fixed by $\widetilde{\pA}$. Then, we find in $\widetilde{R}$ a fixed point $\widetilde{P'}$ for $\pA$.

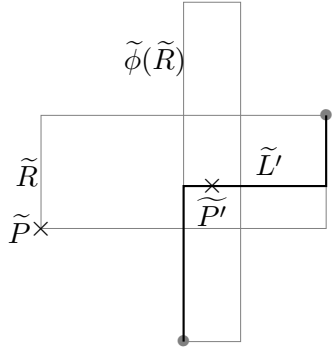
\begin{figure}[htb]
\begin{center}
\begin{tikzpicture}[scale=0.75]

\draw[very thin, gray] (0,0) rectangle (5,2) node {$\bullet$};
\draw[very thin, gray]  (2.5,-2) node {$\bullet$} rectangle (3.5,4);
\draw[thick] (3,0.75) node {$\times$} node[below] {$\widetilde{P'}$} -- node[midway,above] {$\widetilde{L'}$} (5,0.75) -- (5,2);
\draw[thick] (3,0.75) -- (2.5,0.75) -- (2.5,-2);

\draw (-0.25,1) node {$\widetilde{R}$};
\draw (2,3) node {$\widetilde{\pA}(\widetilde{R})$};
\draw (0,0) node {$\times$} node[left] {$\widetilde{P}$};
\end{tikzpicture}
\caption{Step 3: changing the fixed point}
\label{step3}
\end{center}
\end{figure}
Now we consider $\widetilde{L'}$ obtained as follows: take the horizontal segment with left end $\widetilde{P'}$ and whose right end is in $\widetilde{L_x}$, then consider the vertical segment that ends in the same singularity as $\widetilde{L}$ (see Figure~\ref{step3}). We claim that $\widetilde{L'}$ projects in $S$ into a suitable curve $L'$. 
Indeed, otherwise $L'_x\cup \pA(L'_x)$ either intersects the interior of $L'_y$  or the interior of $\pA(L'_y)$. In both cases, we find another intersection point between $L_x$ and $\pA(L'_y)$ which contradicts the minimality of $d$. Therefore, we have found a suitable curve, the theorem is proven. 
\end{proof}
\par
The next proposition implies the  second part of the Geometric Statement. 
\begin{Proposition}\label{all:is:SRV}
Let $\pA$ be an affine pseudo-Anosov map on a surface $S$ in a hyperelliptic component, and $\tau$ be the hyperelliptic involution. Then, $\tau\circ\pA$ is also an affine pseudo-Anosov map on $S$. Denote  $\{\pA,\tau\circ\pA\}=\{\pA^+,\pA^-\}$ such that $\pA^+$ preserves the orientation of the vertical and horizontal foliations. We have the following:
\begin{itemize}
\item $\pA^-$ is obtained by the symmetric Rauzy--Veech construction.
\item If $\pA^+$ is not obtained by the usual Rauzy--Veech construction, then there are exactly two regular fixed points for $\pA^-$, that are interchanged by the hyperelliptic involution.
\end{itemize}
\end{Proposition}
\begin{proof}
We prove the first part. From the previous theorem, all we need to show is that $\pA^-$ has a regular fixed point. 
For a homeomorphism $\pA$, we denote by $\pA_*$ the linear action of $\pA$ on the homology $H_1(S,\mathbb{R})$. We recall the Lefschetz formula:
$$2-Tr(\pA_*)=\sum_{\pA(x)=x} Ind(\pA,x).$$
When $\pA$ is of type pseudo-Anosov, and $x$ a fixed point, we can show that:
\begin{itemize}
\item  $Ind(\pA,x)<0$ if there is a fixed separatrix.
\item $Ind(\pA,x)=1$ otherwise. 
\end{itemize}
Assume that $\pA=\pA^+$ and the underlying translation surface is in $\mathcal{H}(2g-2)$. Then the unique singularity $P$ is necessarily fixed and has index $\leq 1$. All the other fixed point (possibly zero) are regular point so have negative index. Hence $2-Tr(\pA_*)\leq 1$. Since $Tr(\tau\circ \pA)=-Tr(\pA)$, we conclude that $2-Tr(\pA^-_*)\geq 3$. Therefore, there must be at least 2 regular fixed points for $\pA^-$. 

Now assume that the underlying translation surface is in $\mathcal{H}(g-1,g-1)$. The two singularities $P_1,P_2$ are either fixed or interchanged by $\pA^+$.
\begin{itemize}
\item If $P_1,P_2$ are fixed. Then, as before, $2-Tr(\pA^+_*)\leq 2$ hence $2-Tr(\pA^-_*)\geq 2$. So there are a least two fixed points. But $P_1,P_2$ are interchanged by $\pA^-$, hence the two fixed points are regular.
\item If $P_1,P_2$ are interchanged by $\pA^+$, then $2-Tr(\pA^+_*)\leq 0$, hence $2-Tr(\pA^-_*)\geq 4$, there are at least four fixed points so, at least 2 regular fixed points.
\end{itemize}
Now, we see that in the above proof, the case where $\pA^+$ is not obtained by the usual Rauzy--Veech construction (\emph{i.e.} when $\pA^+$ does not have negative index fixed point) corresponds to the equality case. There is exactly one pair of regular fixed points $\{Q_1,Q_2\}$. Since $\pA_-$ and $\tau$ commute, we see that $\tau(Q_1)$ is a fixed point, hence $Q_1$ or $Q_2$. It cannot be $Q_1$ otherwise $\tau\circ \pA_-=\pA^+$ has a regular fixed point, contradicting the hypothesis. Hence, $\tau(Q_1)=Q_2$, this concludes the proof.
\end{proof}

We end this section with the following useful proposition.
\begin{Proposition}
\label{prop:centralpermutation}
If an admissible path from $\pi$ to $s(\pi)$ passes through the central permutation then $\tau\circ\pA$ is obtained by the usual Rauzy--Veech construction, where $\tau$ is the hyperelliptic involution.
\end{Proposition}

\begin{proof}[Proof of Proposition~\ref{prop:centralpermutation}]
The map  $\psi=\tau\circ \pA$ preserves the orientation of the vertical and horizontal leafs. Hence, all we need to show is that there is a fixed separatrix. 

Let $(\pi,\lambda,\tau)$ be the weak suspension data associated to the admissible path. Let $h$ be a height, and $I=I_h\subset X(\pi,\lambda,\tau)$. 

We claim that there is an immersed Euclidean rectangle with $\psi(I)$ as one horizontal side, and the other one is a subinterval of $I$. Assuming the claim, 
 there is an isometry $f$ from $\psi(I)$ to $I$ obtained by following a vertical leaf. The map $f\circ \psi$ is therefore a contracting map from $I$ to itself (its derivative is $\theta^{-1}$), hence has a fixed point. It means that there is an element $x$ in $I$ whose image by $\psi$ is in the vertical leaf $l$ passing through $x$. Thus, this vertical leaf $l$ is preserved by $\psi$. Since $\psi$, restricted to $l$ as derivative $\theta\neq 1$, there is a fixed point of $\psi$ on $l$. This fixed point is either a conical singularity or a regular point. In any case, $\psi$ fixes a vertical separatrix. Hence $\psi$ fixes also a horizontal separatrix. It is therefore obtained by the usual Rauzy--Veech construction.

Now we prove the claim. By construction of $\pA$, there is base segment $I$ such that $I'=\pA(I) \subset I$. This subinterval $I'$ is obtained from $I$ after doing Rauzy--Veech induction until the end of the defining path in the Rauzy diagram. By hypothesis, there is a step of the form $(\pi_n,\lambda'',\tau'')$, where $\pi_n$ is the central permutation, and that corresponds to a interval $I''$, with $I'\subset I'' \subset I$. Now, we easily see that $\tau(I'')$ and $I''$ are the horizontal sides of a immersed rectangle. We now conclude as in the proof of Proposition~\ref{prop:construction:path} that there is a fixed separatrix.
\end{proof}

\section{Renormalization and changing the base permutation}
\label{sec:reduce}

The aim of this section is to reduce our analysis to the set of paths of $\mathcal D_n$ starting from the central loop. 
First, recall that, from \cite{BL12}, pseudo-Anosov homeomorphisms in hyperelliptic connected components obtained by the Rauzy--Veech construction have expansion factors at least 2, and are not the minimal ones. Recall that an admissible path is pure if the corresponding map is not obtained by the usual Rauzy--Veech construction.

\begin{Theorem}
\label{thm:main:reduction}
Let $\gamma$ be a pure admissible path in $\mathcal D_n$ joining a permutation $\pi$ to its symmetric $s(\pi)$, and let $\pA$ be the corresponding pseudo-Anosov map. Then, $\pA$ can  also be obtained by a path $\gamma'$ starting 
from the central loop. In addition the first step of $\gamma'$ is on a secondary loop ({\em i.e.} of type 'b').
\end{Theorem}

The construction of $\gamma'$ in terms of $\gamma$ is not obvious. To prove the
theorem, we appeal to the dynamics of the induction.

\subsection{Renormalization: The ZRL  acceleration}

In this section we will consider only pure admissible paths from $\pi$ to $s(\pi)$. If we denote by $(n_1,\dots ,n_k)$ the coordinates of $\pi$ then $k$ is even (otherwise Proposition~\ref{prop:coordinates} implies that $\gamma$ passes through the central permutation and Proposition~\ref{prop:centralpermutation} is a contradiction with the path being pure).

\begin{Convention}\label{conv}
We will always assume that the first Rauzy move is '$t$', \emph{i.e.} $\pi$ is obtained by applying the sequence $t^{n_1}b^{n_2}\dots t^{n_{k-1}}$ to $\pi_n$.
\end{Convention}
\par
\begin{Proposition}
\label{prop:ZRL}
Let $\alpha=\pi_b^{-1}(n)$ and $\beta=\pi_t^{-1}(1)$. Then $\phi(\zeta_\alpha)=\zeta_\beta$.

In particular there exists a canonical  base segment $I'$ obtained from the previous one as follows: 
We apply a right induction until $\alpha$ is is loser (\emph{i.e.}  Rauzy path of type $b^l t$, for some $l\geq 0$) followed by a left induction until $\beta$ is loser (\emph{i.e.}  left Rauzy path of type $\bar{t}^m \bar{b}$, for some $m\geq 0$). This segment $I'$ defines a new Rauzy path $\gamma'$ starting point $\pi'$ of $\gamma'$ satisfying Convention~\ref{conv} (up to permuting top and bottom).
\end{Proposition}

\begin{figure}[htb]
\begin{tikzpicture}[yscale=0.7, xscale=1.5]
\coordinate (h) at (0,2) ;
\coordinate (l1) at ($(1,0)$);
\coordinate (t1) at (0,-1);
\coordinate (z1) at ($(t1)+(l1)$); 
\coordinate (l2) at ($(2.29,0)$);
\coordinate (t2) at (0,-0.423);
\coordinate (z2) at ($(t2)+(l2)$); 
\coordinate (l3) at ($(1.54,0)$);
\coordinate (t3) at (0,3.54);
\coordinate (z3) at ($(t3)+(l3)$); 
\coordinate (l4) at ($(1.18,0)$);
\coordinate (t4) at (0,-6.27);
\coordinate (z4) at ($(t4)+(l4)$); 
\coordinate (o) at (0,0);


\draw (h)--++ (z1) node[midway,above right] {$\zeta_1=\pA(\zeta_2)$} node {\tiny $\bullet$}
                     --++ (z2) node[midway,above] {$\zeta_2$} node {\tiny $\bullet$}
                     --++ (z3) node[midway,above] {$\zeta_3$} node {\tiny $\bullet$}
                     --++ (z4) node[midway,right] {$\zeta_4$} node {\tiny $\bullet$}
                     --++ ($-1*(z2)$) node[midway,below] {$\zeta_2$} node {\tiny $\bullet$}
                     --++ ($-1*(z3)$) node[midway,below] {$\zeta_3$} node {\tiny $\bullet$}
                     --++($-1*(z1)$) node[midway,below] {$\zeta_1$} node {\tiny $\bullet$}
                     --++($-1*(z4)$) node[midway,left] {$\zeta_4$}  node {\tiny $\bullet$}                  
                     --cycle;
 
                        
\draw[dashed] ($(h)+0.3*(l1)+0.3*(l2)+0.3*(l3)+0.3*(l4)+0.7*(t1)+0.7*(t2)+0.7*(t3)+0.7*(t4)$) node {\tiny $\bullet$} -- ($(h)+0.7*(t1)+0.7*(t2)+0.7*(t3)+0.7*(t4)+(l1)+(l3)+(l4)$) --
($(h)+(z1)+(z3)+(z4)$) node[near start,right] {$L'$} ;

\draw[dashed] ($(h)+0.3*(l1)+0.3*(l2)+0.3*(l3)+0.3*(l4)+0.7*(t1)+0.7*(t2)+0.7*(t3)+0.7*(t4)$) node {\tiny $\bullet$}  -- ($(h)+0.7*(t1)+0.7*(t2)+0.7*(t3)+0.7*(t4)+(l1)$) -- ($(h)+(z1)$) node[midway, right] {$\pA(L')$};                    
\end{tikzpicture}
\caption{Finding a new suitable curve: $\pA(\zeta_2)=\zeta_1$.}
\label{fig:new:base:segment}
\end{figure}
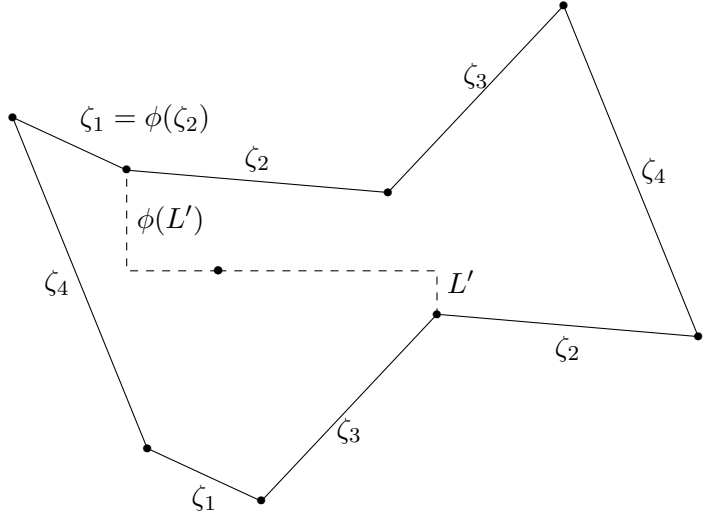
\begin{Definition}
The map $ZRL$ from the space of pure admissible paths satisfying Convention~\ref{conv} to itself is defined by $ZRL(\gamma)=\gamma'$, where $\gamma,\gamma'$ are as in Proposition~\ref{prop:ZRL}.
\end{Definition}
\par
\begin{proof}[Proof of Proposition~\ref{prop:ZRL}]
By hypothesis, $\pi$ is in the connected component of $\mathcal{D}_n\backslash\{\pi_n\}$ where $\beta=\pi_t^{-1}(1)$ is never winner.  Since $\gamma$ does not pass through the central permutation, $\alpha$ is never winner, hence the parameter $\zeta_\beta$ is unchanged during the whole Rauzy path. In particular, by definition of $\pA$ (see Figure~\ref{SRV:construction}), the segment corresponding to $\zeta_\alpha$ is sent by $\pA$ to the segment corresponding to $\zeta_\alpha$. Hence, the curve $L'$ as in Figure~\ref{fig:new:base:segment} is admissible and defines a new admissible path $\gamma'$ from a permutation $\pi'$ to $s(\pi')$. The horizontal part of $L'\cup \pA(L')$ defines a new base segment $I'$ and parameters $(\pi',\zeta')$ as described in Section~\ref{sec:suspension}. We obtain $(\pi',\zeta')$ by the sequence of right and left Rauzy move described in the statement of the proposition. If $\pi'$ does not satisfy Convention~\ref{conv}, we interchange the two lines, that is equivalent to interchanging "up" and "down" on the surface, and therefore conjugating $\pA$ with the surface orientation changing homeomorphism.
\end{proof}
\begin{Remark}
ZRL stands for Zorich acceleration of Right-Left induction. By construction the two surfaces
$S(\gamma)$ and $S(\gamma')$ are in the same Teichm\"uller orbit (perhaps up to an affine isometry with derivative map $ \left(\begin{smallmatrix} 1&0\\ 0& -1 
\end{smallmatrix}\right)$). It is difficult to express the path $\gamma'$ in a simple way in terms of the path $\gamma$. 
However the coordinates for the permutations (introduced in Section~\ref{coordinates}) allow us to express the new starting point $\pi'$ in an easier way, only in terms of the old starting point (see Proposition~\ref{prop:action:of:ZRL} below).
\end{Remark}
In view of considering iterates of the map ZRL, we will use the following lemma.

\begin{Lemma}\label{infinite:winners}
Assume that $\gamma$ is a pure admissible path. 
Then the ZRL orbit of $\gamma$ is infinite and all the letters are winner and loser infinitely often.
\end{Lemma}

\begin{proof}
By Proposition~\ref{prop:ZRL}, the new path $\gamma'$ is a pure admissible path. Hence, the ZRL orbit is infinite. In this proof, we do not interchange top and bottom line in order to follow Convention~\ref{conv}, but consider a sequence of base segment $(I_n)_{n\in \mathbb{N}}$ on the same underlying surface $S$, with $I_{n+1}\subset I_n$. Note that the vertical flow on $S$ is minimal (and uniquely ergodic) since $S$ carries an affine pseudo-Anosov homeomorphism. We identify those segments as subsets of the real line, $I_n=]b_n,a_n[$, with $(b_n)$ increasing and $(a_n)$ decreasing. As for the usual Rauzy induction, all letter are winner infinitely often if and only if $|I_n|=a_n-b_n$ tends to zero. Let us assume that it is not the case. Then $a_n\to a$, $b_n\to b$ with $\pA(a)=b$. We consider the sequence $(\lambda_n,\tau_n)$ of corresponding suspension data. Define $t_n$ to be the time when the vertical unit speed geodesic starting from $a_n$ reach a singularity, note that $t_n$ can be positive or negative, observe that $|t_n|\to \infty $ otherwise the set of singularities in the surface $S$ would not be discrete. 
Without loss of generalities, we can assume that there is a subsequence $(n_k)$ such that $t_{n_k}>0$. There are two cases.
\begin{enumerate}
\item If the geodesic  in the positive direction starting from $a$, denoted as $g_a$, is infinite, then it follows the (finite) one starting from $a_{n_k}$ for a longer and longer time. By density there is a time $t$ such that $g_b(t)$ intersects $]b,a[$, hence for $k$ large enough, the geodesics starting from $b_{n_k}$ intersects $]b,a[$, hence $I_{n,k}$ before reaching a singularity. Contradiction.
\item If the geodesic $g_a$ is finite in the positive direction, then it is necessarily infinite in the negative direction. If there is a subsequence $t_{m_k}<0$, the same argument as above gives a contradiction. Hence, for $n$ large enough $t_n>0$. The same argument also works for an non accelerated Right-Left induction. Hence for $n$ large enough, the `right' part of the ZRL move is necessarily $b$. Similarly, the `left' part of the ZRL move is necessarily $\bar{t}$. But such move cannot hold for an infinite number of step. Contradiction.
\end{enumerate}
The lemma is proved.
\end{proof}

\subsection{Action of ZLR on the starting point}

As promised we now explain how ZRL acts on the starting point, in terms of the coding introduced in Section~\ref{coordinates}.

\begin{Proposition}
\label{prop:action:of:ZRL}
Let $\gamma$ be an admissible path, starting from a permutation $\pi$. 
Let  $(n_1,\dots ,n_k)$, $k\geq 4$ be the coding of $\pi$. The coding of the starting point of $ZRL(\gamma)$ is obtained by the following rules.
\begin{itemize}
\item we replace $(n_{k-1},n_k)$ by $\left\{\begin{array}{ll} (n_{k-1}+1,n_k-1),&\ \mathrm{or} \\  (n_{k-1},n_k-1,1),&\ \mathrm{or} \\ (n_{k-1},\bar{l},1,n_k-1-\bar{l}).\end{array}\right.$
\item we replace $(n_1,n_2)$ by $\left\{\begin{array}{ll} (n_1-1,n_2+1),&\ \mathrm{or} \\  (1,n_1-1,n_2),&\ \mathrm{or} \\ (n_1-1-\bar{m},1,\bar{m},n_2).\end{array}\right.$
\end{itemize}
where $l,m$ are positive integers with $l<n_k-1$, $m<n_1-1$,
and $\bar{l}$, respectively $\bar{m}$, is the remainder of the Euclidean division of $l$, respectively $m$, by $n_k$, respectively $n_1$.
\end{Proposition}
\begin{proof}[Proof of Proposition~\ref{prop:action:of:ZRL}]
By convention, the starting permutation $\pi$ is of the form $t^{n_1}b^{n_2}\dots t^{n_{k-1}}$. The map ZRL acts on $\pi$ by a sequence of Rauzy moves of the form $b^lt \bar{t}^m \bar{b}$ and then followed (perhaps) by a permutations of the lines (which does not change the coding).

The Rauzy moves  $b^l t$ acts on the coding as:
\begin{enumerate}
\item $(n_1,\dots ,n_k)\mapsto (n_1,\dots ,n_{k-1}+1,n_{k}-1)$ if $l=0 \mod(n_k)$.
\item $(n_1,\dots ,n_k)\mapsto (n_1,\dots ,n_{k-1},n_k-1,1)$ if $l=-1 \mod (n_k)$.
\item $(n_1,\dots ,n_k)\mapsto (n_1,\dots ,n_{k-1},\bar{l},1,n_k-1-\bar{l})$ otherwise.
\end{enumerate}
(where $\bar{l}$ is the remainder of the Euclidean division of $l$ by $n_k$).\\
This gives the first part of the proposition.

The remaining part is obtained similarly: we must act on the left by the moves 
$\bar{t}^m \bar{b}$, which is equivalent to the moves $s.b^m. t.s$. This proves the proposition.
\end{proof}

\subsection{Proof Theorem~\ref{thm:main:reduction}}
We are now ready to prove the theorem announced at the beginning of this section.
\begin{proof}[Proof Theorem~\ref{thm:main:reduction}]
Let $\gamma$ be an admissible path, with corresponding expansion factor $\lambda<2$. We need to show that there exists an iterate of ZRL which  starts from a permutation $\pi$ starting from the central loop, \emph{i.e.} satisfying $\pi_t^{-1}(n)=\pi_b^{-1}(1)$.

We first observe the following: let $(l,\dots ,l')$ be a coordinate of $\pi$, then $l\leq l'$. Indeed, if $l>l'$, then a path joining $\pi$ to $s(\pi)$ must pass through the central permutation.

Now, let $(n_1,,n_2,\dots ,n_{k-1},n_k)=(l,x,\dots ,x',l')$ be the coding of a permutation, with $k\geq 4$ even. We prove by induction the following property $\mathcal{P}(l)$:
\begin{enumerate}
\item $l=l'$
\item After applying a finite sequence of ZRL, we reach the permutation coded by $(l+x,\dots ,x'+l')$, and during this sequence, the  letters ``between the blocks corresponding to $x$ and $x'$'', \emph{i.e.}  $\pi_\varepsilon^{-1}(j)$ for $j\in \{l+x+2,\dots ,n-1-l'-x'\}$ are constant along the ZRL-orbit and non-winner. Note that if $k\geq 4$, this set of letters is nonempty.
\end{enumerate}

Initialisation corresponds to the case $(1,x,\dots ,x',l')$. Assume that $l'>1$, then after one step of $ZRL$, the left part must be $(1+x,\dots )$, but in this case, in order to preserve the parity of the numbers of blocks, we must have:
$$(x+1,\dots ,x', l'-1,1)$$ 
with $x+1>1$, which contradicts the initial observation. So, we have $l'=1$, and ZRL maps $(1,x,\dots ,x',1)$ to $(x+1,\dots ,x'+1)$, and we see directly that the second condition is fulfilled.

Now, let $1< l \leq l'$ and let the initial permutation be coded by $(l,x,\dots ,x',l')$. Assume $\mathcal{P}(l'')$ for any $l''<l$. By Proposition~\ref{prop:action:of:ZRL} and observing again that the parity of the number of blocks is constant, after one step of ZRL the coding is one of the following:
\begin{enumerate}
\item $(1,l-1,x,\dots ,x',l'-1,1)$.
\item $(l-1,x+1,\dots ,x'+1,l'-1)$.
\item $(l-1,x+1,\dots ,x',l_2',1,l_1')$, for some $l_1',l_2'$ satisfying $l_1'+l_2'+1=l'$.
\item $(l_1,1,l_2,x,\dots ,x'+1,l'-1)$, for some $l_1,l_2$ satisfying $l_1+l_2+1=l$.
\item $(l_1,1,l_2,x,\dots ,x',l_2',1,l_1')$, for $l_1,l_2,l_1',l_2'$ as above.
\end{enumerate}
Now, we study these different cases:
\begin{enumerate}
\item The following step is necessarily $(l,x,\dots ,x',l')$.
\item By the induction hypothesis, $l-1=l'-1$, and after some steps, we obtain $(l+x,\dots ,l'+x')$.
\item By induction hypothesis, we have $l-1=l_1'$, and after some steps of ZRL we have $(l+x,\dots ,x',l_2',l)$, which again contradicts the first observation.
\item By induction hypothesis, we have $l_1=l'-1$. But, $l_1=l-1-l_2'<l-1\leq l'-1$. Contradiction.
\item By induction hypothesis, we have $l_1'=l_1$ and after some steps, we have $(l_1+1,l_2,x,\dots ,x',l_2',l_1+1)$, and again after some steps, we have $(l,x,\dots ,x',l')$. During these two sequences of ZRL, all the letters between the blocks corresponding to $l_2$ and $l_2'$ are unchanged and non-winner, and this set is nonempty.
\end{enumerate}
Note that it is impossible to repeat infinitely the steps $(1)$ or $(5)$ by Lemma~\ref{infinite:winners}, hence we will eventually get Step $(2)$, proving $\mathcal{P}(l)$.

Hence, after a finite number of ZRL steps, we obtain a permutation with a coding with two blocks (k=2), which corresponds to a starting point in the central loop. This proves the first part of Theorem~\ref{thm:main:reduction}. \medskip

We now turn into the proof of the second part: the first step of the corresponding admissible Rauzy path leaves the central loop, \emph{i.e.} is of type $b$.
In this case, ZRL acts on the starting point $\pi$ by $t\bar{b}$. Write $\pi$ as:
$$
\begin{pmatrix}
a&***&b \\ b&***&c 
\end{pmatrix}
$$
We easily see that $t\bar{b}$ preserves $\pi$, and $\lambda_b$ becomes $\lambda_b-\lambda_a-\lambda_c$. Iterating ZRL, after a finite number of steps, $b$ is not winner any more. Theorem~\ref{thm:main:reduction} is proved.
\end{proof}
\section{Reducing to a finite number of paths}
\label{sec:special}

In view of Section~\ref{sec:new:construction}, for a given $n$, one needs to control the spectral radii 
of matrices $V(\gamma)$ for all paths $\gamma$ in $\mathcal D_n$. Section~\ref{sec:reduce} shows that it is enough to consider paths $\gamma$ starting from the central loop of $\mathcal D_n$. However this still produces an infinite number of paths to control. The proposition below shows how one can restrict
our problem to a \emph{finite} number of paths.

\begin{Proposition}\label{prop:reduce:to:thetank}
Let $\gamma$ be an admissible path starting from a permutation $\pi=\pi_n.t^k$ in the central loop. 
We assume that the first step goes in the secondary loop. 
\begin{enumerate}
\item If $k> K_n$, then $\theta(\gamma)>2$.
\item If $k\leq K_n$ then $\theta(\gamma) \geq \theta_{n,k}$. 
\item If $V(\gamma_{n,k})$ is not primitive, then either $\theta(\gamma)>2$ or there exists $l\in\{1,\dots,2n-2-3k\}$, $l\neq n-1-k$, such that 
$\theta(\gamma)\geq \theta_{n,k,l}$.
\end{enumerate}
\end{Proposition}

For two paths $\gamma,\gamma'$  we will write $\gamma' \leq \gamma$ if the path $\gamma'$
is a subset of the graph $\gamma$ (viewed as a ordered collection of edges).
For a real matrix $A\in M_n(\R)$, we will write $A\geq 0$ (respectively, $A\gg0$)
to mean that $A_{ij}\geq 0$ (respectively, $A_{ij}> 0$) for all indices $1\leq i,j\leq n$, 
and similarly for vectors $v\in \R^n$. The notation $A\geq B$ means $A-B\geq 0$.
Before proving Proposition~\ref{prop:reduce:to:thetank}, we will use the following:
\begin{Proposition}\label{prop:reduce:to:sub:path}
Let $\gamma' \leq \gamma$ i.e. the path $\gamma$ is obtained from the path $\gamma'$ by adding (possibly zero) closed loops. Then $V(\gamma') \leq V(\gamma)$ and
$\theta(\gamma') \leq \theta(\gamma)$. Moreover if $V(\gamma)$ is primitive and $\gamma\neq \gamma'$ then 
$\theta(\gamma') < \theta(\gamma)$.
\end{Proposition}
\par
\begin{proof}[Proof of Proposition~\ref{prop:reduce:to:sub:path}]
If $V(\gamma')=V_1\cdot V_2\cdot  \dots \cdot V_l \cdot P'$ is the matrix associated to the path 
$\gamma'$, where $V_i$ are the elementary Rauzy--Veech matrices and $P'$ is a permutation matrix, 
then the matrix associated to $\gamma$ has the form 
$$
V(\gamma)=V_1\cdot N_1 \cdot V_2\cdot N_2 \cdot  \dots \cdot V_l\cdot N_l \cdot P,
$$
where:
\begin{itemize}
\item $N_1,\dots ,N_l$ are products of (possibly empty) elementary Rauzy--Veech matrices, hence 
of the type $I+N_i'$, where $I$ is the identity matrix and $N_i'$ is a matrix with nonnegative coefficients.
\item $P$ is the permutation matrix corresponding to the end point of $\gamma$.
\end{itemize}
Since the labeled Rauzy diagram and the reduced Rauzy diagram coincide,  the endpoints of 
$\gamma$ and $\gamma'$ also coincide in the labeled Rauzy diagram, hence $P=P'$. 
From these facts we deduce that $V(\gamma') \leq V(\gamma)$. Let us show that $\theta(\gamma')\leq \theta(\gamma)$.

Recall that $V(\gamma')$ is not necessarily primitive. 
However there is a permutation matrix $P_\sigma$ such that $P_\sigma V(\gamma') P_{\sigma^{-1}} = \begin{pmatrix}
A_1& 0 & \dots  &0 \\
* &A_2 &0& 0 \\
\vdots &  & \ddots & \vdots\\
* & \dots &*&A_s
\end{pmatrix},
$
where the matrices $A_i$ are primitive matrices. Up to a change of basis one can assume that the spectral 
radius of $A_s$ is achieved by $\theta(\gamma')$. Thus there is a non-negative vector $w$ such that
$A_sw=\theta(\gamma')w$. Now $v': = P_{\sigma^{-1}} \begin{pmatrix} 0 \dots 0\ w \end{pmatrix}^T$ 
is a non-negative right eigenvector of $V(\gamma')$ for the eigenvalue $\theta(\gamma')$.

Let $v$ be a positive left eigenvector of $V(\gamma)$: 
$vV(\gamma)=\theta(\gamma) v$ and $v>0$.
From $V(\gamma') \leq V(\gamma)$ one has
$$
vV(\gamma)v' \geq vV(\gamma')v'.
$$
Hence $\theta(\gamma) vv' \geq \theta(\gamma')vv'$ and since $vv'> 0$ we draw 
$\theta(\gamma) \geq \theta(\gamma')$ as desired.

We now prove that last claim: assume $V(\gamma')\leq V(\gamma)$ and $V(\gamma)\neq V(\gamma')$.
Then there exists $k\in\Z_{>0}$ such that $V(\gamma)^k \gg V(\gamma')^k$. In particular we can find $\alpha>1$ such that $V(\gamma)^k \geq \alpha V(\gamma')^k$. Thus $\rho(\gamma)^k\geq \alpha \rho(\gamma')^k$ proving the proposition.
\end{proof}

\begin{proof}[Proof of Proposition~\ref{prop:reduce:to:thetank}]
We prove the first assertion. Let $k$ be any integer satisfying $k>(n-1)/2$. Then $n-1-k<k$, hence $s(\pi)=t^{n-1-k}\pi_n$ is ``before'' $\pi$ in the central loop. Since $\gamma$ is admissible, $\gamma$ is passes thought the central permutation.
Now if $k=(n-1)/2$ then, up to relabeling, $\pi=s(\pi)$.
By using the alphabet $\mathcal{A}=\{1,\dots ,n\}$, we have:
$$
\pi=
\begin{pmatrix}
1&2&\dots &n\\
n&\frac{n-1}{2}&\dots &\frac{n+1}{2}
\end{pmatrix}.
$$
Hence the relabeling does not changes the letter $n$. In particular it must be winner at least once. 
This imply that $\gamma$ contains the step $\pi\to t.\pi$, hence passes through the central permutation. In both
situation $\theta>2$. \medskip

We now turn to the proof of the second part of the proposition. The Rauzy diagram that we consider have the particular property that removing any vertex disconnects it (except the the particular ones that satisfies $t.\pi=\pi$ or $b.\pi=\pi$). This implies that the path $\gamma$ is obtained from the path $\gamma_{n,k}$ by adding (possibly zero) closed loops, namely $\gamma_{n,k} \leq \gamma$. Hence we are in a position to apply Proposition~\ref{prop:reduce:to:sub:path}: this finishes the proof of this part.

The proof of the last statement is similar, once we remark that if $V_{n,k}$ is not primitive, then the path $\gamma$ is obtained from $\gamma_{n,k}$ by adding at least one closed loop. Assuming that $\theta(\gamma)<2$ this gives by definition a path of the form $\gamma_{n,k,l}$ for some $l\in\{1,\dots,2n-2-3k\}$. Thus $\gamma_{n,k,l}\leq \gamma$ and Proposition~\ref{prop:reduce:to:sub:path} again applies. If $l=n-1-k$ then $\gamma_{n,k,l}$ 
is obtained from $\gamma_{n,k}$ by adding twice the same loop. Hence $V_{n,k,l}$ is {\em not} primitive and
$\gamma\neq \gamma_{n,k,l}$. The same argument above applies and $\gamma$ must contain 
$\gamma_{n,k,l}$ for some $l\neq n-1-k$. We conclude with Proposition~\ref{prop:reduce:to:sub:path}.
\end{proof}

\appendix

\section{Matrix computations}
\label{appendix:matrix}

The aim of this section is the computation of the different Rauzy--Veech matrices in $M_n(\Z)$ and especially their characteristic polynomials. The rome technique (see below) reduces these computations to matrices in $M_2(\Z[X])$ and $M_3(\Z[X])$ making them possible. \medskip

In the sequel we will denote by $P_{n,k}$ (respectively, $P_{n,k,l}$) the characteristic polynomial of the matrices $V(\gamma_{n,k})$ (respectively, $V(\gamma_{n,k,l})$) \emph{multiplied} by $(X+1)$. Their  maximal real root are $\theta_{n,k}$ and $\theta_{n,k,l}$, respectively.

\subsection{The rome technique (after~\cite{BGMY})}
\label{sec:rome}
To compute the characteristic polynomial of matrices we will use the rome method, developed in~\cite{BGMY}. 
To this end it is helpful to represent a matrix $V$ into the form of a combinatorial graph which amounts to draw all
paths of length $1$ associated to $V$.

Given a $n\times n$ matrix $V = (v_{ij})$, a path $\eta=(\eta_i)_{i=0}^{l}$
of width $w(\eta)$ and length $l$ is a sequence of elements of $\{1,2,\dots,n\}$ such that $w(\eta)=\prod_{j=1}^{l} v_{\eta_{j-1}\eta_j} \neq 0$.
If $\eta_{l} = \eta_0$ we say that $\eta$ is a loop.

A subset $R \subset \{1,2,\dots,k\}$ is called a ``rome'' if there is no loop outside $R$. 
Given $r_i,r_j\in R$, a path from $r_i$ to $r_j$ is a ``first return path'' is it does not
intersect $R$, except at its starting and ending points. This allows us to define an $r\times r$
matrix-valued real function $V_R(X)$, where $r$ is the size of $R$, by setting $V_R(X) = (a_{ij}(X))$, where $a_{ij}(X) = \sum_\eta w(\eta)\cdot
X^{-l(\eta)}$, where the summation is over all first return paths beginning at $r_i$ and ending at $r_j$.
\begin{NoNumberTheorem}[Theorem~1.7 of \cite{BGMY}]
If $R$ is a rome of cardinality $r$ of a $n\times n$ matrix $V$ then
the characteristic polynomial $\chi_V(X)$ of $V$ is equal to
$$
 (-1)^{n-r} X^n \det(V_R(X) - \mathrm{Id}_r).
$$
\end{NoNumberTheorem}
\begin{Remark}
The matrices $V=V_{n,k}$ or $V=V_{n,k,l}$ can be seen as the action of homeomorphisms on absolute homology if $n$ is even and relative homology otherwise. Thus their characteristic polynomials are reciprocal polynomials (resp. anti-reciprocal polynomials) {\em i.e.} 
$\chi_V(X)=X^n\chi_V(X^{-1})$ (resp. $\chi_V(X)=-X^n\chi_V(X^{-1})$). Thus
$$
\chi_{V} = (-1)^{r}\det(V_R(X^{-1}) - \mathrm{Id}_r)
$$
\end{Remark}

\subsection{The paths $\gamma_{n,k}$}

We briefly recall Notation~\ref{notation:paths} (see also Figure~\ref{fig:rauzy}).\\
For any $k=1,\dots,K_n$ and any $l=1,\dots, n-2-k$ we define the path
$$
\begin{array}{lll}
\gamma_{n,k} :& \pi:=\pi_n.t^{k} \longrightarrow s(\pi):&  b^{n-1-k}t^{n-1-2k}
\end{array}
$$

\begin{Lemma}
\label{lm:reduce:gamma:nk}
Let $n\geq 4$ and $1\leq k \leq K_n$. Set $d=\gcd(n-1,k)$. We denote by $n'=\frac{n-1}{d}+1$ and $k'=\frac{k}{d}$. 
Then the matrix $V_{n',k'}$ is primitive and $\theta_{n,k} = \theta_{n',k'}$.
\end{Lemma}

\begin{proof}[Proof of Lemma~\ref{lm:reduce:gamma:nk}]
We first assume that $k$ and $n-1$ are relatively prime. 
We compute the matrix $V_{n,k}$ associated to the path $\gamma_{n,k}$. For the sequel, 
in order to be able to compare the top eigenvalues of the matrices $V_{n,k}$ we will compute them with a labelling depending on $k$, with alphabet $\mathcal{A}=\{1,\dots ,n\}$. To do this, we start from the central permutation, with the following  labelling
$$\pi_n= \begin{pmatrix}
\alpha_1 & \alpha_2& \dots &\alpha_{n-1}& n \\ n & \alpha_{n-1} & \dots  & \alpha_2& \alpha_1
 \end{pmatrix}$$
where  we have for each $i\leq n-1$ $\alpha_i k=i-1 \mod (n-1)$ and $\alpha_i\in \{1,\dots ,n-1\}$. This is well defined since $k$ and $n-1$ are relatively prime. In particular $1=\alpha_{k+1}$, $2=\alpha_{2k+1}$, $\alpha_1=n-1$, and more generally, $\alpha_{i+k}=\alpha_i+1 \mod n-1$.
The starting point of $\gamma_{k,n}$ is 
$$ \begin{pmatrix}
\alpha_1 & \alpha_2 & \dots &\dots &\alpha_{n-1}&n \\ n &\dots &\alpha_1&\alpha_{n-1}& \dots &\alpha_{k+1}
 \end{pmatrix}$$
 The path $\gamma_{n,k}$ consists of the Rauzy moves $b^{n-1-k} t^{n-1-2k}$, hence we have the following sequence of winners/losers:
 \begin{itemize}
\item $1=\alpha_{k+1}$ is successively winner against $\alpha_{k+2},\dots ,\alpha_{n-1},n$
\item then, $n$ is successively winner against $\alpha_{k+1},\dots ,\alpha_{n-1-k}$. Note that $k+1<n-1-k$.
\end{itemize}
also, the first line of the labeled permutation $s(\pi')$, where $\pi'$ is the endpoint of $\gamma_{n,k}$ is
$$(\alpha_1',\dots ,\alpha_{n-1}',n)=(\alpha_{n-k},\alpha_{n-k+1},\dots ,\alpha_{n-1},\alpha_0,\dots ,\alpha_{n-1-k},n)$$
\emph{i.e.} $\alpha'_i=\alpha_{i-k}$ or $\alpha'_{i}=\alpha'_{i+n-1-k}$ depending which of $i-k$ or $i+n-1-k$ is in $\{1,\dots ,n-1\}$. In any case, we obtain, $\alpha'_i=\alpha_{i}-1$. Hence the matrix $V_{n,k}$ is obtained from the product of elementary Rauzy--Veech matrices by translating cyclically the first $n-1$ columns by 1 on the right. Finally, we have
$$
V_{n,k}=
\begin{pmatrix}
a_{n-1}&2& a_2  &\dots & \dots &a_{n-2}&1 \\
0 &0 &1&0& \dots  &\dots & 0 \\
\vdots &  & \ddots & \ddots & \ddots && \vdots \\
\vdots & &  &\ddots & \ddots & \ddots& \vdots\\
0 & \dots & \dots &\dots &0  &1& 0\\
1 &0 &&\dots && 0& 0 \\
b_{n-1}&1&b_2&\dots &\dots &b_{n-2}&1
\end{pmatrix}
$$
Where for $i\in \{2,\dots ,n-2\}$, we have:
$$a_i=
\left\{ \begin{array}{ccc}
2 & \textrm{ if }& i\in \{\alpha_{k+2},\dots ,\alpha_{n-1-k}\}\\
1 & \textrm{ if }& i\in \{\alpha_{n-k},\dots ,\alpha_{n-1}\}\\
0 & \textrm{ if }& i\in \{\alpha_{1} \dots \alpha_{k}\}
\end{array} \right.
$$
and $b_i=1$ if and only if $a_i=2$, and $b_i=0$ otherwise.
Note that $a_{n-1}=0$, hence $b_{n-1}=0$. The matrix $V_{n,k}$ is clearly irreducible
(see Figure~\ref{fig:graph:vnk}) and thus primitive since there is a non zero diagonal element. \medskip

Now if $d=\gcd(n-1,k)>2$, we define $\alpha$ in the following way:
\begin{itemize}
\item $\alpha_{1}=n-1$, and $\alpha_{[1+ik]}=i$ for $i\leq \frac{n-1}{d}$, where $[1+ik]$ is the representative modulo $n-1$ of $1+ik$, which is in $\{1,\dots ,n-1\}$.
\item The other $\alpha_i$ are chosen in any way.
\end{itemize}
The matrix $V_{n,k}$, with this labelling, is 
$$
\left(\begin{array}{c|c}
V_{\frac{n-1}{d}+1,\frac{k}{d}} & *\\
\hline \\
0 & *
\end{array} \right)
$$
where the bottom right corner is a permutation matrix.
This ends the proof of the lemma.
\end{proof}

\begin{Lemma}
\label{lm:a}
Let $n\geq 4$ and $1\leq k \leq K_n$. If $\gcd(n-1,k)=1$ then 
$$
P_{n,k}=X^{n+1}-2X^{n-1}-2 \sum_{j\in J_{n,k}} X^j -2X^2+1
$$
where $J_{n,k}=\left\{3,\dots ,n-2\right\}\backslash \left\{\ceil{\frac{i(n-1)}{k}} +\varepsilon\ |\ \ i=1,\dots ,k-1 \ \mathrm{ and }\ \varepsilon= 0,1\right\}$. In particular:
\begin{enumerate}
\item For even $n$
$$
P_{n,K_n} =X^{n+1} - 2X^{n-1} - 2X^2 + 1.
$$
and if $n\not \equiv 4 \mod 6$, $\gcd(n-1,K_n-1)=1$ and
$$
P_{n,K_n-1} = X^{n+1} - 2X^{n-1} - 2X^{\ceil{2n/3}}- 2X^{\floor{n/3}+1} - 2X^2 + 1.
$$
\item For $n=1 \mod 4$, we have $\gcd(n-1,K_n)=1$ and 
$$P_{n,K_n}=X^{n+1} - 2X^{n-1} -2X^{\frac{n+1}{2}}- 2X^2 + 1$$
\item For $n=3 \mod 4$, we have $\gcd(n-1,K_n-1)=1$ and 
$$P_{n,K_n}=X^{n+1}-2X^{n-1}-2X^{n-2\floor{\frac{n-5}{8}}-2}-2X^{\frac{n+1}{2}}-2X^{2\floor{\frac{n-5}{8}}+3}-2X^2+1$$
\item More generally, for $n$ even, and $k\leq K_n-1$, the highest nonzero monomial (except $X^{n+1}$ and $-2X^{n-2}$) has degree at least $\ceil{\frac{2n}{3}}$.
\end{enumerate}

 \end{Lemma}
\begin{proof}[Proof of Lemma~\ref{lm:a}]
We will use the rome method as explained in Section~\ref{sec:rome}. We use the notation
of the proof of Lemma~\ref{lm:reduce:gamma:nk}. We will write $V_{n,k}$ as $V_{n,k}=A_{n}-B_{n,k}$ where
$$
A_{n}=
\left(\begin{smallmatrix}
0&2& 2  &\dots & \dots 2&1&1 \\
\vdots &\ddots &1&0& \dots  &\dots & 0 \\
\vdots &  & \ddots & \ddots & \ddots && \vdots \\
\vdots & &  &\ddots & \ddots & \ddots& \vdots\\
0 & \dots & \dots &\dots &0  &1& 0\\
1 &0 &&\dots && & 0 \\
0&1&1&\dots &\dots 1&0&1
\end{smallmatrix}\right),
$$
and the only non zero entries of $B_{n,k}$ are as follows:
$$
\textrm{For any} i=1,\dots,k-1, \textrm{ set } l:=\floor{\frac{i(n-1)}{k}}+1. \textrm{ Then }
\left\{\begin{array}{l}
b_{1,l} = 1 \\
b_{1,l+1} = 2 \\
b_{n,l} = b_{n,l+1} = 1
\end{array}\right.
$$
Observe that $1\leq k\leq n/2-1$, hence for $i\in \{1,\dots ,k-1\}$ 
$$ 2 < 2i\frac{(n-1)}{n-2}\leq \frac{i(n-1)}{k}\leq  \frac{(k-1)(n-1)}{k}=n-1-\frac{n-1}{k}<n-3.$$
 In particular, all integers of the form $\ceil{i(n-1)/k}+\varepsilon$ for $i\in \{1,\dots ,k-1\}$ and  $\varepsilon\in \{0,1\}$ are mutually disjoint and in $\{3,\dots ,n-2\}$.
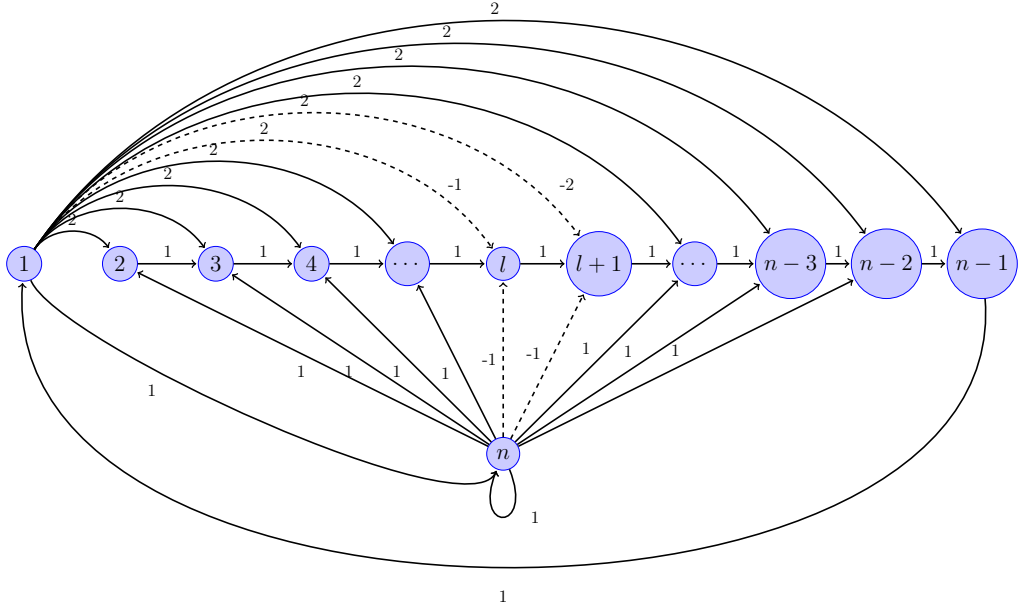
\begin{figure}[htbp]
\scalebox{0.6}{
\begin{tikzpicture}[scale=0.7,shorten >=1pt, auto, node distance=3cm, 
   edge_style/.style={line width=1pt,draw=black,->}]
   \foreach [count=\i] \x/\y/\t in {0/0/$1$,3/0/$2$,6/0/$3$,9/0/$4$,12/0/$\dots$,15/0/$l$,18/0/$l+1$,21/0/$\dots$,24/0/$n-3$,27/0/$n-2$,30/0/$n-1$,	15/-6/$n$}
     \node [circle,draw=blue,fill=blue!20!,font=\sffamily\Large\bfseries]
        (v\i) at (\x,\y) {\t};

   \foreach \i/\j/\t in {2/3/1,3/4/1,4/5/1,5/6/1,6/7/1,7/8/1,8/9/1,9/10/1,10/11/1,12/2/1,
   12/3/1,12/4/1,12/5/1,12/8/1,12/9/1,12/10/1}
    \draw [edge_style]  (v\i) edge node {\t} (v\j); 
    
   \foreach \i/\j/\t in {1/12/1} \draw [line width=1pt,draw=black,->]  (v\i) ..  controls  +(1,-2.5) and +(-1,-2.5) .. (v\j);     \draw  (4,-4) node  {1};
   
 \foreach \i/\j/\t in {12/6/-1,12/7/-1}     \draw [line width=1pt,dashed,draw=black,->]  (v\i) edge node {\t} (v\j); 
 
    \foreach \i/\j/\t in {12/12/t} \draw [edge_style]   (v\i) ..  controls  +(1,-2.5) and +(-1,-2.5) .. (v\j); 
    \draw  (16,-8) node  {1};

    \foreach \i/\j/\t in {11/1/t} \draw [edge_style]  (v\i) ..  controls  +(1,-12.5) and +(-1,-12.5) .. (v\j); 
    \draw  (15,-10.5) node  {1};

    \foreach \i/\j in {1/2,1/3,1/4,1/5,1/8,1/9,1/10,1/11}   \draw[edge_style] (v\i) to[bend left=55] node {2}(v\j);


    \foreach \i/\j in {1/6,1/7} \draw [dashed,edge_style]  (v\i) to[bend left=55] node {2}(v\j);
    \draw  (13.5,2.5) node  {-1};  \draw  (17,2.5) node  {-2};
\end{tikzpicture}
}
\caption{
The graph associated to $V_{n,k}$. In dashed line we have represented arrow that need to be removed from the graph in blue.
Multiplicity are also indicated. To be more more precise, there is one arrow
from vertex labelled 1 to the vertex labelled l and no arrow from vertex labelled 1 to the vertex labelled l+1.
In the graph $l=\floor{\frac{i(n-1)}{k}}+1$ for any $i=1,\dots,k-1$.
Obviously the graph associated to $A_n$ is drawn in blue colour}
\label{fig:graph:vnk}
\end{figure}
\par
Clearly the set $R=\{1,n\}$ is a rome for $A_n$. Thus it is also a rome for $V_{n,k}$ (since 
we pass from $A_n$ to $V_{n,k}$ by removing some paths). The $2\times 2$ matrix
$(A_n)_R$ is easily obtained as
$$
(A_n)_R = \left( \begin{smallmatrix} X^2+2S_n  & X \\
S_n & X
\end{smallmatrix}\right) \textrm{ where } S_n=\sum_{i=3}^{n-1}X^i.
$$
To obtain the matrix $(V_{n,k})_R$ one has to subtract the polynomial corresponding
to the paths passing through arrows in dashed line (passing through vertices $l$ and $l+1$ where $l=\floor{\frac{i(n-1)}{k}}+1$ for some $i=1,\dots,k-1$).
At this aim we define
$$
\left\{\begin{array}{l}
T_{n,k} = \sum_{i=1}^{k-1}X^{n-1-(\floor{\frac{i(n-1)}{k}}+1)} \\
Q_{n,k} = \sum_{i=1}^{k-1}X^{n-1-\floor{\frac{i(n-1)}{k}}} =  \sum_{i=1}^{k-1} X^{\ceil{i(n-1)/k}}.
\end{array}\right.
$$
The polynomial $T_{n,k}$ (respectively, $2Q_{n,k}$) takes into account all simple paths from the vertex $1$ to the vertex $1$
passing through an arrows in dashed line connecting $1$ to $l$ (respectively, to $l+1$). Hence
$$
(V_{n,k})_R = \left( \begin{smallmatrix} X^2+2S_{n} -T_{n,k}-2Q_{n,k} & X \\
S_{n} - T_{n,k} - Q_{n,k} & X
\end{smallmatrix}\right)
$$
Note that $T_{n,k}=XQ_{n,k}$.
By using~\cite{BGMY}, a straightforward computation gives:
$$
\chi_{V_{n,k}} = X^3-X^2-X+1+S_{n}\cdot (X-2) + 2Q_{n,k} 
$$
Hence, 
$$P_{n,k}=X^{n+1}-2\sum_{i=2}^{n-1} X^i +1+2\sum_{i=1}^{k-1}\left( X^{\ceil{i(n-1)/k}}+ X^{\ceil{i(n-1)/k}+1}\right)$$
which implies the first statement of the lemma.

We give a little more information on the set $J_{n,k}$. We write $J_{n,k}=\{j_1,\dots ,j_r\}$ with $j_1<\dots < j_r$ (possibly, $r=0$). We have $\ceil{\frac{i(n-1)}{k}}=2i+\ceil{\frac{i(n-1-2k)}{k}}$. In particular, we see that $j_s=2i_s+s$, where $i_s\geq 0$ is the smallest integer such that $\frac{i_s(n-1-2k)}{k}>s$, \emph{i.e} $i_s=\floor{\frac{sk}{n-1-2k}}+1$. Note also that $P_{n,k}$ must be reciprocal, hence $j\in J_{n,k}$ if and only if $n+1-j\in J_{n,k}$.

Now we compute particular cases.
\begin{enumerate}
\item If $n$ is even, $K_n=n/2-1$. We have  $n-1-2K_n=1$, $2i_1+1=n-1>n-2$, hence $J_{n,K_n}=\emptyset$. For $k=K_n-1=n/2-2$ and $n$ even, $n\neq 4 \mod 6$ we have
\begin{eqnarray*}
2i_1+1&=&2\floor{(n-4)/6}+3=\floor{n/3}+1=j_1 \\
 2i_2+2&=&2\floor{(n-4)/3}+4=\ceil{2n/3}=j_2=n+1-j_1\\
\end{eqnarray*}
Hence  $J_{n,K_n-1}=\{\floor{n/3}+1,\ceil{2n/3}\}$.
\item If $n=1 \mod 4$, $K_n=(n-1)/2-1=\frac{n-3}{2}$, we have $n-1-2K_n=2$, hence:
\begin{eqnarray*}
2i_1+1&=&2\floor{(n-3)/4}+3=\frac{n+1}{2}=j_1 \\
\end{eqnarray*}
So, $J_{n,K_n}=\{\frac{n+1}{2}\}$.

\item If $n=3 \mod 4$, $k=K_n-1=\frac{n-5}{2}$, $n-1-2k=4$, hence
\begin{eqnarray*}
2i_1+1&=&2\floor{(n-5)/8}+3=j_1 \\
 2i_2+2&=&2\floor{(n-5)/4}+4=\frac{n+1}{2}=j_2\\
 2i_3+3&=&2\floor{3(n-5)/8}+5=n+1-j_1
\end{eqnarray*}
So, $J_{n,K_n}=\{2\floor{(n-5)/8}+3, \frac{n+1}{2},n-2\floor{(n-5)/8}-2\}$
\item More generally, for $n$ even and $k\leq K_n-1$ we evaluate $j_1$.
$$2i_1+1=2\floor{\frac{n-1-2k}{k}}+3\leq 2\floor{\frac{n-1-2(K_n-1)}{K_n-1}}+3=\floor{n/3}+1$$
since $k\mapsto \frac{n-1-2k}{k}$ is increasing, and we conclude by using that the polynomial is reciprocal.
\end{enumerate}
This finishes the proof.
\end{proof}

\subsection{The paths $\gamma_{n,K_n,l}$}
We recall Notation~\ref{notation:paths} (see also Figure~\ref{fig:rauzy}) when $k=K_n$.
For any $l\in \{1,\dots ,2n-2-3K_n\}$ we define the path
$$
\gamma_{n,k,l} : \pi \longrightarrow s(\pi): \left\{
\begin{array}{lll}
b^{l}t^{L_n+1-l}b^{L_n+1-l}t^{n-1-2K_n} &\mathrm{if}& 1\leq l \leq L_n \\
b^{L_n+1}t^{l-(L_n+1)}b^{2(L_n+1)-l}t^{2n-2-3K_n-l} &\mathrm{if}& L_n < l \leq 2n-2-3K_n
\end{array}
\right.
$$
Note that when $n$ is even $2n-2-3K_n=L_n+2$ otherwise $2n-2-3K_n=L_n+3$.
\begin{Proposition}
\label{prop:n:even:nKl}
Let $n\geq 4$ be an even integer.
\begin{enumerate}
\item If $l\in \{1,\dots,L_n\}$ then
$$
P_{n,K_n,l} = P_{n,K_n} - \frac{X^{n-2l + 2}+X^{n-2l + 4}+2X^{n-1} -X^{2l + 1}-X^{2l-1}- 2X^{4}}{(X + 1)(X - 1)}
$$
\item For $l=L_n$ one has
$$
P_{n,K_n,L_n} = X^{n+1} - 2X^{n-1} - X^{n-3} - X^4 - 2X^2 + 1
$$
\item For $l=L_n+2$ $V_{n,K_n,L_n+2} = V_{n,K_n} + B_n$ where $B_n=(b_{i,j})$ and 
$$
\left\{\begin{array}{l}
b_{1,2j} = b_{n-2,2j} = 1, \textrm{ for } 1 \leq  j \leq K_n-1, \\
b_{1,n} = b_{n-2,n} = 1,\\
b_{i,j} = 0 \textrm{ otherwise.}
\end{array}\right.
$$
\end{enumerate}
\end{Proposition}
\begin{proof}
Again we appeal to the rome method in oder to compute the characteristic polynomial.
Since $n$ is even, the matrix $V_{n,K_n}$ is primitive. 
Thus Lemma~\ref{lm:reduce:gamma:nk} applies and we have a nice expression for 
the matrix $V_{n,K_n}$ (see the proof of Lemma~\ref{lm:reduce:gamma:nk}).
Let $l\in \{1,\dots,L_n\}$. 
By construction, $\gamma_{n,K_n,l}$ is a path obtained from $\gamma_{n,K_n}$ by
adding a closed loop (of type 't' and of length $n-1-K_n-l$) at $b^lt^{K_n}\pi_n$. Hence
the matrices $V_{n,K_n}$ and $V_{n,K_n,l}$ differ by a non negative matrix, namely
$V_{n,K_n,l} - V_{n,K_n} = C_{n,l}=(c_{i,j})$ where
$$
\left\{\begin{array}{l}
c_{2l,2} = c_{2l,2l+1} = 1, \\
c_{2l,2j+3} = 2 \textrm{ for } l \leq  j \leq K_n-1, \\
c_{i,j} = 0 \textrm{ otherwise.} \\
\end{array}\right.
$$
On the way, the same argument applied to the matrix $V_{n,K_n,L_n+2}$ applies, proving
the last statement.

We are now in a position to compute $P_{n,K_n,l}$. The graph associated to the matrix $V_{n,K_n}$ is
already presented in the proof of Lemma~\ref{lm:a}. For readability purposes, we reproduce this
graph below, where we have added the edges corresponding to the graph associated to $C_{n,l}$.

Clearly the set $R=\{1,2l,n\}$ is a rome for $V_{n,K_n,l}$, and thus for $V_{n,K_n}$ (since 
we pass from $V_{n,K_n,l}$ to $V_{n,K_n}$ by removing some paths). The $3\times 3$ matrix
$(V_{n,K_n})_R$ is easily obtained as
$$
(V_{n,K_n})_R = \left( \begin{matrix}
    \sum \limits_{\underset{i \ \mathrm{even}}{i=2}}^{n-2l}X^i  & 2X^{2l-1} +  \sum \limits_{\underset{i \ \mathrm{even}}{i=2}}^{2l-2}X^i & X\\
 \sum \limits_{\underset{i \ \mathrm{even}}{i=2}}^{n-2l}X^i & X^{2l-1} & 0 \\
0 & X^{2l-1} & X
\end{matrix}\right)
%
$$
where $\sum \limits_{\underset{i \ \mathrm{even}}{i=2}}^{2l-2}X^i=0$ if $l=1$.

Adding the matrix $C_{n,l}$ consists of adding two arrows form the vertex labelled by $2l$ 
to the vertices $2,2l+1$ (with multiplicity $1$) and $K_n-l$ arrows to the vertices $2l+3,\dots,n-3,n-1$
(with multiplicity $2$). In this situation $R$ is still a rome. To compute the matrix $(V_{n,K_n,l})_R$ we need to consider all paths passing through a dashed edge in the graph in Figure~\ref{fig:graph:vnkl}.
Thus $(V_{n,K_n,l})_R(X) = (V_{n,K_n})_R(X) + C(X)$ where
$$
C(X) =
\left(\begin{matrix}
0&0&0 \\
2 \sum \limits_{\underset{i \ \mathrm{even}}{i=2}}^{n-2l-2}X^i + X^{n-2l} & X^{2l-1} & 0 \\
0  & 0 & 0
\end{matrix} \right)
$$
Hence by~\cite{BGMY} we draw $\chi_{V_{n,K_n,l}} = -\det((V_{n,K_n})_R(X) + C(X) - \mathrm{Id}_3)$. By multi-linearity, one has:
$$
P_{n,K_n,l} = P_{n,K_n} - (1+X)\cdot \det(W_{n,l})
$$ 
where
$$
W_{n,l} = \left( \begin{matrix}
    \sum \limits_{\underset{i \ \mathrm{even}}{i=2}}^{n-2l}X^i -1 & 2X^{2l-1} +  \sum \limits_{\underset{i \ \mathrm{even}}{i=2}}^{2l-2}X^i & X\\
2 \sum \limits_{\underset{i \ \mathrm{even}}{i=2}}^{n-2l-2}X^i + X^{n-2l} & X^{2l-1} & 0 \\
0 & X^{2l-1} & X-1
\end{matrix}\right)
%
%
$$
A direct computation gives
$$
\det(W_{n,l}) = \frac{X^{n-2l + 2}+X^{n-2l + 4}+2X^{n-1} -X^{2l + 1}-X^{2l-1}- 2X^{4}}{(X + 1)^2(X - 1)}
$$
that is the desired result.

The second assertion comes for free since $2L_n=n-2$ and by Lemma~\ref{lm:a} one has $P_{n,K_n}=X^{n+1} - 2X^{n-1} - 2X^2 + 1$. Proposition~\ref{prop:n:even:nKl} is proved.
\end{proof}

\begin{Proposition}
\label{prop:n:odd:nKl}
Let $n\geq 4$ be an integer with $n\equiv 3 \mod 4$. Fix $l\in \{1,\dots,L_n+3\}$:
\begin{enumerate}
\item If $l$ is even then $V_{n,K_n,l}$ is reducible and $\theta_{n,K_n,l} = \theta_{n',K_{n'},l'}$ 
with $n'=(n+1)/2$, $l'=l/2$ and $K_{n'}=K_n/2$.
\item If $l\leq L_n$ is odd then $V_{n,K_n,l}$ is primitive and
$$
P_{n,K_n,l} = \frac{S_n(X)+2\left( X^{\frac{n+7}{2}-l} -X^l + X^{l + \frac{n-1}{2}} -  X^{n+3-l} \right) }{(X-1) (X+1)}
$$
where $S_n=1  - 3 X^2 - 2 X^{\frac{n-1}{2}} + 8 X^{\frac{n+3}{2}}   - 2 X^{\frac{n+7}{2}}   -   3 X^{n+1}+ X^{n+3}$.
\item If $l=L_n+2$ then  $V_{n,K_n,l} =
\left(\begin{array}{c|c|c}
0_{K_n\times K_n} &  \mathrm{Id}_{K_n\times K_n} & 0_{K_n \times 3} \\
\hline 
2 \cdots 2 &  0 \cdots 0 & 2\ 3\ 2 \\
\hline 
1 \cdots 1 &  0 \cdots 0 & 0\ 2\ 1 \\
\hline 
\mathrm{Id}_{K_n\times K_n} &  0_{K_n\times K_n}& 0_{K_n \times 3} \\
\hline 
0 \cdots 0 &  0 \cdots 0 & 1\ 1\ 1 \\
\end{array} \right)
$
\end{enumerate}
\end{Proposition}
\begin{proof}[Proof of Proposition~\ref{prop:n:odd:nKl}]
We follow the strategy of the proof of the previous proposition.
The first point is clear. In the sequel, let $m=K_n+1=(n-1)/2$.
Again we appeal to the rome method in oder to compute the characteristic polynomial.
We have $V_{n,K_n,l} = A_n + B_{n,l}$ where
$$
A_n =
\left(\begin{array}{c|c|c}
0_{K_n\times K_n} &  \mathrm{Id}_{K_n\times K_n} & 0_{K_n \times 3} \\
\hline 
1 \cdots 1 &  0 \cdots 0 & 2\ 2\ 1 \\
\hline 
0 \cdots 0 &  0 \cdots 0 & 0\ 1\ 0 \\
\hline 
\mathrm{Id}_{K_n\times K_n} &  0_{K_n\times K_n}& 0_{K_n \times 3} \\
\hline 
0 \cdots 0 &  0 \cdots 0 & 1\ 1\ 1 \\
\end{array} \right)
$$
and the only non zero entries of $B_{n,l}=(b_{i,j})$ are
$$
\left\{\begin{array}{l}
b_{n-l,i} = 2 \textrm{ for } i=1,\dots,m-l \\
b_{n-l,l} = 1 \\
b_{n-l,n-1} = 2\\
b_{n-l,n-2} = 1
\end{array}
\right.
$$
It is helpful to represent the matrices in form of a combinatorial graph which amounts to draw 
all paths. 
\par
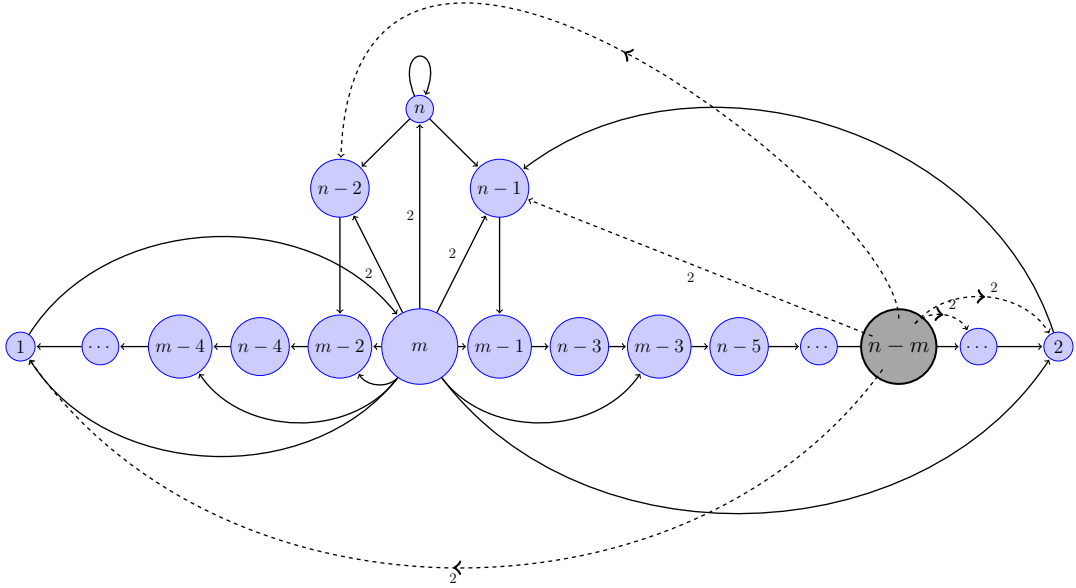
\begin{figure}[htbp]
\scalebox{0.5}{
\begin{tikzpicture}[scale=0.7,shorten >=1pt, auto, node distance=3cm, 
   edge_style/.style={line width=1pt,draw=black,->}]
   \foreach [count=\i] \x/\y/\t in {0/0/$1$,3/0/$\dots$,6/0/$m-4$,9/0/$n-4$,12/0/$m-2$,15/0/$m$,18/0/$m-1$,21/0/$n-3$,24/0/$m-3$,27/0/$n-5$,30/0/$\dots$,33/0/$n-l$,36/0/$\dots$,39/0/$2$,12/6/$n-2$,18/6/$n-1$,15/9/$n$}
\ifthenelse{\i=6}   
     {\node [circle,minimum size=2cm,draw=blue,fill=blue!20!,font=\sffamily\Large\bfseries]
        (v\i) at (\x,\y) {\t}}
     {\node [circle,draw=blue,fill=blue!20!,font=\sffamily\Large\bfseries]
        (v\i) at (\x,\y) {\t}};

   \foreach \i/\j/\t in {6/5/,5/4/,4/3/,3/2/,2/1/,6/7/,7/8/,8/9/,9/10/,10/11/,11/12/,
   12/13/,13/14/,6/15/2,6/16/2,6/17/2,17/16/,17/15/,15/5/,16/7/}
    \draw [edge_style]  (v\i) edge node {\t} (v\j); 

\node [circle,draw=black,fill=gray!70!,font=\sffamily\LARGE\bfseries,line width=1.5] at (33,0) {$n-m$};

    \foreach \i/\j/\t in {17/17/t} \draw [edge_style]   (v\i) ..  controls  +(-1,2.5) and +(1,2.5) .. (v\j); 

    \foreach \i/\j in {6/5,6/3,6/1}   \draw[edge_style] (v\i) to[bend left=55](v\j);
    
    \foreach \i/\j in {1/6}   \draw[edge_style] (v\i) to[bend left=55](v\j);

\foreach \i/\j in {6/9,6/14,14/16}    \draw[edge_style] (v\i) to[bend right=55](v\j);
           
  \foreach \i/\j/\t in {12/16/2}     \draw [line width=1pt,dashed,draw=black,->]  (v\i) edge node {\t} (v\j); 
  
    \foreach \i/\j in {12/15}   \draw [dashed,edge_style,fleche=0.5:black]  (v\i) ..  controls  +(0,5.5) and +(0.5,15.5) .. (v\j);    
    
        \foreach \i/\j in {12/1,12/13,12/14}   \draw[dashed,edge_style,fleche=0.5:black] (v\i) to[bend left=55] node {2} (v\j);
\end{tikzpicture}
}
\caption{The graph associated to $V_{n,K_n,l}$. In
dashed line we have represented arrow coming from the matrix $B_{n,l}$.
The multiplicity is indicated only when it is $2$, otherwise it is $1$.}
\label{fig:graph:vnkl}
\end{figure}
\par
The graph associated to $A_n$ is rather simple. Clearly a rome is made of the subsets
labelled $R=\{n,n-1,m\}$. The matrix $A_R(X)$ in this label is
\begin{displaymath}
A_R(X) =
\left(\begin{array}{ccc}
X & X & X^m \\
0 & X^m & 0 \\
X  & R & S
\end{array} \right)
\end{displaymath}
$$
\begin{array}{lll}
R =2X + \sum_{i=3, i \textrm{ odd}}^m X^i &=& 2X + X^3\cdot \frac{1-X^{m-1}}{1-X^2}\\
S = 2X^m + \sum_{i=2, i \textrm{ even}}^{m-1} X^i&=&2X^m + X^2\cdot \frac{1-X^{m-1}}{1-X^2}
\end{array}
$$
Adding the matrix $B_{n,l}$ consists of adding arrows form the vertex labelled by $n-l$ 
to the vertices $1,\dots,m-l,n-2$ with multiplicity $1$ and to the vertex $n-1$ with multiplicity $2$.
In this situation $R$ is still a rome. To compute the matrix $(V_{n,K_n,l})_R$ we need to consider all
paths passing through a dashed  edge on the graph in Figure~\ref{fig:graph:vnkl}.
Thus $(V_{n,K_n,l})_R(X) = A_R(X) + C(X)$ where
$$
C(X) =
\left(\begin{array}{ccc}
0&0&0 \\
0 & P & Q \\
0  & \sum_{i=0}^{(l-1)/2-1}\frac{P}{X^{2i}} & \sum_{i=0}^{(l-1)/2-1}\frac{Q}{X^{2i}}
\end{array} \right)
$$
and 
$$
P = X^{m} + 2\sum_{i=l, i \textrm{ i odd}}^{m-2}X^{i} = X^{m} + 2 X^l\cdot \frac{1-X^{m-l}}{1-X^2},
$$
$$
Q = X^{m+l-1} + 2\sum_{i=l+1, i \textrm{ i even}}^{m-1}X^{i} = X^{m+l-1} + 2 X^{l+1}\cdot \frac{1-X^{m-l}}{1-X^2}.
$$
Hence by~\cite{BGMY} we draw $\chi_{V_{n,K_n,l}} = -\det(A_R(X) + C(X) - \mathrm{Id}_3)$. By using
the fact that $C(X) =\left(\begin{smallmatrix}
0&0&0 \\
0 & P & Q \\
0  & \frac{1-X^{1-l}}{1-X^{-2}} P & \frac{1-X^{1-l}}{1-X^{-2}} Q
\end{smallmatrix} \right)
$ and $Q= X^m(X^{l-1}-X)+XP$ we easily obtained the desired equality.

The last assertion is also easy to derive from the fact that $\gamma_{n,K_n,L_n+2}$ is obtained
from $\gamma_{n,K_n}$ by adding a 'b' loop at $t^{K_n+1}\pi_n$ and from the shape of $A_n$.
This ends the proof of Proposition~\ref{prop:n:odd:nKl}.
\end{proof}


\section{Comparing roots of polynomials}
This section is devoted to comparing $\theta_{n,k}$ and $\theta_{n,k,l}$ (the maximal real roots of the polynomials $P_{n,k}$ and $P_{n,k,l}$ introduced in Appendix~\ref{appendix:matrix}) for various $n,k,l$.
Observe that by construction $\theta_{n,k} > \sqrt{2}$ and $\theta_{n,k,l} > \sqrt{2}$.

One key ingredient for comparing maximal real roots of these polynomials is the easy
\begin{Lemma}\label{compare:roots}
Let $P_1,P_2$ be two unitary polynomials of degree at least one such that for  $x>\sqrt{2}$, $P_1(x)-P_2(x)>0$. We assume that $P_1$ has a root $\theta_1>\sqrt{2}$. Then $P_2$ has a root $\theta_2>\theta_1$.
\end{Lemma}
\begin{proof}
The assumption implies $P_2(\theta_1)<0$. Since $P_2$ is unitary, the result follows by the mean value theorem.
\end{proof}

\begin{Lemma}
\label{lm:decreasing}
The followings hold:
\begin{enumerate}
\item The sequence $(\theta_{2n,K_n})_n$ is a decreasing sequence.
\item The sequence $(\theta_{1+4n,K_n})_n$ is a decreasing sequence.
\item The sequence $(\theta_{3+4n,K_n,L_n})_n$ is a decreasing sequence.
\end{enumerate}
\end{Lemma}
\begin{proof}[Proof of Lemma~\ref{lm:decreasing}]
We establish the lemma case by case.

\noindent {\bf Case 1.} By Lemma~\ref{lm:a}, $\theta_{2n,K_{2n}}$ is the largest root of
$$
P_{2n}=X^{2n+1}-2X^{2n-1}-2X^2+1.
$$
Observe that $P_{2n+2}-X^2P_{2n}=2X^4-X^2+1$. Thus $P_{2n+2}(x)-x^2P_{2n}(x)>0$ for $x>\sqrt{2}$ and Lemma~\ref{compare:roots} gives that $\theta_{2n,K_{2n}} < \theta_{2n+2,K_{2n+2}}$. \medskip

\noindent {\bf Case 2.} We observe that $\theta_{1+4n,K_{1+4n}}$  is the largest root of
$$
P_{1+4n,K_{1+4n}}=X^{4n+2} - 2X^{4n} -2X^{2n+1}- 2X^2 + 1.
$$
Again a simple computation establishes $P_{5+4n,K_{5+4n}}-X^4P_{1+4n,K_{1+4n}}=2X^{2n+5}-2X^{2n+3}+2X^6-2X^2-X^4+1$. Hence 
$P_{5+4n,K_{5+4n}}(x)-x^4P_{1+4n,K_{1+4n}}(x) > 0$ for $x>\sqrt{2}$ and by Lemma~\ref{compare:roots}:
$$
\theta_{5+4n,K}<\theta_{1+4n,K}.
$$

\noindent {\bf Case 3.} Similarly 
$$
P_{3+4n,K,L}=X^{4n+4}-2X^{4n+2}-4X^{2n+3}+4X^{2n+1}+2X^2-1.
$$
Hence, for $x>\sqrt{2}$:
$$
P_{7+4n,K,L}-x^4P_{3+4n,K,L}=(x^2-1)(4x^{2n+5}-4x^{2n+3}-2x^4-x^2+1)>4x^{2n+3}-2x^4-x^2+1>0.
$$
Hence Lemma~\ref{compare:roots} applies and $\theta_{7+4n,K,L}<\theta_{3+4n,K,L}$. The lemma is proved.
\end{proof}

\begin{Lemma}
\label{lm:compare:n}
Let $n\geq 4$ and $1\leq k<k' \leq K_n$. If $\gcd(n-1,k)=\gcd(n-1,k')=1$ then
$$
\theta_{n,k'}<\theta_{n,k}.
$$
\end{Lemma}
\begin{proof}[Proof of Lemma~\ref{lm:compare:n}]
From the proof of Lemma~\ref{lm:a}, we have $P_{n,k'}-P_{n,k}=(X+1)\left(Q_{n,k'}-Q_{n,k}\right)$, where 
$Q_{n,k}= \sum_{i=1}^{k-1} X^{\ceil{i(n-1)/k}}$.
From Lemma~\ref{compare:roots}, we need to show that $P_{n,k'}(x)-P_{n,k}(x)>0$ for $x> \sqrt{2}$.
First we observe that
$$
Q_{n,k'}-Q_{n,k}=\sum_{i=k}^{k'-1} X^{\ceil{\frac{i(n-1)}{k'}}}+\sum_{p=1}^{k-1}X^{\ceil{\frac{(k'-p)(n-1)}{k'}}}-X^{\ceil{\frac{(k-p)(n-1)}{k}}}.
$$
Now for any $p\in \{1,\dots ,k-1\}$
$$\ceil{\frac{(k-p)(n-1)}{k}}\leq \ceil{\frac{(k'-p)(n-1)}{k'}}.$$
So, for any $x>1$, $P_{n,k'}(x)-P_{n,k}(x)>0$, proving the lemma.
\end{proof}

Before comparing roots using polynomials, we end this subsection with a simple lemma:

\begin{Lemma}
\label{lm:comparing:matrix}
Let $n\geq 7$ be an integer satisfying $n\equiv3 \mod 4$. Then $\theta_{n,K_n,L_n+2} > 2$. \\
Let $n\geq 4$ be an even integer. Then $\theta_{n,K_n,L_n+1} > 3^{\frac1{2}}$.\\
Let $n\geq 6$ be an even integer. Then $\theta_{n,K_n,L_n+2} > 6^{\frac1{4}}$.
\end{Lemma}

\begin{proof}[Proof of Lemma~\ref{lm:comparing:matrix}]
We will use the following classical inequality for 
the Perron root $\rho(A)$ of a non negative primitive matrix $A=(a_{ij})_{i,j=1,\dots,n}$: 
$\rho(A) > \delta(A)$ where $\delta(A)=\min_{j=1}^n \sum_{i=1}^n a_{ij}$ (see {\em e.g.}~\cite[Proposition 4.2.]{BL12}).

We prove the first assertion. The matrix $V(\gamma_{n,K_n,L_n+2})$ is primitive (by
Proposition~\ref{prop:n:odd:nKl}) and $\theta_{n,K_n,L_n+2}^2$ is the Perron root of 
$V_{n,K_n,L_n+2}^2$. It suffices to show that  $\delta(V_{n,K_n,L_n+2}^2) > 4$. By Proposition~\ref{prop:n:odd:nKl} one has
$$
V_{n,K_n,L_{n+2}} =
\left(\begin{array}{c|c|c}
0_{K_n\times K_n} &  \mathrm{Id}_{K_n\times K_n} & 0_{K_n \times 3} \\
\hline 
2 \cdots 2 &  0 \cdots 0 & 2\ 3\ 2 \\
\hline 
1 \cdots 1 &  0 \cdots 0 & 0\ 2\ 1 \\
\hline 
\mathrm{Id}_{K_n\times K_n} &  0_{K_n\times K_n}& 0_{K_n \times 3} \\
\hline 
0 \cdots 0 &  0 \cdots 0 & 1\ 1\ 1 \\
\end{array} \right)
$$ 
The result then follows from an easy matrix computation.

For the second claim, The matrix $V(\gamma_{n,K_n,L_n+2})$ is primitive (by
Lemma~\ref{lm:reduce:gamma:nk}). By Proposition~\ref{prop:n:even:nKl} and
a matrix computation, we draw $\delta(V_{n,K_n,L_n+2}^4)=6$ for any $n\geq 6$. Hence
$\theta_{n,K_n,L_n+2} > 6^{\frac1{4}}$. Lemma~\ref{lm:comparing:matrix} is proved.
\end{proof}
\subsection{Comparing $\theta_{n,k,l}$ when $n\equiv 3 \mod 4$}

\begin{Lemma}
\label{cor:l:odd}
Let $n\geq 7$ such that $n\equiv 3 \mod 4$. If $l,l'\in \{1,\dots,L_n \}$ are odd and $l<l'$ then
$$
\theta_{n,K_n,l}>\theta_{n,K_n,l'}.
$$
\end{Lemma}
\begin{proof}[Proof of Lemma~\ref{cor:l:odd}]
We follow the notation of the proof of Proposition~\ref{prop:n:odd:nKl}. A simple computation shows:
$$
P_{n,K_n,l}-P_{n,K_n,l'} = \frac{2(X^l - X^{l'}) (X^m-1) (X^{l + l'} + X^{4 + m})}{X^{l + l'}(X-1) (X+1)}.
$$
In particular $P_{n,K_n,l}(x)-P_{n,K_n,l'}(x)<0$ for $x>\sqrt{2}$. Lemma~\ref{compare:roots} gives 
$\theta_{n,K_n,l}>\theta_{n,K_n,l'}$.
\end{proof}

\begin{Proposition}\label{compare:roots:n:3:mod4}
Let $n\geq 7$ be an integer satisfying $n\equiv3 \mod 4$.
\begin{enumerate}
\item If $n'=\frac{n+1}{2}$ then $\theta_{n',K_{n'},L_{n'}}>\theta_{n,K_n,L_n}$.
\item If $n'=\frac{n+1}{2}$ then $\theta_{n',K_{n'},L_{n'}+2}>\theta_{n,K_n,L_n}$.
\item Let $1\leq k\leq K_n-1$. If $d=\gcd(k,n-1)$, $n'=\frac{n+1}{d}$ and $k'=k/d$ then 
$\theta_{n,K_n,L}<\theta_{n',k'}$.
\end{enumerate}
\end{Proposition}

\begin{proof}[Proof of Proposition~\ref{compare:roots:n:3:mod4}]
\noindent {\bf Case (1).} We start with the first statement. 
Using  Proposition~\ref{prop:n:even:nKl} and Proposition~\ref{prop:n:odd:nKl}, we have:
$$
P_{n',K_{n'},L_{n'}}=x^{n'+1} - 2x^{n'-1} - x^{n'-3} - x^4 - 2x^2 + 1
$$
$$
P_{n,K_n,L_n} =x^{n+1} - 2 x^{n-1} -  4 x^{(n+3)/2} + 4  x^{(n-1)/2} + 2 x^2  - 1
$$
Noticing that $n-n'=\frac{n-1}{2}$, hence $n-n'+2=\frac{n+3}{2}$, we have:
$$
P_{n,K_n,L_n}-x^{n-n'}P_{n',K_{n'},L_{n'}} =x^{n-3}+x^{4+n-n'}-2x^{2+n-n'}+3x^{n-n'}+2x^2-1
$$
This polynomial clearly takes only positive values for $x>\sqrt{2}$, which proves the required 
inequality. \medskip

\noindent {\bf Case (2).} Now we come to the second statement. Assume $n \geq 11$ (for $n=7$ we directly 
prove the inequality). In this situation $n'\geq 4$ and Lemma~\ref{lm:comparing:matrix} gives
$$
\theta_{n',K_{n'},L_{n'}+2} > 6^{\frac1{4}}
$$
Let $\theta=\theta_{n,K_n,L_n}$ for simplicity. By Proposition~\ref{prop:n:odd:nKl}, we have:
$$
\theta^{n-1}(\theta^2 - 2) =  4 \theta^{(n+3)/2} - 4  \theta^{(n-1)/2} - 2 \theta^2  + 1
$$
Hence
$$
\theta - \sqrt{2} = \frac{1}{\theta+\sqrt{2}} \frac{4  \theta^{(n-1)/2}( \theta^{2} - 1) + 1 - 2 \theta^2}{\theta^{n-1}}
< \frac{4}{\theta^{(n-1)/2}} < \frac1{\sqrt{2}^{(n-5)/2}}
$$
Obviously $\frac1{\sqrt{2}^{(n-5)/2}} < 6^{\frac1{4}}-\sqrt{2}$ for $n\geq 23$. Hence
$\theta_{n,K_n,L_n} < 6^{\frac1{4}} < \theta_{n',K_{n'},L_{n'}+2}>$. For $n < 23$ we check directly 
the that this inequality holds. \medskip

\noindent {\bf Case (3).} 
Finally we prove the last statement. Assume first  that $\gcd(n-1,k)=1$, then $n'=n$ and $k'=k$.
Note that $\gcd(n-1,K_{n-1}-1)=1$. From Lemma~\ref{lm:compare:n} $\theta_{n,k} \geq \theta_{n,K_n-1}$ 
for any $k=1,\dots,K_n-1$. Since $L$ is odd we have, by Proposition~\ref{prop:n:odd:nKl} and Lemma~\ref{lm:a}:
\begin{eqnarray*}
P_{n,K_n-1}&=&x^{n+1}-2x^{n-1}-2x^{n-2\floor{\frac{n-5}{8}}-2}-2x^{\frac{n+1}{2}}-2x^{2\floor{\frac{n-5}{8}}+3}-2x^2+1\\
P_{n,K_n,L}&=&x^{n+1} - 2 x^{n-1} -  4 x^{(n+3)/2} + 4  x^{(n-1)/2} + 2 x^2  - 1
\end{eqnarray*}
Hence, for $x>\sqrt{2}$
\begin{eqnarray*}
P_{n,K_n,L}-P_{n,K_n-1}&=&2x^{n-2\floor{\frac{n-5}{8}}-2}-4 x^{(n+3)/2}+2x^{\frac{n+1}{2}}+4  x^{(n-1)/2}+2x^{2\floor{\frac{n-5}{8}}+3}+4x^2-2 \\
&>&2x^{n-2\floor{\frac{n-5}{8}}-2}-4 x^{(n+3)/2}
%
\end{eqnarray*}
For $n\geq 11$, $n-2\ceil{\frac{n-5}{8}}-2\geq \frac{n+3}{2}+2$, hence $P_{n,K_n,L}-P_{n,K_n-1}>0$. 
For $n=7$, we compute directly the roots: we have $\theta_{n,K_n-1}\approx 1,96$ and $\theta_{n,K_n,L}\approx 1,84$. \medskip

Now we assume that  $\gcd(k,n-1)=d>1$.
\begin{enumerate}
\item If $n'$ is odd (thus $n'\equiv3 \mod 4$) then by the above case $\theta_{n',k'} > \theta_{n',K_{n'},L_{n'}}$. We conclude with Lemma~\ref{lm:decreasing}.
\item If $n'$ is even then we need to show directly that $\theta_{n',k'} > \theta_{n,K_n,L}$. Note $n'$ even implies that $d$ is an odd multiple of 2. 
There are two cases: 
\begin{itemize}
\item $d\geq 6$. In this case, $\theta_{n',k'}\geq \theta_{n',K_{n'}}$ ($k'=K_{n'}$ is possible).
We have
\begin{eqnarray*}
P_{n,K_n,L}-x^{n-n'}P_{n',K_{n'}}&=&2x^{2+n-n'}-x^{n-n'}-4x^{\frac{n+3}{2}}+4x^{\frac{n-1}{2}}+2x^2-1 \\
&>& 2 x^{n-n'}- 4x^{\frac{n+3}{2}}
\end{eqnarray*}
We necessarily have $n> 10$ hence $n-n'\geq  \frac{n+3}{2}+2$ hence the above polynomial is positive for $x>\sqrt{2}$.
\item $d=2$. In this case $\theta_{n',K_{n'}}=\theta_{n,K_n}<\theta_{n,K_n,L}$, but $k'<K_{n'}$. Hence the previous strategy does not work. 
Since $n'$ and $K_{n'}-1$ are not necessarily relatively prime, 
 we compare directly $\theta_{n',k'}$ with $\theta_{n,K_n,L}$ by using Statement~4 of Lemma~\ref{lm:a}. We have
$$
P_{n,K_n,L}-x^{n-n'}P_{n',k'}\geq 2x^{n-n'+\ceil{\frac{2n'}{3}}}-4x^{\frac{n+3}{2}}
$$
For $n\geq 8$, we have $n-n'+\ceil{\frac{2n}{3}}\geq \frac{n+3}{2}+2$, implying the desired inequality. For $n=7$, we necessarily have $d=1$. 
\end{itemize}
\end{enumerate}
This ends the proof of the proposition.
\end{proof}

\subsection{Comparing $\theta_{n,k,l}$ when $n\equiv 0 \mod 2$}

\begin{Lemma}
\label{lm:n:even:nK:2}
Let $n\geq 4$ be an even integer. If $l,l'\in \{1,\dots,L_n \}$ satisfy $l<l'$ then
$$
\theta_{n,K_n,l}>\theta_{n,K_n,l'}.
$$
\end{Lemma}
\begin{proof}[Proof of Lemma~\ref{lm:n:even:nK:2}]
Proposition~\ref{prop:n:even:nKl} and a simple computation show
$$
P_{n,K_n,l}-P_{n,K_n,l'} = \frac{(X^l - X^{l'})(X^2 + 1)(X^{2l + 2l'} + X^{n + 3})(X^l + X^{l'})}{X^{2l + 2l' + 1}
(X+1) (X-1)}.
$$
In particular $P_{n,K_n,l}(x)-P_{n,K_n,l'}(x)<0$ for $x>\sqrt{2}$. Lemma~\ref{compare:roots} gives 
$\theta_{n,K_n,l}>\theta_{n,K_n,l'}$.
\end{proof}
\begin{Proposition}\label{compare:roots:second}
Let $n\geq 18$ be an even integer satisfying $n\not \equiv 4 \mod 6$. Then, $\gcd(n-1,K_n-1)=1$ and the followings hold:
\begin{enumerate}
\item For any $k=1,\dots,K_n-2$ one has $\theta_{n,k}>\theta_{n,K_n-1}$.
\item For any $l=1,\dots,L_n$ one has $\theta_{n,K_n,l} > \theta_{n,K_n-1}$.
\item For $l=L_n+1$, one has $\theta_{n,K_n,L_n+1} > \theta_{n,K_n-1}$.
\item For $l=L_n+2$, one has $\theta_{n,K_n,L_n+2} > \theta_{n,K_n-1}$.
\end{enumerate}
\end{Proposition}
\begin{proof}
Let us consider the first claim and set $d=\gcd(k,n-1)$. If $d=1$ then $\theta_{n,k}>\theta_{n,K_n-1}$ by Lemma~\ref{lm:compare:n}.
Otherwise let $n'=\frac{n-1}{d}+1<n$ and $k'=\frac{k}{d}$. Note that $\gcd(k',n'-1)=1$ and $\theta_{n,k}=\theta_{n',k'}$. By Lemma~\ref{lm:compare:n} $\theta_{n',k'} > \theta_{n',K_{n'}}$. It suffices to show $\theta_{n',K_{n'}} > \theta_{n,K_n-1}$.
We have, for $x>\sqrt{2}$,
\begin{eqnarray*}
P_{n,K_n-1}-x^{n-n'}P_{n',K_{n'}}&=&2x^{2+n-n'}-x^{n-n'}-2x^{\ceil{2n/3}}-2x^{\floor{n/3}+1}-2x^2+1\\
&>& 2x^{n-n'}-2x^{\ceil{2n/3}}+x^{n-n'}-2x^{\floor{n/3}+1}-2x^2+1
\end{eqnarray*}
Note that $n-1$ is odd and not a multiple of 3, hence $d\geq 5$. 
Since $n$ is large enough, we have
$$n-n'-\frac{2n}{3}=(n-1)(1-\frac{1}{d})-\frac{2n}{3}\geq \frac{4}{5}(n-1)-\frac{2n}{3}\geq 0$$
Hence, $\floor{n-n'-\frac{2n}{3}}=n-n'-\ceil{\frac{2n}{3}}\geq 0$. Similarly, 
$n-n'\geq 4+ \floor{n/3}+1$. 

Hence
$$P_{n,K_n-1}-x^{n-n'}P_{n',K_{n'}}> 2x^{\floor{n/3}+1}-2x^2+1>0. $$
The inequality $\theta_{n',K_{n'}} > \theta_{n,K_n-1}$ follows by Lemma~\ref{compare:roots}. \medskip

We now prove the second claim for $l=1,\dots,L_n$. Since $\theta_{n,K_n,l} > \theta_{n,K_n,L}$ by Lemma~\ref{lm:n:even:nK:2}, it suffices to show
$
\theta_{n,K_n,L} > \theta_{n,K_n-1}
$. We have:
$$P_{n,K_n-1}-P_{n,K_n,L}=x^{n-3}-2x^{\ceil{2n/3}}-2x^{\floor{n/3}+1}+x^4$$
If $n\geq 24$, we get $n-3\geq 4+ \ceil{2n/3}$, hence:
$$P_{n,K_n-1}-P_{n,K_n,L}>2x^{\ceil{2n/3}}-2x^{\floor{n/3}+1}+x^4>0$$
For $n\in \{18,20\}$, we check directly that $P_{n,K_n-1}-P_{n,K_n,L}>0$. The inequality $\theta_{n,K_n,L} > \theta_{n,K_n-1}$ follows by Lemma~\ref{compare:roots}. \medskip

Next we prove the third claim for $l=L_n+1$ (in this case, the computation is different).
By Lemma~\ref{lm:comparing:matrix} we have $\theta_{n,K_n,L_n+1}>3^{\frac1{2}}$.
For simplicity let $\theta=\theta_{n,K_n-1}$. By Lemma~\ref{lm:a}:
$$
\theta^{n-1}(\theta^{2} - 2) = 2\theta^{\ceil{2n/3}}+2\theta^{\floor{n/3}+1} + 2\theta^2 - 1 < 2\theta^{\ceil{2n/3}}+2\theta^{\floor{n/3}+1} + 2\theta^2
$$
Thus
$$
\theta-\sqrt{2} < \frac{2}{\theta+\sqrt{2}} \left( \frac1{\theta^{n-1-\ceil{2n/3}}} + \frac1{\theta^{n-2-\floor{n/3}}} + \frac1{\theta^{n-3}}  \right) < \frac{3}{\sqrt{2}^{n/3-1}}
$$
since $\theta>\sqrt{2}$ and $n>3$. Clearly $\frac{3}{\sqrt{2}^{n/3-1}} < 3^{\frac1{2}}-\sqrt{2}$ for $n>22$ hence $\theta<3^{\frac1{2}}<\theta_{n,K_n,L_n+1}$ that is the desired inequality.
For $n=20,22$ we directly check the inequality.

Finally we prove the last claim for $l=L_n+2$. By Lemma~\ref{lm:comparing:matrix} we have $\theta_{n,K_n,L_n+2}>6^{\frac1{4}}$. We can check that $\frac{3}{\sqrt{2}^{n/3-1}} < 6^{\frac1{4}}-\sqrt{2}$ for $n>28$ hence $\theta<6^{\frac1{4}}<\theta_{n,K_n,L_n+2}$ that is the desired inequality (for $n<30$ we directly check the inequality). The proposition is proved.
\end{proof}

\subsection{Case $n\equiv 1 \mod 4$}

\begin{Proposition}
\label{compare:roots:n:1:mod4}
Let $n\geq 5$ such that $n\equiv 1 \mod 4$. For any $1\leq k\leq K_n-1$ we define $d=\gcd(k,n-1)$
and $n'=\frac{n-1}{d}+1$, $k'=k/d$. Then  $\theta_{n',k'}>\theta_{n,K_n}$.
\end{Proposition}

\begin{proof}[Proof of Proposition~\ref{compare:roots:n:1:mod4}]
Let $k\in \{1,\dots,K_n-1\}$. If $\gcd(k,n-1)=1$ then 
Lemma~\ref{lm:compare:n} implies that $\theta_{n,k} > \theta_{n,K_n}$ as desired.\\
If $\gcd(k,n-1)=d>1$ there are three cases depending the value of $n' \mod 4$.
\begin{enumerate}
\item If $n'\equiv1 \mod 4$ then the previous argument shows that $\theta_{n',k'} > \theta_{n',K_{n'}}$.
By Lemma~\ref{lm:decreasing}, the sequence $(\theta_{n,K_n})_n$ is decreasing for $n\equiv 1\mod 4$, so we have $\theta_{n',K_{n'}}> \theta_{n,K_n}$ as desired.

\item If $n'$ is even then Lemma~\ref{lm:compare:n} implies $\theta_{n',k'} > \theta_{n',K_{n'}}$. By Lemma~\ref{lm:a}, for any $x>\sqrt{2}$:
\begin{eqnarray*}
P_{n,K_n}-x^{n-n'}P_{n',K_{n'}}&=&2x^{2+n-n'}-x^{n-n'}-2x^{\frac{n+1}{2}}-2x^2+1\\
&>& 3x^{n-n'}-2x^{\frac{n+1}{2}}-2x^2 >0
\end{eqnarray*}
(the last inequality comes from $d\geq 4$ and $n\geq 5$). By Lemma~\ref{compare:roots} $\theta_{n',K_{n'}} > \theta_{n,K_n}$.

\item If $n'\equiv3 \mod 4$ then Proposition~\ref{compare:roots:n:3:mod4} implies
 $\theta_{n',k'} > \theta_{n',K_{n'},L_{n'}}$. For $x>\sqrt{2}$:
 \begin{eqnarray*}
 P_{n,K_n}-x^{n-n'}P_{n',K_{n'},L_{n'}}&=&4x^{n-n'+\frac{n'+3}{2}}-4x^{n-n'+\frac{n'-1}{2}} -2x^{2+n-n'}+x^{n-n'}-2x^{\frac{n+1}{2}}-2x^2+1\\
 &>& 4x^{n-n'+\frac{n'-1}{2}}-2x^{2+n-n'}-2x^{\frac{n+1}{2}}+x^{n-n'}-2x^2
\end{eqnarray*}
Assumption on $n,n'$ implies that $n\neq 5,9$, hence $n\geq 13$ and $n'\geq 7$. This implies that 
 $n-n'+\frac{n'-1}{2}=(n-1)(1-\frac{1}{2d})\geq \frac{n+1}{2}$, $n-n'+\frac{n'-1}{2}\geq 2+n-n'$ and $n-n'\geq 4$.
Thus $P_{n,K_n}-x^{n-n'}P_{n',K_{n'},L_{n'}}>0$ and Lemma~\ref{compare:roots} implies $\theta_{n,K_n}<\theta_{n',K_{n'},L_{n'}}$. 
 \end{enumerate}
Proof of Proposition~\ref{compare:roots:n:1:mod4} is complete.
\end{proof}

\section{A naive attempt to generalize the Rauzy--Veech construction}
\label{first:attempt}
The classical construction of pseudo-Anosov homeomorphism by Rauzy induction necessarily produces maps that preserves a singularity, and a horizontal  separatrix. This clearly comes from the fact that only right Rauzy induction is used. So, it is natural to expect to produce pseudo-Anosov homeomorphism that do not fix a separatrix by combining right and left induction. 

For instance, we consider a path $\gamma$ in the labeled (extended) Rauzy diagram such that.
\begin{itemize}
\item the image of $\gamma$ in the reduced extended Rauzy diagram is closed.
\item $\gamma$ is the concatenation of a path $\gamma_1$ that consists only of right Rauzy moves, and a path $\gamma_2$ that consists only of left Rauzy moves.
\end{itemize}

As above,  we associate to such path a matrix $V$ by multiplying the corresponding product of the transition matrices by a suitable permutation matrix. Assume now that the matrix $V$ is a primitive. Let   $\theta > 1$ be its Perron-Frobenius eigenvalue.  We choose a positive eigenvector $\lambda$   for   $\theta$.  As before, $V$ is symplectic, thus let  us choose an eigenvector $\tau$
for the  eigenvalue $\theta^{-1}$. It turns out that $\tau$ is not necessarily a suspension datum, but it is a weak suspension datum (up to replacing it by its opposite). Indeed, the set of weak suspension data is a open cone $W$, which is by construction invariant by $V^{-1}$, and we conclude as previously (the proof is the same as in Proposition~\ref{prop:construction:path}).


\begin{Proposition}
\label{prop:rauzy:leftright:fixed}
A pseudo-Anosov homeomorphism affine on a translation surface, and constructed as above fixes a vertical separatrix. In particular it is obtained by the usual Rauzy--Veech construction.
\end{Proposition}

\begin{proof}
Let $(\pi,\lambda,\tau)$ be the weak suspension datum defined as above and $h$ be a height. We will denote the associated surface by $X=X(\pi,\lambda,\tau)$ and by $I=I_h$ the corresponding horizontal interval.

After the prescribed the prescribed sequence of right and left Rauzy induction, we obtain the suspension datum $(\pi,\lambda',\tau')=(\pi,\frac{1}{\theta}\lambda,\theta \tau)$ defining the same surface $X$, with corresponding interval $I'_{h'}=I'\subset I$ (recall the Rauzy--Veech induction corresponds to cutting the interval on the right, or on the left). Also, $\theta h$ is an obvious height for $(\pi,\frac{1}{\theta}\lambda,\theta \tau)$, and the corresponding interval $I'_{\theta h}$ is the image by $\pA$ of the interval $I_h$. 

Hence there is an isometry $f$ from $I'_{\theta h}$ to  $I'_{h'}$ obtained by following a vertical leaf (see Section~\ref{sec:suspension}). The map $f\circ \pA$ is therefore a contracting map from $I_{h}$ to itself (its derivative is $\theta^{-1}$), hence has a fixed point. It means that there is an element $x$ in $I_{h}$ whose image by $\pA$ is in the vertical leaf $l$ passing through $x$. Thus, this vertical leaf $l$ is preserved by $\phi$. Since $\phi$, restricted to $l$ as derivative $\theta\neq 1$, there is a fixed point of $\phi$ on $l$. This fixed point is either a conical singularity or a regular point. In any case, $\phi$ fixes a vertical separatrix. Hence $\phi$ fixes also a horizontal separatrix. It is therefore obtained by the usual Rauzy--Veech construction.
\end{proof}

%
%
%



\begin{thebibliography}{EMM2}


\providecommand{\bysame}{\leavevmode  ---\  } \providecommand{\og}{``}
\providecommand{\fg}{''}


\bibitem[BGMY1980]{BGMY}  {\scshape Block, J Guckenheimer, M Misiurewicz {\normalfont  and} L.-S. Young}
  -- {\og Periodic points and topological entropy of one-dimensional maps~\fg} 
  Global theory of dynamical systems ({P}roc. {I}nternat.
  {C}onf., {N}orthwestern {U}niv., {E}vanston, {I}ll., 1979)
     \emph{Lecture Notes in Math.} volume 819, pages 18--34. Springer, Berlin, 1980.

\bibitem[BL12]{BL12}
{\scshape C.~Boissy , {\normalfont and} E.~Lanneau} --
{\og Pseudo-Anosov homeomorphisms on translation surfaces in hyperelliptic 
components have large entropy \fg}, 	\emph{Geom.  Funct. Anal.} {\bf 22} (2012), no. 1, 74--106.


\bibitem[BL09]{BL}  {\scshape C.~Boissy {\normalfont  and} E.~Lanneau}
  -- {\og Dynamics  and geometry of the  Rauzy-Veech induction
    for    quadratic   differentials~\fg},    \emph{Ergodic   Theory
    Dynam. Systems }  {\bf 29}  (2009),  no. 3, pp.~767--816.

%
%
\bibitem[Buf06]{Bufetov2006} {\scshape A.~Bufetov } --
 {\og Decay of correlations for the Rauzy-Veech-Zorich induction map on the space of interval exchange transformations and the central limit theorem for the Teichm\"uller flow on the moduli space of abelian differentials \fg}, 
\emph{J. Amer. Math. Soc.} {\bf 19}  (2006),  no. 3, pp.~579--623.
%
%
%
%
\bibitem[Far06]{Farb}  {\scshape B.~Farb  } --  {\og Some  problems on
mapping class  groups and moduli space\fg},  \emph{Problems on Mapping
Class   Groups  and   Related   Topic}  ed.   by   B.  {\cita   Farb},
\emph{Proc.  Symp.   Pure  and  Applied   Math.}  \textbf{74},  (2006)
pp.~11--55.


\bibitem[FM11]{FM}  {\scshape B.~Farb and D.~Margalit } --  {\og A primer on mapping class groups\fg},
Princeton University Press (2011).


%
%
%
%
%
%
%
%
%
\bibitem[Ker85]{Kerckhoff1985}
{\scshape S. Kerckhoff } -- 
{\og  Simplicial systems for interval exchange maps and measured foliations \fg}, \emph{Ergodic Theory Dynam. Systems}  
{\bf 5} (1985),  no. 2, pp.~257--271.
%
%
\bibitem[KZ03]{KZ} {\scshape M.~Kontsevich {\normalfont and} A.~Zorich
}  --  {\og Connected  components  of  the  moduli spaces  of  Abelian
differentials with prescribed singularities\fg}, \emph{Invent.\ Math.}
{\bf 153} (2003), no.~3, pp.~631--678.
%
%
%
%
%
%
%
%
%
\bibitem[LT10]{Lanneau:Thiffeault}  {\scshape  E.~Lanneau {\normalfont
and}  J-L.  Thiffeault  }  --   {\og  On  the  minimum  dilatation  of
pseudo-Anosov   homeomorphisms  on   surfaces   of  small   genus\fg},
\emph{Annales Fourier} {\bf 60} (2010).
%
%
%
\bibitem[MMY05]{Marmi:Moussa:Yoccoz}   {\scshape  S.~Marmi,  P.~Moussa
{\normalfont and}  J.-C.~Yoccoz }  -- {\og The  cohomological equation
for    Roth   type   interval  exchange   transformations\fg},
\emph{Journal of the Amer. Math. Soc.}  {\bf 18} (2005), pp.~823--872.
%
%
%
%
%
%
%
%
%
%
%
\bibitem[Rau79]{Rauzy}  {\scshape   G.~Rauzy  }  --   {\og  \'Echanges
d'intervalles et transformations induites\fg}, \emph{Acta Arith.} {\bf
34} (1979), pp.~315--328.
%
%
%
%
\bibitem[Vee82]{Veech1982} {\scshape W.~Veech } -- {\og Gauss measures
for  transformations  on  the  space of  interval  exchange  maps\fg},
\emph{Ann. of Math. (2)} {\bf 115} (1982), no. 1, pp.~201--242.
%

\end{thebibliography}
\end{document}